\documentclass[a4paper]{amsart}
\usepackage{amssymb, amsmath, amsthm}
\usepackage{stmaryrd, esint}
\usepackage{graphicx}
\usepackage{wrapfig}
\DeclareGraphicsExtensions{.png}

\newtheorem{theorem}{Theorem}[section]
\newtheorem{lemma}[theorem]{Lemma}
\newtheorem{proposition}[theorem]{Proposition}
\newtheorem{corollary}[theorem]{Corollary}

\theoremstyle{definition}
\newtheorem{definition}{Definition}[section]
\newtheorem{example}[definition]{Example}

\theoremstyle{remark}
\newtheorem{remark}[definition]{Remark}

\numberwithin{equation}{section}

\newcommand{\abs}[1]{\lvert#1\rvert}

\newcommand{\norm}[1]{\left\lVert#1\right\rVert}
\newcommand{\tr}[1]{\bigr\vert_{\scriptscriptstyle{#1}}}
\newcommand{\A}{\mathcal{A}}

\newcommand{\F}{\mathcal{F}}
\newcommand{\G}{\mathcal{G}}
\newcommand{\h}{\mathcal{H}}

\newcommand{\p}{\mathcal{P}}
\newcommand{\R}{\mathbb{R}}
\newcommand{\Sp}{\mathcal{S}}

\newcommand{\N}{\mathbb{N}}
\newcommand{\Z}{\mathbb{Z}}

\DeclareMathOperator{\grad}{grad}
\DeclareMathOperator{\Div}{div}
\DeclareMathOperator{\dist}{dist}
\title[boundary regularity]{Boundary regularity of Dirichlet minimizing $Q$-valued functions}

\author[J.Hirsch]{Jonas Hirsch}

\begin{document}

\begin{abstract}
We consider the H\"older continuity for the Dirichlet problem at the boundary. Almgren introduced the multivalued/ $Q$-valued functions for studying regularity of minimal surfaces in higher codimension. The H\"older continuity in the interior for Dirichlet minimizers is an outcome of Almgren's original theory \cite{Almgren}, to which C. De Lellis and E.N. Spadaro's work have given a simpler alternative approach \cite{Lellis}. We extend the H\"older regularity for Dirichlet minimizing $Q$-valued functions up to the boundary assuming $C^1$ regularity of the domain and $C^{0, \alpha}$ regularity of the boundary data with $\alpha > \frac{1}{2}$.
\end{abstract}
\maketitle

\section*{Introduction} 
\label{sec_B:introduction}
Multivalued maps with focus on Dirichlet integral minimizing maps have been introduced by F.~Almgren in his pioneering work \cite{Almgren}. He introduced them as $Q$-valued functions. $Q \in \N$, fixed, indicates the number of values the function takes, counting multiplicity. We will refer to them from now on as $Q$-valued functions. Their purpose had been the development of a proof of a regularity result on area minimizing rectifiable currents. The author recommends \cite{Lellis note} for a motivation of their definition, an overview of Almgrens program. Furthermore it compares different modern approachs to $Q$-valued functions inspired for instance by a metric analysis and surveys some recent contributions. A complete modern revision of Almgrens original theory and results can be found in \cite{Lellis}. We follow their notation, compare section \ref{sec_I:Q-valued functions}.\\ One introduces a Dirichlet energy for $Q$-valued maps. A function is Dirichlet minimizing if it is minimzing with respect to compact variations. \cite{Spadaro} gives a modern proof to a large class of examples arising from complex varieties. The H\"older continuity in the interior was already settled by Almgrens original theory and nicely presented in \cite{Lellis}.
Many results of Almgren have been extended in several directions; \cite{Lellis select}, \cite{Goblet09}, \cite{Zhu06Fre}, \cite{dePauw} consider $Q$-valued functions mapping into non-euclidean ambient spaces, \cite{Zhu06An}, \cite{Zhu06flow}, \cite{Zhu06reg}, \cite{GobletZhu08}, \cite{Goblet06}, focus on other objects in the $Q$-valued setting like differential inclusions, geometric flows and quasi minima. \cite{Mattila83}, \cite{Lellis09} extend some theorems to more general energy functionals.

 Nonetheless many regularity questions concerning these functions remain open. Some of them has been already proposed by Almgren himself and can be found in \cite{Brothers} and \cite{Lellis note}.\\

We address the following regularity question concerning Almgrens multivalued functions, posed for example by C.De Lellis in \cite[section 8, (7)]{Lellis note}:\\
Are Dirichlet minimizers continuous, or ever H\"older, up to the boundary if the boundary data are sufficient regular?\\

The following result gives a rather general first answer:

\begin{theorem}\label{theo_B:0.1}
	Let $\frac{1}{2}<s\le 1$ be given. There is a constant $\alpha=\alpha(N,Q,n,s)>0$ with the property that, if
	\begin{itemize}
		\item[(a1)] $\Omega \subset \R^N$ is a bounded $C^1$ regular domain;
		\item[(a2)] $u \in W^{1,2}(\Omega, \A_Q(\R^n))$ is Dirichlet minimizing;
		\item[(a3)] $u \tr{\partial \Omega} \in C^{0,s}(\partial \Omega)$; 
	\end{itemize}
	then $u \in C^{0,\alpha}(\overline{\Omega})$.
\end{theorem}

To my knowledge, the only boundary regularity theorem proved in this context prior to theorem \ref{theo_B:0.1} is contained in \cite{Zhu} where, assuming the domain of the Dirichlet minimizer is a $2$-dimensional disk, the author proved that continuity holds up to the boundary if the boundary data are continuous. We will give a proof on different lines that continuity extends up the boundary of Lipschitz regular domains, cp. with section \ref{sub:continuity_up_to_boundary}.\\

In terms of notation, for single valued functions, Sobolev spaces are denoted by $W^{1,p}(\Omega, \R^n)$ and $ W^{1,p}(\Omega)$, fractional Soblev spaces by $W^{s,p}(\Omega)$. In the case of multivalued functions we will always mention the target explicitly i.e. $W^{1,p}(\Omega, \A_Q(\R^n))$ for Sobolev spaces and the fractional ones by $W^{s,p}(\Omega, \A_Q(\R^n))$. In the case of single valued function we will sometimes use as well $H^1(\Omega)$, $H^s(\Omega)$ for $W^{1,2}(\Omega)$ and $W^{s,2}(\Omega)$ ($p=2$). The trace for a Sobolev function is denoted by $u\tr{\partial \Omega}$. It will be clear from the context if it is the trace of a single valued or multivalued function. 

The equivalent "classical" statement of Theorem \ref{theo_B:0.1} for single valued harmonic functions states:\\
$f: \Omega \to \R^n$ harmonic, $f\tr{\partial \Omega} \in C^{0,\beta}(\partial \Omega)$ for some $0<\beta<1$ then $f \in C^{0,\beta}(\overline{\Omega})$.\\
	
Harmonic functions with finite energy belong to $H^1(\Omega, \R^n)$, but $u \in H^1(\Omega)$ if and only if $u\tr{\partial \Omega} \in H^{\frac{1}{2}}(\partial \Omega)$. $H^{\frac{1}{2}}(\partial \Omega)$ can be characterised using the Gagliardo semi-norm $\int_{\partial \Omega\times\partial \Omega} \frac{\abs{f(x)-f(y)}^2}{\abs{x-y}^{N}} \,dxdy$ that is controlled by the $C^{0,\beta}(\partial \Omega)$-norm for $\beta>\frac{1}{2}$. Nonetheless our result is suboptimal in the sense that for classical harmonic functions $u\tr{\partial \Omega} \in W^{\frac{1}{2},2}(\partial \Omega)\cap C^{0,\beta}(\partial \Omega)$ for any $0<\beta<1$ implies $u\in C^{0,\beta}(\overline{\Omega})$. In contrast, the H\"older exponent we claim in Theorem \ref{theo_B:0.1} is not explicit. For dimension three and higher that is not really surprising since the optimal (or even an explicit) exponent is not known in the interior so far.\\

The result for two dimensions is somewhat unsatisfactory. In two dimensions the optimal H\"older exponent for the interior regularity for $Q$-valued Dirichlet minimizers is known and explicit: it is $\frac{1}{Q}$.  We obtain the two dimensional case of theorem \ref{theo_B:0.1} by "lifting it" to three dimensions. So we get a "bad", not explicit exponent. Therefore we try to give some additional information. So we prove, as mentioned, that continuity extends up the boundary data on a $2$-dimensional Lipschitz regular domain if the boundary data is continuous. 
Concerning the optimal exponent we can give a partial first answer. At least on conical subsets of $\Omega$ the interior regularity extends up to the boundary for boundary data $u\tr{\partial \Omega} \in C^{0,\beta}(\partial \Omega)$, $\beta>\frac{1}{2}$. \\

The appendix contains a short introduction to fractional Sobolev spaces for single valued functions. It includes some perhaps less known results. Furthermore an interpolation lemma in the spirit of Luckhaus with boundaries functions in a fractional Sobolev space $W^{s,2}$ with $s>\frac{1}{2}$ is presented. Afterwards these results are extended to $Q$-valued functions. Additionally we present a concentration compactness result for $Q$-valued functions. It is along the same lines and indeed inspired  by C.~De Lellis and E.~Spadaro's version \cite[Lemma 3.2]{Lellis Lp}. Furthermore it contains a $W^{s,p}$ selection criterion, needed in the two dimensional setting.\\

Outline of this article: 
section \ref{sec_I:Q-valued functions} recalls the basic definition and results on $Q$-valued functions that are of interest in our context,  
section \ref{sec:general_assumptions_and_notation} fixes notation and general assumptions, section \ref{sec:h"older_continuity_for_nge_3_} contains the proof of theorem \ref{theo_B:0.1} for dimension three and higher, section \ref{sec:boundary_regularity_in_dimension_n_2_} considers the two dimensional setting. Finally the appendix with sections \ref{sec:fractional_sobolev_spaces}, \ref{sec:_q_valued_functions} and \ref{sec:construction_of_bilipschitz_maps_between_b_1_and_omega_fcap_b_1_} provides tools needed in the proof.
\section*{Acknowledgements}
My most humble and sincere thanks to my supervisors Camillo De Lellis and Emanuele Spadaro for introducing me to F. Almgren’s $Q$-valued functions. Reading their modern review of the theory gave me the idea to start this project. Their insights and stimulating discussions really helped my work. Their knowledge and expertise, on more topics than I can ever hope to know, was invaluable.

\tableofcontents

\section{Q-valued functions} 
\label{sec_I:Q-valued functions}
As announced this section recalls the basic definitions and results on $Q$-valued functions needed in here. The theory is presented omitting the actual proofs. They can be found for instance in C.~De Lellis and E.~Spadaro's work \cite{Lellis}. More refined results are presented in the appendix. In there a concentration compactness result is presented. It is along the same lines and indeed inspired  by C.~De Lellis and E.~Spadaro's version \cite[Lemma 3.2]{Lellis Lp}. Furthermore an interpolation lemma in the spirit of Luckhaus with boundary functions in a fractional Sobolev space and a $W^{s,p}, s> \frac{1}{2}$ selection criterion.\\

We follow mainly the notation and terminology introduced by C.~De Lellis and E.~Spadaro in \cite{Lellis}. It differs slightly from Almgren's original one. $Q, Q_1, Q_2, \dotsc$ are always natural numbers.\\
The space of unordered sets of $Q$ points in $\R^n$ can be made into a complete metric space.
\begin{definition}\label{def_I:1.101}
	$\left(\A_Q(\R^n), \G \right)$ denotes the metric space of unordered $Q$-tuples given by
	\begin{equation*}
	\mathcal{A}_Q(\R^n)= \left\{ T= \sum_{i=1}^Q \llbracket t_i \rrbracket \colon t_i \in \R^n, i = 1, \dotsc, Q \right\}
	\end{equation*}
	and if $\mathcal{P}_Q$ is the permutation group of $\{1, \dotsc, Q\}$ the metric is given by
	\begin{equation*}
		\G(S,T)^2= \min_{\sigma \in \mathcal{P}_Q} \sum_{i=1}^Q \abs{s_i - t_{\sigma(i)}}^2.
	\end{equation*}
\end{definition}
We use the convention $\llbracket t \rrbracket = \delta_t$ for a Dirac measure at a point $t \in \R^n$. Considering $T=\sum_{i=1}^Q \llbracket t_i \rrbracket$ as a sum of $Q$ Dirac measures one notice that $\A_Q(\R^n)$ corresponds to the set of  $0$-dimensional integral currents of mass $Q$ and positive orientation. Hence we will write 
\[
	spt(T)=\{ t_1, \dots, t_Q \colon T=\sum_{i=1}^Q \llbracket t_i \rrbracket \} \subset \R^n.
\]
 
Furthermore $\A_Q(\R^n)$ is endowed with an intrinsic addition:
\[
	+\colon \A_{Q_1}(\R^n) \times \A_{Q_2}(\R^n) \to \A_{Q_1+Q_2}(\R^n) \quad S+T = \sum_{i=1}^{Q_1} \llbracket s_i \rrbracket + \sum_{i=1}^{Q_2} \llbracket t_i \rrbracket.
\]
We define a translation operator
\[
	\oplus\colon A_Q(\R^n) \times \R^n \to \A_Q(\R^n) \quad T\oplus s = \sum_{i=1}^Q \llbracket t_i + s \rrbracket.
\]
The metric $\G$ defines continuity, modulus of continuity, H\"older and Lipschitz continuity and (Lebesgue) measurability for functions from a set $\Omega \subset \R^N$ into $\A_Q(\R^n)$, i.e.$u: \Omega \to \A_Q(\R^n)$. \\
As it has been shown in \cite[Proposition 0.4]{Lellis} for any measurable function $u: \Omega\to \A_Q(\R^n)$ we can find a measurable selection i.e.
\[
	v=(v_1, \dotsc , v_Q): \Omega \to (\R^n)^Q \text{ measurable s.t. } u(x)=[v](x)=\sum_{i=1}^Q \llbracket v_i(x) \rrbracket.
\]
Selections of higher regularity are considered in \cite{Lellis select}, \cite[Proposition 1.2]{Lellis} and in the appendix \ref{sub:selection}.\\
We will write $\abs{u(x)}=\sqrt{\sum_{i=1}^Q \abs{v_i(x)}^2}=\G(u(x),Q\llbracket 0 \rrbracket)$.
\begin{definition}\label{def_I:1.102}
	The Sobolev space $W^{1,2}(\Omega, \A_Q(\R^n))$ is defined as the set of measurable functions $u: \Omega \to \A_Q(\R^n)$ that satisfy 
	\begin{itemize}
		\item[(w1)] $x \mapsto \G(u(x),T) \in W^{1,2}(\Omega, \R_+)$ for every $T \in \A_Q(\R^n)$;
		\item[(w2)] $\exists \varphi_j \in L^2(\Omega, \R_+)$ for $j=1, \dots, N$ s.t. $\abs{D_j\G(u(x),T)} \le \varphi_j(x)$ for any $T \in \A_Q(\R^n)$ and a.e. $x \in \Omega$.
	\end{itemize}
\end{definition}
It is not difficult to show the existence of minimal functions $\tilde{\varphi}_j$, in the sense that $\tilde{\varphi}_j(x)\le \varphi_j(x)$ for a.e. $x$ and any $\varphi_j$ satisfying property (w2), \cite[Proposition 4.2]{Lellis}. Such a minimal bound is denoted by $\abs{D_ju}$ and is explicitly characterised by
\[
	\abs{D_ju}(x)= \sup\left\{ \abs{D_j \G(u(x), T_i)}\colon \{T_i\}_{i\in \N} \text{ dense in } \A_Q(\R^n)\right\}.
\]
The Sobolev "semi-norm", or Dirichlet energy, is defined by integrating the measurable function $\abs{Du}^2=\sum_{j=1}^N \abs{D_ju}^2$:
\begin{equation}\label{eq_I:1.101}
	\int_\Omega \abs{Du}^2 = \int_{\Omega} \sum_{j=1}^J \abs{D_ju}^2.
\end{equation} 
Strictly speaking it is not a "semi-norm". $W^{1,2}(\Omega, \A_Q(\R^n))$ is not a linear space since $\A_Q(\R^n)$ lacks this property.\\
A function $u \in W^{1,2}(\Omega, \R^n)$ is said to be Dirichlet minimizing if
\begin{equation}\label{eq_I:1.102}
	\int_{\Omega} \abs{Du}^2= \inf \left\{ \int_{\Omega} \abs{Dv}^2 \colon v \in W^{1,2}(\Omega, \A_Q(\R^n)), \G(u(x), v(x)) \in W^{1,2}_0(\Omega, \R_+) \right\}.
\end{equation}

On Lipschitz regular domains $\Omega \subset \R^N$ one has a continuous trace operator as for classical single valued Sobolev functions
\[
	\circ\tr{\partial \Omega}: W^{1,2}(\Omega, \A_Q(\R^n)) \to L^2(\partial \Omega, \A_Q(\R^n)).
\]
The definition of $W^{1,2}(\Omega, \A_Q(\R^n))$, definition \ref{def_I:1.102}, implies that on a Lipschitz regular domain $\Omega \subset \R^N$ one has that $\G(u(x),v(x)) \in W^{1,2}_0(\Omega)$ corresponds to $u\tr{\partial \Omega}=v\tr{\partial \Omega}$ for any $u,v \in W^{1,2}(\Omega, \A_Q(\R^n))$.\\

As a consequence of a Rademacher theorem for multivalued Lipschitz functions, \cite[section 1.3 \& Theorem 1.13]{Lellis} a Sobolev function $u \in W^{1,2}(\Omega, \A_Q(\R^n))$ is a.e. approximately differentiable in the sense
\begin{itemize}
	\item[(1)] $\exists \mathcal{U}_x: \Omega \to \A_Q(\R^n \times Hom(\R^N, \R^n))$, $x \mapsto \mathcal{U}_x= \sum_{i=1}^Q \llbracket (u_i(x), U_i(x)) \rrbracket$ measurable with $U_i(x)=U_j(x)$ whenever $u_i(x)=u_j(x)$;
	\item[(2)] $\mathcal{U}_x$ defines a 1-jet $J\mathcal{U}_x: \Omega \times \R^N \to \A_Q(\R^n)$ by $J\mathcal{U}_x(y)=\sum_{i=1}^Q \llbracket u_i(x) + U_i(x)(y-x) \rrbracket$, that has the additional property that $J\mathcal{U}_x(x)= u(x)$ for a.e. $x \in \Omega$;\\
	\item[(3)] for a.e. $x\in \Omega$,  $\exists E_x \subset \Omega$ having density $1$ in $x$ s.t.  $\G(u(y), J\mathcal{U}_x(y))=o(\abs{y-x})$ on $E_x$. 
\end{itemize}
As one may guess the 1-jet corresponds to a first order "Taylor expansion", that becomes apparent in the proof of Rademacher's theorem, \cite[Theorem 1.13]{Lellis}. 
One can show that $\abs{D_ju}(x)= \sum_{i=1}^Q \abs{U_i(x)e_j}^2$ for a.e. $x \in \Omega$, \cite[Proposition 2.17]{Lellis}. From now on we will write $Du_i(x)$ for $U_i(x)$ and $D_ju_i(x)$ for $U_i(x)e_j$.\\

A useful tool is Almgren's bi-Lipschitz embedding of $\A_Q(\R^n)$ into some $\R^N$. A remark of Brian White improved it, compare \cite[Theorem 2.1 \& Corollary 2.2]{Lellis}:

\begin{theorem}[bi-Lipschitz embedding]\label{theo_I:1.101}
	There exists $m=m(Q,n)$ and an injective map $\boldsymbol{\xi}: \A_Q(\R^n) \to \R^m$ with the properties
	\begin{itemize}
		\item[(i)] $Lip(\boldsymbol{\xi})\le 1$ and $Lip(\boldsymbol{\xi}^{-1}\vert_{\boldsymbol{\xi}(\A_Q(\R^n))})\le C(Q,n)$;
		\item[(ii)]  $\forall T \in \A_Q(\R^n)$ $\exists\delta=\delta(T) >0$ such that $\abs{\boldsymbol{\xi}(T)-\boldsymbol{\xi}(S)}=\G(T,S)$ for all $S \in B_{\delta}(T) \subset \A_Q(\R^n)$.
	\end{itemize}
There is a retraction $\boldsymbol{\rho}: \R^m \to \A_Q(\R^n)$ because of (i) and the Lipschitz extension Theorem, e.g. \cite[Theorem 1.7]{Lellis}.
\end{theorem}
As a consequence $\abs{Du}(x)=\abs{D\boldsymbol{\xi}\circ u}(x)$ for a.e. $x \in \Omega$ for any $u \in W^{1,2}(\Omega, \mathcal{A}_Q(\R^n))$.\\
We want to remark that the image of $\A_Q(\R^n)$ under $\boldsymbol{\xi}$ in $\R^m$ is not convex neither a $C^2$ manifold. Thus there is no "nearest point" projection not even in a tubular neighborhood.

Two cornerstones in the context of Dirichlet minimizers that are of interest for us in the following are (c.p. with \cite[Theorem 0.8 \& Theorem 0.9]{Lellis}):
.
\begin{theorem}[Existence of Dirichlet minimizers]\label{theo_I:1.102}
	Let $v\in W^{1,2}(\Omega, \A_Q(\R^n))$ be given, then there exists a (not necessarily unique) Dirichlet minimizing $u\in W^{1,2}(\Omega, \A_Q(\R^n))$ with $\G(u(x),v(x)) \in W^{1,2}_0(\Omega,\R_+)$.  
\end{theorem}

\begin{theorem}[interior H\"older continuity]\label{theo_I:1.103}
There is a constant $\alpha_0=\alpha_0(N,Q)>0$ with the property that if $u \in W^{1,2}(\Omega, \A_Q(\R^n))$ is Dirichlet minimizing, then $u\in C^{0,\alpha_0}(K, \A_Q(\R^n))$ for any $K\subset \Omega\subset \R^N$ compact. Indeed, $\abs{Du}$ is an element of the Morrey space $L^{2,N-2-2\alpha_0}$ with the estimate
\begin{equation}\label{eq_I:1.103}
	r^{2-N-2\alpha_0} \int_{B_r(x)} \abs{Du}^2 \le R^{2-N-2\alpha_0} \int_{B_R(x)} \abs{Du}^2 \text{ for } r\le R, B_R(x) \subset \Omega.
\end{equation}
For two-dimensional domains $\alpha_0(2,Q)=\frac{1}{Q}$ is explicit and optimal.
\end{theorem}
Both results had been proven first by Almgren in \cite{Almgren} and nicely reviewed by C. De Lellis and E. Spadaro in \cite{Lellis}.\\

J.~Almgren presents in  \cite[Theorem 2.16]{Almgren} an example of non-uniqueness: there are two Dirichlet minimizers $f \neq h\in W^{1,2}(B_1, \A_2(\R^2))$, $B_1 \subset \R^2$, with $f = h$ on $\partial B_1$. Given any other minimzer that agrees with $f$ or $h$ at the boundary must be either $f$ or $h$.

\section{General assumptions and further notation} 
\label{sec:general_assumptions_and_notation}
From now on, if not indicated differently, we will consider the following setting:
$\Omega \subset \R^N$ is a bounded $C^1-$regular domain i.e. to every $z\in \partial \Omega$ there exists $R=R(z)>0$, $F=F_z\in C^{1}(\R^{N-1}, \R)$ s.t. ( up to a rotation )
\[
	\Omega \cap B_{R}(z)=\{z+(x',x_N) \colon \abs{x}<R, \, x_N > F(x') \}.
\]
In particular for $F \in C^{1}(\R^{N-1},R)$ we set
\[
	\Omega_F= \{(x',x_N) \colon \, x_N > F(x') \}.
\]
Since $\partial \Omega$ is compact, the $C^1$ regularity implies that
\begin{itemize}\label{A1}
	\item[(A1)] for any given $\epsilon_F>0$, $\exists R=R(\Omega,\epsilon_F)>0$ with the property that for any $z\in \Omega$ there is $F\in C^{1}(\R^{N-1}, \R)$ with $F(0)=0$, $\grad{F}(0)=0$, $\norm{\grad{F}}_\infty < \epsilon_F$ and (up to a rotation): 
\[
	\Omega \cap B_{R}(z)=\{z+(x',x_N) \colon \abs{x}<R, \, x_N > F(x') \}= \Omega_F\cap B_R.
\]
\end{itemize} 
In other words $\partial \Omega$ is locally the graph of a $C^1$ function with small gradient over the tangent space $T_z\partial \Omega$.

Let $0<r\le R$ and  $z\in \partial \Omega$. We define the following scaled (and translated) $\Omega$:
\[
	\Omega_{z,r}= \{ x\in \R^{N}\colon z+rx \in \Omega \}.
\]
Boundary regularity is a local question so we will often consider
\[
\Omega_{z,r}\cap B_1 = \{(x',x_N) \colon \abs{x}<1, x_N > F_{0,r}(x') \}=\Omega_{F_{0,r}}\cap B_1
\]
with $F_{0,r}(x')= r^{-1} F(r x')$ ( observe that $\norm{\grad( F_{0,r})}_{\infty,B_1} = \norm{\grad F}_{\infty,B_r}$ ).\\
Frequently we will study such a special domain $\Omega_F$ defined by
\begin{itemize}\label{A2}
	\item[(A2)]
	\[
		\Omega_F=\{(x',x_N) \colon \, x_N > F(x') \}
	\]
	with $F\in C^{1}(\R^{N-1}, \R)$ with $F(0)=0$, $\grad{F}(0)=0$, $\norm{\grad{F}}_\infty < \epsilon_F$. Moreover we set
	\[
		\Gamma_F=\partial \Omega_F\cap B_1 = \{ (x',x_N)\colon \abs{x}< 1, x_N=F(x') \}.
	\]
	$\Gamma_F$ denotes a boundary portion of the boundary to such a special domain.
\end{itemize}

The upper half space $\R^N_+$ is a particular case of such a domain i.e. $\Omega_0=\R^N_+$ for $F=0$. The boundary  of the upper half ball $B_{1+}=\R^N_+\cap B_1$ is the union of $\Gamma_0=B_1\cap \{x_N=0\}$ and the upper half of the sphere $\Sp^{N-1}_+= \Sp^{N-1} \cap \{x_N>0\}$.\\

Fractional Soblev spaces, named $W^{s,2}$, occur naturally, when dealing with boundary regularity for elliptic problems. A short introduction is given in the appendix \ref{sec:fractional_sobolev_spaces}. We define the Gagliardo semi-norms for $0<s<1$ and $m$ dimensional submanifolds $\Sigma \subset \R^N$
\begin{align*}
	\llfloor f \rrfloor^2_{s,\Sigma} &= \int_{\Sigma\times\Sigma} \frac{\abs{f(x)-f(y)}^2}{\abs{x-y}^{m+2s}} \,dxdy, \quad f\in L^2(\Sigma)\\
	\llfloor u \rrfloor^2_{s,\Sigma} &= \int_{\Sigma\times\Sigma} \frac{\G(u(x),u(y))^2}{\abs{x-y}^{m+2s}} \,dxdy, \quad u\in L^2(\Sigma,\A_Q(\R^n)).
\end{align*}
The notation $\llfloor \cdot \rrfloor_{s, \Sigma}$ has been chosen in similarity to the classical notation $[\cdot]_{\alpha, \Sigma}$ for the H\"older semi-norm with exponent $\alpha$.
We extend it to $s=1$ by (abusing the notation a little):
\begin{align*}
	\llfloor f \rrfloor^2_{1,\Sigma} &= \int_{\Sigma} \abs{D_\tau f}^2, \quad f\in W^{1,2}(\Sigma)\\
	\llfloor u \rrfloor^2_{1,\Sigma} &= \int_{\Sigma} \abs{D_\tau u}^2, \quad u\in W^{1,2}(\Sigma,\A_Q(\R^n))
\end{align*}
where $D_\tau$ denotes the total tangential derivative on $\Sigma$. For single a valued functions $f \in W^{1,2}(\Sigma)$ and an orthonormal frame $\tau_1, \dotsc, \tau_m$ of $T_x\Sigma$ we have $\abs{D_\tau f(x)}^2= \sum_{j=1}^Q \abs{\frac{\partial f}{\partial \tau_j}}^2$. In the case of multivalued function $u$ we make use of the approximately differentiability of Sobolev functions: for a.e. $x\in \Sigma$ we have $\abs{D_\tau u}^2(x)= \sum_{j=1}^m\sum_{i=1}^Q \abs{U_i(x)\tau_j}^2$ where $U_i(x)$ are the elements of the 1-jet $J\mathcal{U}_x$, c.f. the the discussion below definition \ref{def_I:1.102} for precise statement to the approximate differentiability and the definition of the 1-jet.

\section{H\"older continuity for $N\ge 3$} 
\label{sec:h"older_continuity_for_nge_3_}
A more precise version of theorem \ref{theo_B:0.1} is:

\begin{theorem}\label{theo_B:2.1}
For any $\frac{1}{2}<s \le 1$, there are constants $C>0$ and $\alpha_1>0$ depending on $N,n,Q,s$, $N\ge 3$ with the property that, if
\begin{itemize}
	\item[(a1)] $u \in W^{1,2}(\Omega, \A_Q(\R^n))$ is Dirichlet minimizing;
	\item[(a2)] $u\tr{\partial \Omega} \in W^{s,2}(\partial \Omega, \A_Q(\R^n))$ and for some $0<\beta$ there is a constant $M_u>0$ s.t. 
	\[
		r^{2(s-\beta)-(N-1)} \llfloor u \rrfloor_{s,B_r(z)\cap \partial \Omega}^2 \le M_u^2 \text{ for all } z\in \partial \Omega, r>0;
	\]
\end{itemize}
then the following holds
\begin{itemize}
	\item[(i)] $\abs{Du}$ is an element of the Morrey space $L^{2,N-2+2\alpha}$ for any $0< \alpha < \min\{\alpha_1, \beta\}$, more precisely the following estimate holds
	\begin{equation}\label{eq_B:2.1}
		r^{2-N-2\alpha} \int_{B_r(x)\cap \Omega} \abs{Du}^2 \le 2^N R_0^{2-N-2\alpha} \int_{B_{2R_0}(x)\cap \Omega} \abs{Du}^2 + C \frac{R_0^{2(\beta-\alpha)}}{\beta-\alpha} M_u^2
	\end{equation}
	for any $r<\frac{R_0}{2}$. The positive constant $R_0$ depends only on $N,n,Q,s,\Omega$ but not on the specific $u$;\\
	\item[(ii)] $u \in C^{0,\alpha}(\overline{\Omega})$.
\end{itemize}
\end{theorem}

\begin{lemma}\label{lem_B:2.101}
	There is a relation between assumption (a2) and the H\"older continuity of $u\tr{\partial \Omega}$:
	\begin{itemize}
		\item[(i)] (a2) is satisfied if $u\tr{\partial \Omega}\in C^{0,\beta}(\partial \Omega)$ for $\beta>\frac{1}{2}$ i.e. there is a dimensional constant $C>0$ s.t. for $0<s<\beta$
		\[
			r^{2(s-\beta)-(N-1)} \llfloor u \rrfloor_{s,B_r(z)\cap \partial \Omega}^2 \le \frac{C}{\beta-s} [u]^2_{\beta,\partial \Omega} \quad\forall z\in \partial \Omega, 0<r<R(\Omega,1);
		\]
		\item[(ii)] if (a2) holds then $u\tr{\partial \Omega}\in C^{0,\beta}(\partial \Omega)$ i.e. there is a dimensional constant s.t.
		\[
			\G(u(x),u(y)) \le C M \abs{x-y}^{\beta} \quad \forall x,y \in \partial \Omega, \abs{x-y} \le \frac{R(\Omega,1)}{2}.
		\]
	\end{itemize} 
\end{lemma}
\begin{proof}
	To prove (i) let $z \in \partial \Omega$, $0<r<R(\Omega, 1)$ be given and $F\in C^1(\R^{N-1},\R)$ the function of (A1), then
	\begin{align*}
		&\int_{B_r(z)\cap \partial \Omega \times B_r(z)\cap \partial \Omega} \frac{\G(u(x), u(y))^2}{\abs{x-y}^{N-1+2s}}\, dx dy \\
		&\le [u]^2_{\beta, \partial \Omega} \int_{B_r(z)\cap \partial \Omega \times B_r(z)\cap \partial \Omega} \abs{x-y}^{2(\beta-s)-(N-1)} \,dx dy\\
		&\le [u]^2_{\beta, \partial \Omega} (1+ \norm{\grad(F)}_\infty^2)^2 \int_{B_r\times B_r} \abs{x'-y'}^{2(\beta-s)-(N-1)} \,dx' dy'\\
		&\le \frac{4 (N-1)\omega_{N-1}^2}{2(\beta-s)} [u]^2_{\beta, \partial \Omega}\, r^{2(\beta-s)+(N-1)}.
	\end{align*}
	To prove (ii) we observe that using the function $F$ of (A1) to write $\partial \Omega$ locally as a graph we can transform it to a local question on $\R^{N-1}$. Furthermore making use of Almgren's bilipschitz embedding, Theorem \ref{theo_I:1.101}, it is sufficient to check it for single valued functions. Hence (ii) is equivalent to check that\\
	\emph{ There is a dimensional constant $C>0$ s.t. if $f\in W^{s,2}(\R^N, \R^n)$ and $M_f>0$ be given with the property that 
	\begin{equation}\label{eq:141}
		r^{2(s-\beta)-N} \llfloor f \rrfloor^2_{s,B_r(z)} \le M_f^2 \quad \forall B_r(z) \subset \R^N, 0<r<R_0
	\end{equation}
	then $f\in C^{0,\beta}(\R^N, \R^n)$ with
	\begin{equation}\label{eq:142}
		\abs{f(x)-f(y)} \le C M_f \abs{x-y}^\beta \quad \forall \abs{x-y}<R_0.
	\end{equation}}
	Let us write $f(z,r)=\fint_{B_r(z)} f$ for any $B_r(z)\subset \R^N$, then using twice Cauchy's inequality we have
	\begin{align*}
		&\fint_{B_r(z)} \abs{f-f(z,r)} \le \abs{B_r(z)}^{-2} \int_{B_r(z)\times B_r(z)} \abs{f(x)-f(y)}\, dx dy \\
		&\le \abs{B_r(z)}^{-2} \int_{B_r(z)} \left( \int_{B_r(z)} \abs{x-y}^{N+2s}\,dy \right)^{\frac{1}{2}} \left(\int_{B_r(z)} \frac{\abs{f(x)-f(y)}^2}{\abs{x-y}^{N+2s}}\,dy\right)^{\frac{1}{2}} dx\\
		&\le \left(\frac{4^N}{\omega_N^2} r^{2s-N} \llfloor f\rrfloor_{s,B_r(z)}^2\right)^{\frac{1}{2}} \le C r^{\beta} \, M_f. 
	\end{align*} 
	Hence for any $r<R_0$ and $k \in \N$
	\[
		\abs{f(z,2^{-k-1}r)- f(z,2^{-k}r)} \le 2^N \fint_{B_{2^{-k}r}(z)} \abs{f- f(z,2^{-k}r)} \le C M_f\, r^{\beta}\, 2^{-\beta k};
	\]
	i.e. $k\mapsto f(z,2^{-k}r)$ is a Cauchy sequence because $\sum_{k=0}^\infty \abs{f(z,2^{-k-1}r)- f(z,2^{-k}r)} \le \frac{C M_f}{1-2^{-\beta}} r^\beta$. Furthermore for any $z_1,z_2 \in \R^N$ with $\abs{z_1-z_2}=r<R_0$ we finf
	\begin{align*}
		\abs{f(z_1)-f(z_2)} &\le\sum_{i=1}^2 \abs{f(z_i)-f(z_i,r)} + \fint_{B_r(z_i)\cap B_r(z_2)} \abs{f(x)-f(z_i)} \,dx\\
		&\le \sum_{i=1}^2 \frac{C M_f}{1-2^{-\beta}} r^\beta + \frac{C M_f}{1-2^{-\beta}} r^\beta \le 4 \frac{C M_f}{1-2^{-\beta}} r^\beta;
	\end{align*}
	this shows that $f \in C^{0,\beta}$.	
\end{proof}

The core of the proof of theorem \ref{theo_B:2.1} is the estimate stated in proposition \ref{prop_B:2.3} below. To make its proof more accessible it is  presented in the next subsection and split into several lemmas. 

\begin{proposition}\label{prop_B:2.3}
For any $\frac{1}{2}< s \le 1$ there are constants $\epsilon_0>0$, $0<\delta<\frac{1}{N-2}$ and $C>0$ depending on $N, n, Q,s$ with the property that, if (A2) holds with $\epsilon_F\le \epsilon_0$, then
\begin{equation}\label{eq_B:2.2}
	\int_{\Omega_F \cap B_1} \abs{Du}^2 \le \left(\frac{1}{N-2} - \delta \right) \int_{\Sp^{N-1}\cap \Omega_F} \abs{D_\tau u}^2 + C \llfloor u \rrfloor^2_{s,\Gamma_F}.
\end{equation}
for any Dirchilet minimizer $u \in W^{1,2}(B_1\cap \Omega_F, \A_Q(\R^n))$.
\end{proposition}
Let us take the previous proposition, i.e. the estimate \eqref{eq_B:2.2}, for granted and close the argument in the proof of theorem \ref{theo_B:2.1}.

\begin{proof}[Proof of Theorem \ref{theo_B:2.1}]
Let $\epsilon_0, \delta$ be the constants of proposition \ref{prop_B:2.3}. Fix $\alpha_1\le \alpha_0$ ( $\alpha_0$ being the H\"older exponent of theorem \ref{theo_I:1.102} ) s.t. $(N-2+2\alpha_1)\left(\frac{1}{N-2}-\delta\right)\le 1$. Let $R_0=R_0(\Omega, \epsilon_0)$ be the radius defined of (A1) for $\epsilon_F=\epsilon_0$\\
	
Due to the choice of $R_0$, for any $0<r\le R_0$, $z \in \partial \Omega$ the rescaled map
\[
	u_{z,r}(x)=u(z+rx) \quad \text{ for } x \in B_1\cap \Omega_{z,r}
\]	
belongs to $W^{1,2}(\Omega_{z,r}\cap B_1, \A_Q(\R^n))$ and satisfies the assumptions of the proposition \ref{prop_B:2.3}. One readily checks that for $\frac{1}{2}<s\le 1$
\[
\llfloor u_{z,r} \rrfloor^2_{s,B_1\cap \partial \Omega_{z,r}}= r^{2s-(N-1)}\llfloor u \rrfloor^2_{s,B_r(z)\cap \partial \Omega}.
\]
Applying \eqref{eq_B:2.2} and assumption (a2) we get
\begin{align*}
	&r^{2-N} \int_{B_r(z)\cap \Omega} \abs{Du}^2 = \int_{B_1 \cap \Omega_{z,r}} \abs{Du_{z,r}}^2 \\
	&\le \left(\frac{1}{N-2}-\delta\right) \int_{\Sp^{N-1}\cap \Omega_{z,r}} \abs{D_\tau u_{z,r}}^2 + C \llfloor u_{z,r} \rrfloor^2_{s, B_1 \cap \partial \Omega_{z,r}}\\
	&\le \frac{1}{N-2+2\alpha_1} r^{3-N}\int_{\partial B_r(z)\cap \Omega} \abs{D_\tau u}^2 + C r^{2\beta} M_u^2.
\end{align*}
Hence for a.e. $0<r<R_0$ and $0<\alpha<\min\{\alpha_1, \beta\}$
\begin{align*}
	&-\frac{\partial }{\partial r} \left( r^{2-N-2\alpha} \int_{B_r(z)\cap \Omega} \abs{Du}^2\right)\\& = - r^{2-N-2\alpha} \int_{\partial B_r(z)\cap \Omega} \abs{Du}^2 + (N-2+2\alpha) r^{-1-2\alpha} r^{2-N} \int_{B_r(z)\cap \Omega} \abs{Du}^2\\
	&\le r^{2-N-2\alpha} \int_{\partial B_r(z)\cap \Omega} \abs{D_\tau u}^2 - \abs{Du}^2 + (N-2+2\alpha) C r^{2(\beta-\alpha)-1} M_u^2\\
	&\le (N-2+2\alpha) C r^{2(\beta-\alpha)-1} M_u^2.
\end{align*}
Integrating in $r$ we achieve the following inequality for any $z \in \partial \Omega$ and $0< r \le R_0$:
\begin{equation}\label{eq_B:2.3}
	r^{2-N-2\alpha} \int_{B_r(z)\cap \Omega} \abs{Du}^2  - R_0^{2-N-2\alpha} \int_{B_{R_0}(z)\cap \Omega} \abs{Du}^2 \le \frac{C}{\beta-\alpha} R_0^{2(\beta-\alpha)} \, M_u^2.
\end{equation}
Now we can conclude \eqref{eq_B:2.1}. If $x\in \overline{\Omega}$ satisfies $\operatorname{dist}(x, \partial \Omega)> \frac{R_0}{2}$, then $B_r(x)\subset B_{\frac{R_0}{2}}(x)\subset \Omega$ for any $0<r<\frac{R_0}{2}$ and so, by \eqref{eq_I:1.103} in Theorem \ref{theo_I:1.103}
\begin{align}\label{eq_B:2.4}
	r^{2-N-2\alpha}\int_{B_r(x)} \abs{Du}^2 &\le \left(\frac{R_0}{2}\right)^{2-N-2\alpha} \int_{B_{\frac{R_0}{2}}(x)} \abs{Du}^2\\ \nonumber &\le 2^N R_0^{2-N-2\alpha} \int_{B_{2R_0}(x)\cap \Omega} \abs{Du}^2.
\end{align}
Assume therefore $x \in \overline{\Omega}$ has $\operatorname{dist}(x,\partial \Omega)\le \frac{R_0}{2}$. Fix $z\in \partial \Omega$ s.t. $\operatorname{dist}(x,\partial \Omega)=\abs{x-z}$, and for $0<r\le \frac{R_0}{2}$ set $r_1=\max\{r, \abs{x-z}\}$, $r_2=r_1+\abs{x-z}\le 2r_1\le R_0$. Then 
\begin{align}\label{eq_B:2.5}
	&r^{2-N-2\alpha} \int_{B_r(x)\cap \Omega} \abs{Du}^2 \le {r_1}^{2-N-2\alpha} \int_{B_{r_1}(x)\cap \Omega} \abs{Du}^2\\ \nonumber
	&\le \left(\frac{r_2}{r_1}\right)^{N-2+2\alpha} r_2^{2-N-2\alpha} \int_{B_{r_2}(z)\cap \Omega} \abs{Du}^2 \\\nonumber &\le 2^{N} \left(R_0^{2-N-2\alpha} \int_{B_{R_0}(z)\cap \Omega} \abs{Du}^2 + \frac{C}{\beta-\alpha} R_0^{2(\beta-\alpha)} \, M_u^2 \right) \\ \nonumber
	&\le 2^{N} \left( R_0^{2-N-2\alpha} \int_{B_{2 R_0}(x)\cap \Omega} \abs{Du}^2 + \frac{C}{\beta-\alpha} R_0^{2(\beta-\alpha)} \, M_u^2 \right).
\end{align}

The fact (ii) i.e. $u \in C^{0,\alpha}(\overline{\Omega})$ follows now classically. We established that $\abs{Du}$ is an element of the Morrey space $L^{2,N-2+2\alpha}(\Omega)$. $\Omega$ is $C^1$ regular and therefore by Poincar\'es inequality this implies that $\boldsymbol{\xi} \circ u$ is an element of the Campanato space $\mathcal{L}^{2,N+2\alpha}(\Omega)$, see for instance \cite[Proposition 3.7]{Giusti}. Furthermore  $\mathcal{L}^{2,N+2\alpha}(\Omega)= C^{0,\alpha}(\overline{\Omega})$, \cite[Theorem 2.9]{Giusti}.
\end{proof}

\subsection{Proof of Proposition \ref{prop_B:2.3}} 
\label{sub:proof_of_proposition}

The proof can be subdivided into two parts:\\
\emph{paragraph \ref{ssub:non_existence_of_certain_non_trivial_minimizers}:}\\
We show that it is necessary and sufficient for a Dirichlet minimizer on the upper half ball $B_1 \cap \{ x_N > 0 \}$ to be trivial that it has constant boundary data on $B_1 \cap \{x_N = 0 \}$. \\
\emph{paragraph \ref{ssub:contradiction_argument}:}\\
We show that if proposition would fail we could construct a non-trivial Dirichlet minimizer on the upper half ball $B_1\cap \{ x_N > 0 \}$ with constant boundary data contradicting the previous step.

\subsubsection{Non-existence of certain non-trivial minimizers} 
\label{ssub:non_existence_of_certain_non_trivial_minimizers}
This paragraph is devoted to establish the following two results for certain Dirichlet minimizers on the upper half ball $B_{1+}=B_1\cap \{x_N >0 \}$ , recalling that $\Sp^{N-1}_+ = \Sp^{N-1}\cap \{ x_N>0 \}$ and $\Gamma_0 =B_1\cap \{x_N=0\}$.

\begin{proposition}\label{prop_B:2.4}
Every 0-homogeneous Dirichlet minimizer in $B_{1+}$ with $u\tr{\Gamma_0}= const.$ is trivial i.e. constant.
\end{proposition}

\begin{corollary}\label{cor_B:2.5}
A Dirichlet minimizer on $B_{1+}$ with $u\tr{\Gamma_0}= const. $ satisfying
\begin{equation}\label{eq_B:2.6}
\int_{B_{1+}} \abs{Du}^{2}=\frac{1}{N-2} \int_{\Sp^{N-1}_+} \abs{D_{\tau}u}^{2}
\end{equation}
needs to be constant.
\end{corollary}

They are both consequence of an appropriately chosen inner variation:

\begin{lemma}[a special kind of inner variation]\label{lem_B:2.6}
Given a Dirichlet minimizer $u\in W^{1,2}(B_{1+},\A_Q(\R^n))$ with $u\tr{\Gamma_0}=const.$ and a vector field $ X=(X_1, \dotsc, X_N) \in C^{1}_{c}(B_1, \R^N)$ with ${e_{N}\cdot X(x',0)= X_N(x',0)\ge 0}$ on $\Gamma_0$, then
\begin{equation}\label{eq_B:2.7}
0\le \int_{B_{1+}} \abs{Du}^{2} \Div (X) - 2\sum^{Q}_{i=1} \langle Du_{i} : Du_{i}\,DX \rangle.
\end{equation}
\end{lemma}

\begin{proof}
Let $u$ and $X$ be given and set $T= u\tr{\Gamma_0}(x)$ for $x\in \Gamma_0$.  Observe that $x_N + tX_N(x',x_N) = x_N + t \left( X_N(x', x_N) - X_N(x',0) \right) + t X_N(x',0)\ge (1- t \norm{DX_N}_\infty) x_N + t X_N(x',0) \ge 0$ for $x_N >0$ and sufficient small $ 0<t<t_0$. Then for $t_0>0$ small  
\begin{equation*}\label{eq:2.1021}
\Phi_{t}(x)= x + t X(x)
\end{equation*}
defines a 1-parameter family of $C^{1}$-diffeomorphism that satisfy
\begin{equation*}
A_t= \Phi_{t}(B_{1+}) \subset B_{1+} \text{ for }  0\le t \le t_{0}.
\end{equation*}
So
\begin{equation*}
v_{t}(x)=
\begin{cases}u\circ \Phi_{t}^{-1}(x) & \text{ for } x \in A_t\\
T &\text{ for } x \in B_1^{+}\setminus A_{t}
\end{cases}
\end{equation*}
defines a $C^1$ family of competitors to $u$. Standard calculations give
\begin{align*}
D\Phi^{-1}_t \circ \Phi_t &= \left(D\Phi_t\right)^{-1}= \sum_{k=0}^\infty (-t)^k \left(DX\right)^k = 1 - t DX + o(t)\\
\det \left(D\Phi_t\right) &= 1 + t \Div(X) + o(t)
\end{align*}
so that
\begin{align*}
\abs{Dv_t}^2\circ \Phi_t &= \sum_{i=1}^Q \abs{Du_i D\Phi_t^{-1}\circ \Phi_t}^2 = \sum_{i=1}^Q \abs{Du_i \left(1 - t DX + o(t) \right)}^2\\
&= \sum_{i=1}^Q \abs{Du_i}^2 - 2 t \sum_{i=1}^Q \langle Du_i : Du_i DX \rangle + o(t).
\end{align*}
In total we found that for all $0 \le t \le t_0$
\begin{align*}
\int_{B_{1+}} \abs{Dv_t}^2 &= \int_{A_t} \abs{Dv_t}^2 = \int_{B_{1+}} \abs{Dv_t}^2 \circ \Phi_t \,\abs{\det D\Phi_t} \\
&= \int_{B_{1+}} \abs{Du}^2 + t \int_{B_{1+}} \abs{Du}^2 \Div(X) - 2 \sum_{i=1}^Q \langle Du_i : Du_i DX \rangle + o(t).
\end{align*}
Since $\int_{B_{1+}} \abs{Dv_t}^2 \ge \int_{B_{1+}} \abs{Du}^2$, we necessarily have
\begin{equation*}
0 \le \int_{B_{1+}} \abs{Du}^2 \Div(X) - 2 \sum_{i=1}^Q \langle Du_i : Du_i DX \rangle.
\end{equation*}
\end{proof}

\begin{proof}[Proof of Proposition \ref{prop_B:2.4}]
$u$ being $0-$homogeneous implies that $u(x)= u(\frac{x}{\abs{x}})$ for a.e. $x$. Thus $\frac{\partial u}{\partial r}(x) =0$ for a.e. $x\in B_{1+}$, which corresponds to 
\begin{equation}\label{eq_B:2.8}
0=\frac{\partial u}{\partial r} (x)= \sum_{i=1}^Q \Big\llbracket \sum_{j=1}^N D_ju_i(x) \frac{x_j}{\abs{x}} \Big\rrbracket.
\end{equation}
Fix $0< R < 1$ and consider the vector field $X(x)=\eta(\abs{x}) e_N=(0, \dotsc, \eta(\abs{x}))$ with
\begin{equation*}
\eta(r)=
\begin{cases}1 - \frac{r}{R} & r \le R\\ 0 &r \ge R.
\end{cases}
\end{equation*}
Thus we have $X_N(x) \ge 0$ and $DX(x)= \eta'(\abs{x}) e_N \otimes \frac{x}{\abs{x}}$. This gives $\Div(X)(x) = \eta'(\abs{x}) \frac{x_N}{\abs{x}}$ and due to \eqref{eq_B:2.8}
\begin{equation*}
\langle Du_i : Du_i DX \rangle 
= \sum_{j=1}^N \left\langle \frac{x_j}{\abs{x}} D_j u_i, D_N u_i\right\rangle \eta'({\abs{x}}) =0 \text{ for a.e. } x.
\end{equation*} 
Using $\eta'(\abs{x})=-\frac{1}{R} \textbf{1}_{B_R}(x)$ and applying Lemma \ref{lem_B:2.6} we get
\begin{equation*}
0 \le -\frac{1}{R} \int_{B_{R+}}\abs{Du}^2 \frac{x_N}{\abs{x}}.
\end{equation*}
This is only possible for $\abs{Du}=0$ on $B_{R+}$ and so $\abs{Du}=0$ on $B_{1+}$.
\end{proof}

\begin{proof}[Proof of corollary \ref{cor_B:2.5}]
Let $u \in W^{1,2}(B_{1+}, \A_Q(\R^n))$ be as assumed. Observe that \eqref{eq_B:2.6} implies that $u \in W^{1,2}(\Sp^{N-1}_+,\A_Q(\R^n))$. Hence $v(x)=u\left(\frac{x}{\abs{x}}\right)$ defines a $0$-homogeneous competitor using $u\tr{\Gamma_0}= const.$. 
\[
\int_{B_{1+}} \abs{Dv}^{2} =\frac{1}{N-2} \int_{\Sp^{N-1}_+} \abs{D_{\tau}v}^{2}=\frac{1}{N-2} \int_{\Sp^{N-1}_+} \abs{D_{\tau}u}^{2}=  \int_{B_{1+}} \abs{Du}^{2}.	
\]
where we used firstly the $0-$homogeneity of $v$, then $u\tr{\Sp^{N-1}_+}=v\tr{\Sp^{N-1}_+}$ and finally \eqref{eq_B:2.6}.
Therefore $v$ has to be minimizing as well, and moreover $Dv=0$ as a consequence of proposition \ref{prop_B:2.4}. This proves the corollary since then $Du=0$ as well. 
\end{proof}


\subsubsection{contradiction argument} 
\label{ssub:contradiction_argument}
In this section we want to establish by contradiction the estimate of Proposition \ref{prop_B:2.3}
\[
	\int_{\Omega_F \cap B_1} \abs{Du}^2 \le \left(\frac{1}{N-2} - \delta \right) \int_{\Sp^{N-1}\cap \Omega_F} \abs{D_\tau u}^2 + C \llfloor u \rrfloor^2_{s,\Gamma_F}.
\]
To prove Theorem \ref{theo_B:2.1} from such an estimate we only needed the scaling property $\llfloor u_{z,r} \rrfloor^2_{s,B_1\cap \partial \Omega_{z,r}}= r^{2s-(N-1)}\llfloor u \rrfloor^2_{s,B_r(z)\cap \partial \Omega}$ and the existence of positive constants $\beta, M_u>0$ both depending possibly on $u$ s.t. in combination $\llfloor u_{z,r} \rrfloor_{s,B_1\cap \partial \Omega_{z,r}} \le r^\beta M_u$.\\
Before coming to the proof we discuass some subtleties in the strategy.A $C^{0,\beta}$-H\"older norm, $[u]_{\beta,\Sigma}= \sup_{x,y \in \Sigma} \frac{\G(u(x),u(y))}{\abs{x-y}^\beta}$, for any $0<\beta<1$ shares this property since
\[
	[u_{r,z}]_{\beta, \partial \Omega_{z,r}\cap B_1} = r^\beta [u]_{\beta, \partial \Omega \cap B_1(z)} \le r^\beta [u]_{\beta, \partial \Omega}.
\]
Replacing the $W^{s,2}(\partial \Omega)$-norm, ($s>\frac{1}{2}$) by a H\"older-norm with exponent $\beta<\frac{1}{2}$ would be desirable since it would get us closer to the already mentioned classical result: $u \in W^{1,2}(\Omega)$ harmonic with $u\tr{\partial \Omega}\in C^{0,\beta}(\partial \Omega)$ for some $\beta>0$ implies $u \in C^{0,\beta}{\overline{\Omega}}$. \\
Nonetheless we cannot hope to prove an estimate like \eqref{eq_B:2.2} by contradiction if the fractional Sobolev norm ($s>\frac{1}{2}$) is replaced by an $C^{0,\beta}$-H\"older norm, $\beta< \frac{1}{2}$ because vanishing of energy through the boundary needs to be excluded. Bounds on $W^{s,2}(\partial \Omega)$-, or $C^{0,s}(\partial \Omega)$-norms with $s<\frac{1}{2}$ are insufficient. This is demonstrated by the following two dimensional example on the disc $B_1\subset \R^2$. It uses polar coordinates $x=\bigl( \begin{smallmatrix} r \cos(\theta)\\ r\sin(\theta) \end{smallmatrix} \bigr)=re^{i\theta}$. 

\begin{example}\label{ex_B:2.1}
	For any $\epsilon>0$ there is a sequence of harmonic functions $f_k \in W^{1,2}(B_1, \R)$ a positive constant $c>0$ with the following properties: for all $k$ we have $\int_{B_1} \abs{Df_k}^2 >c$, $f_k(e^{i\theta})=0$ for $\abs{\theta}>\epsilon$. Furthermore $f_k \to 0$ uniformly on $B_1$ and $\norm{f_k}_{s, \Sp^1}, [f_k]_{s,\Sp^1} \to 0$ for every $s<\frac{1}{2}$.
\end{example}
\begin{proof}[Proof of example \ref{ex_B:2.1}]
To a given $0 < \epsilon < \frac{\pi}{2}$, fix a smooth, symmetric, non-negative bump function $\eta$ with $\eta(0)>0$ and $\eta(\theta)=0$ for $\abs{\theta}\ge \epsilon$. Let $\sum_{l=0}^\infty a_l \cos(l\theta)$ be the Fourier series of $\eta(\theta)$. It is converging uniformly to $\eta$ in the $C^\infty$ topology since $\eta$ is smooth and $\sum_{l=0}^\infty l^m \abs{a_l} < \infty$ for all $m\in \N$. Fix $k_0 \in \N $ sufficient large s.t. $2 \abs{a_k} < a_0=\eta(0) $ for $k \ge k_0$ and set $A= \sum_{l=0}^\infty (l+1) \abs{a_l}\ge \left(\sum_{l=0}^\infty(l+1)a_l^2\right)^{\frac{1}{2}}$. The addition theorem $2\cos(l\theta)\cos(k\theta)= \cos((l+k) \theta) + \cos((l-k)\theta)$ shows that the harmonic extension of $2\eta(\theta)\cos(k\theta)$ in $B_1$ is
\begin{align*}
	g_k(r \, e^{i \theta}) &= \sum_{l=0}^\infty a_l \left(r^{l+k} \cos((l+k)\theta)+ r^{\abs{l-k}} \cos((l-k)\theta)\right)\\
	&=\sum_{m=0}^\infty (a_{m-k}+ a_{m+k}) r^m \cos(m\theta) \quad \text{ with } a_{m-k}=0 \text{ for } m<k.
\end{align*}
For $k \ge k_0$
\begin{align*}
	\frac{1}{\pi}\int_{B_1} \abs{Dg_k}^2 &= \sum_{m=1}^\infty m (a_{m-k}+ a_{m+k} )^2\\
	&\ge k (a_0 + a_{2k})^2 \ge \frac{1}{4} k a_0^2\\
	&\le 2 \sum_{l=0}^\infty (l+k) a_l^2 + \abs{l-k} a_l^2 \le 4k A^2.
\end{align*}
We consider now the sequence of harmonic functions on $B_1$ given by $f_k(x)=\frac{g_k(x)}{k^{\frac{1}{2}}} \in W^{1,2}(B_1)$. $f_k$ has the desired properties: using the equivalence  
\begin{itemize}
	\item[(i)] $\frac{1}{4}a_0^2 \le \frac{1}{\pi}\int_{B_1} \abs{Df_k}^2 = \norm{f_k}^2_{\frac{1}{2},S^1} \le 4 A^2$ for all $k\ge k_0$;
	\item[(ii)] $f_k(e^{i\theta})=0$ for $\abs{\theta}>\epsilon$ and all $k$;
	\item[(iii)] $\norm{f_k}_{\infty} \le \frac{2 \norm{\eta}_\infty}{k^{\frac{1}{2}}} \to 0 $ as $k \to \infty$;
	\item[(iv)] for any $0<s<\frac{1}{2}$
	\begin{align*}
		\norm{f_k}^2_{s,\Sp^1} &= \sum_{m=0}^\infty \frac{m^{2s}}{k} (a_{m-k}+a_{m+k})^2 \le 8 k^{2s-1} A^2\\
		[f_k]_{s,\Sp^1} &\le \sum_{m=0}^\infty \frac{m^{s}}{k^{\frac{1}{2}}} \abs{a_{m-k}+a_{m+k}} \le 2 k^{s-\frac{1}{2}} \sum_{l=0}^\infty (l+1)\abs{a_l}
	\end{align*}
	converging to $0$ as $k \to \infty$.
\end{itemize}
(iii) follows from the maximum principle on harmonic functions. The fact that the $W^{s,2}$-norm on $\Sp^1$ corresponds to the sum in (iii), i.e. the equivalence $H^s(\Sp^1)=W^{s,2}(\Sp^1)$, is the content of corollary \ref{cor:A2.5}. It is straightforward to check that one has $[\varphi]_{\beta, \Sp^1}\le \sum_{l=0}^\infty l^\beta \abs{c_l}$ for a converging Fourier series $\varphi(\theta)=\sum_{l=0}^\infty c_l \cos(l\theta)$. 
\end{proof}

\begin{proof}[Proof of proposition \ref{prop_B:2.3}]
If $u \notin W^{1,2}(\Sp^{N-1}\cap \Omega_F, \A_Q(\R^n))\cap W^{s,2}(\Gamma_F, \A_Q(\R^n))$ the LHS of \eqref{eq_B:2.2} is infinite and so there is nothing to prove. Hence assuming that the proposition would not hold, we can find sequences $F(k) \in C^1(\R^{N-1},\R)$ satisfying (A2) with $\epsilon_F< \frac{1}{k}$ and associated $u(k)\in W^{1,2}(B_1\cap \Omega_{F(k)}, \A_Q(\R^n))$ failing \eqref{eq_B:2.2} i.e.
\begin{equation}\label{eq_B:2.9}
		\int_{\Omega_{F(k)} \cap B_1} \abs{Du(k)}^2 > \left(\frac{1}{N-2} - \frac{1}{k} \right) \int_{\Sp^{N-1}\cap \Omega_{F(k)}} \abs{D_\tau u(k)}^2 + k \llfloor u(k) \rrfloor^2_{s,\Gamma_{F(k)}}.
\end{equation}
We may assume that the LHS of \eqref{eq_B:2.9} is $1$ by dividing each $u(k)$ by its Dirichlet energy $\left( \int_{\Omega_{F(k)} \cap B_1} \abs{Du(k)}^2 \right)^{-\frac{1}{2}}$. We also assume, w.l.o.g., $k>k_0>4$. \\

To every $k$ we may fix a $C^1$-diffeomorphism $G(k): \overline{B_{1+}} \to \overline{\Omega_{F(k)} \cap B_1}$, arguing for example on the base of Lemma \ref{lem:A4.2}. $F(k) \to F_0=0$ in $C^1$ as $k \to \infty$ and therefore $G(k),G(k)^{-1} \to \mathbf{1}$ in $C^1$ ($\textbf{1}$ deontes the indentiy map on $\R^N$).\\
We consider now instead of the sequence $u(k)$ itself the sequence $v(k)=u(k) \circ G(k) \in W^{1,2}(B_{1+}, \A_Q(\R^n))$. $v(k)$ has up to order $o(1)$ the same properties as $u(k)$ since $G(k),G(k)^{-1} \to \mathbf{1}$ in $C^1$ i.e.

\begin{align}\label{eq_B:2.10}
	\int_{B_1^+} \abs{Dv(k)}^2 &= (1+o(1)) \int_{\Omega_{F(k)}} \abs{Du(k)}^2 &\le& 1+ o(1); \nonumber\\
	\int_{\Sp^{N-1}_+} \abs{D_\tau v(k)}^2 &=(1+o(1)) \int_{\Sp^{N-1}\cap \Omega_{F(k)}} \abs{D_\tau u(k)}^2 &<& \frac{1+o(1)}{\frac{1}{N-2}- \frac{1}{k}}< 2N; \\
	 \llfloor v(k)\rrfloor^2_{s,\Gamma_0} &=  (1+o(1)) \llfloor u(k) \rrfloor^2_{s, \Gamma_{F(k)}}  &\le& \frac{1+o(1)}{k}\le \frac{1}{2k}\nonumber.
\end{align}
\eqref{eq_B:2.9} with LHS$=1$ provides the upper bounds. The second and third show that $v(k)\tr{\partial B_{1+}} \in W^{1,2}(\Sp^{N-1}_+, \A_Q(\R^n)) \cap W^{s,2}(\Gamma_0, \A_Q(\R^n))$.\\
To every $k$ fix a mean $T(k)\in \A_Q(\R^n)$ and apply the concentration compactness Lemma \ref{lem_I:A1.1} to the sequences $v(k)$, $T(k)$. For a subsequence $v(k')$ we can find maps $b_j \in W^{1,2}(B_{1+}, \A_{Q_j}(\R^n))$, sequences $t_j(k') \in spt(T(k'))$ and a splitting $T(k')=T_1(k') + \dotsb + T_J(k')$. We will prove now that the $b_j$ satisfy also the following:
\begin{itemize}
	\item[(i)] $b_j\tr{\Sp^{N-1}_+} \in W^{1,2}(\Sp^{N-1}_+, \A_{Q_j}(\R^n))$ and $b_j\tr{\Gamma_0}=const.$;
	\item[(ii)] $\int_{B_{1+}} \abs{Db_j}^2 \le \frac{1}{N-2} \int_{\Sp^{N-1}_+} \abs{D_\tau b_j}^2$ for all $j$;
	\item[(iii)] $b_j \in W^{1,2}(B_{1+}, \A_{Q_j}(\R^n))$ is Dirichlet minimizing and
	\[
		\sum_{j=1}^J \int_{B_{1+}} \abs{Db_j}^2= \lim_{k'\to \infty} \int_{B_{1+}} \abs{Dv(k')}^2 = \lim_{k' \to \infty} \int_{\Omega_{F_{k'}}\cap B_1} \abs{Du(k')}^2=1.
	\]
\end{itemize}
From now on we use $b(k')=\sum_{j=1}^J (b_j \oplus t_j(k'))$ as in the proof of the concentration compactness result. \\

\emph{Proof of (i):} The concentration compactness lemma states that $\boldsymbol{\xi} \circ v(k') \rightharpoonup \boldsymbol{\xi}\circ b(k')$ in $W^{1,2}(B_{1+}, \R^m)$ and $\boldsymbol{\xi}\circ v(k') \to \boldsymbol{\xi} \circ b(k')$ in $L^2(B_{1+}, \R^m)$. This implies that $\boldsymbol{\xi} \circ v(k') \rightharpoonup \boldsymbol{\xi}\circ b(k)$ in $W^{1,2}(\Sp^{N-1}_+, \R^m)$ and $\boldsymbol{\xi} \circ v(k') \to \boldsymbol{\xi}\circ b(k')$ in $L^2(\Sp^{N-1}_+, \R^m)$, because $\boldsymbol{\xi} \circ v(k') \in W^{1,2}(\Sp^{N-1}_+, \R^m)$ is uniformly bounded as seen in \eqref{eq_B:2.10}. The lower semicontinuity of energy together with \eqref{eq_B:2.10} then states 
\begin{align}\label{eq_B:2.11}
	&\frac{1}{N-2}\sum_{j=1}^J \int_{\Sp^{N-1}_+} \abs{D_\tau b_j}^2= \frac{1}{N-2} \int_{\Sp^{N-1}_+} \sum_{j=1}^J\abs{D_\tau \boldsymbol{\xi}\circ b_j}^2\\\nonumber
	  &\le \liminf_{k'\to \infty} \left(\frac{1}{N-2}- \frac{1}{k'} \int_{\Sp^{N-1}_+} \abs{D_\tau \boldsymbol{\xi} \circ v(k')}^2 \right)\le 1.
\end{align}
$\G(v\tr{\Gamma_0}(k'),b\tr{\Gamma_0}(k')) \to 0$ in $L^2(\Gamma_0)$ due to the weak convergence in the interior. Hence due to dominated convergence for any $\delta >0$ and \eqref{eq_B:2.10}
\begin{align*}
	&\sum_{j=1}^J \int_{\substack{\Gamma_0 \times \Gamma_0\\ \abs{x-y} \ge \delta}} \frac{\G(b_j\tr{\Gamma_0}(x), b_j\tr{\Gamma_0}(y))^2}{\abs{x-y}^{N-1+2s}} \\
	&= \lim_{k' \to \infty} \int_{\substack{\Gamma_0 \times \Gamma_0\\ \abs{x-y} \ge \delta}} \frac{\G(v\tr{\Gamma_0}(k')(x), v\tr{\Gamma_0}(k')(y))^2}{\abs{x-y}^{N-1+2s}} \le \lim_{k' \to \infty} \frac{2}{k'} = 0;
\end{align*}
consequently $b_j\tr{\Gamma_0} = const.$ for all $j$.\\

\emph{Proof of (ii):} Having established (i), $a_j(x)= b_j\left(\frac{x}{\abs{x}}\right) \in W^{1,2}(B_{1+}, \A_{Q_j}(\R^n))$ is well-defined and an admissible competitor. 
\[
	\int_{B_{1+}}\abs{Db_j}^2 \le \int_{B_{1+}} \abs{Da_j}^2 = \frac{1}{N-2} \int_{\Sp^{N-1}_+} \abs{D_\tau a_j}^2 = \frac{1}{N-2} \int_{\Sp^{N-1}_+} \abs{D_\tau b_j}^2
\]
for every $j$ due to the $0$-homogeneity of $a_j$ and $a_j\tr{\Sp^{N-1}_+}=b_j\tr{\Sp^{N-1}_+}$.

\emph{Proof of (iii):} Let $G: \overline{B_1} \to \overline{B_{1+}}$ be the bilipschitz map constructed in Lemma \ref{lem:A4.1}. $\llfloor v(k')\circ G\rrfloor_{s, \Sp^{N-1}}$ is uniformly bounded: Firstly apply Corollary \ref{cor:A2.101} to estimate
\[
	\llfloor v(k')\circ G \rrfloor_{s,\Sp^{N-1}} \le C \left(\llfloor v(k')\circ G \rrfloor_{s, \scriptscriptstyle{\Sp^{N-1}\cap \{ x_N > \frac{-1}{\sqrt{5}}\}}} + \llfloor v(k')\circ G \rrfloor_{s,\scriptscriptstyle{ \Sp^{N-1}\cap \{ x_N < \frac{-1}{\sqrt{5}}\}}} \right);
\] 
secondly $G$ is bilipschitz and $G(\Sp^{N-1}\cap \{x_N > \frac{-1}{\sqrt{5}}\})=\Sp^{N-1}_+$ and $G(\Sp^{N-1}\cap \{x_N < \frac{-1}{\sqrt{5}}\})=\Gamma_0$, so that
\begin{align*}
	\llfloor v(k')\circ G \rrfloor_{s, \Sp^{N-1}\cap \{ x_N > \frac{-1}{\sqrt{5}}\}} &\le C \llfloor v(k') \rrfloor_{s, S_+^{N-1}}\\
	\llfloor v(k')\circ G \rrfloor_{s, \Sp^{N-1}\cap \{ x_N < \frac{-1}{\sqrt{5}}\}} &\le C \llfloor v(k') \rrfloor_{s, \Gamma_0};
\end{align*}
thirdly the interpolation property $\llfloor f \rrfloor^2_{s, \Sp^{N-1}_+} \le C \int_{\Sp^{N-1}_+} \abs{Df}^2 $ gives
\[
	\llfloor v(k') \rrfloor_{s, S_+^{N-1}} \le \norm{\abs{Dv(k')}}_{L^2(\Sp^{N-1}_+)};
\]
finally we combine all of them and use \eqref{eq_B:2.10} to conclude
\[
	\llfloor v(k')\circ G \rrfloor_{s,\Sp^{N-1}} \le C \left( \norm{\abs{Dv(k')}}_{L^2(\Sp^{N-1}_+)} + \llfloor v(k') \rrfloor_{s, \Gamma_0} \right) \le C\left( 2N \right).
\] 
The same bound holds for $b(k')\circ G \in W^{s,2}(\Sp^{N-1}, \A_Q(\R^n))$ because of the lower semicontinuity of energy established in \eqref{eq_B:2.11}. Furthermore in the proof of (i) we showed that $\G(v(k'), b(k')) \to 0$ in $L^2(\Sp^{N-1}_+)$ and $L^2(\Gamma_0)$, so that
\[
	\norm{\G(v(k')\circ G, b(k')\circ G)}_{L^2(\Sp^{N-1})}= o(1).
\]
Fix any small $\epsilon >0 $ and $R_\epsilon>0$ determined by the  interpolation Lemma \ref{lem:A2.102}. So to every $k'$ we can find $w(k') \in W^{s,2}(A_{1,R_\epsilon}, \A_Q(\R^n))$ on the annulus $A_{1,R_\epsilon}=B_1\setminus B_{R_\epsilon}$ interpolating between $v(k')\circ G$ and $b(k')\circ G$. Hence $w(k')(x)=v(k')\circ G(x)$, $w(k')(R_\epsilon x) =b(k')\circ G(x)$ for all $x \in \Sp^{N-1}$ and 
\begin{align*}
	&\int_{A_{1,R_\epsilon}} \abs{Dw(k')}^2\\ &\le \epsilon \left(\llfloor v(k')\circ G \rrfloor^2_{s, \Sp^{N-1}} + \llfloor b(k') \circ G \rrfloor^2_{s, \Sp^{N-1}}\right) + C \norm{\G(v(k')\circ G, b(k')\circ G)}^2_{L^2(\Sp^{N-1})}\\
	&\le \epsilon \,C \,4N + C o(1).
\end{align*}
To check the minimizing property let $c_j \in W^{1,2}(B_{1+},\A_{Q_j}(\R^n))$ be an arbitrary competitor to $b_j$ for $j=1, \dotsc, J$. Set $c(k')=\sum_{j=1}^J \left( c_j \oplus t_j(k')\right)$. For $0<R\le 1$ we denote the map $G\circ \frac{1}{R} \circ G^{-1}(x)=\frac{e_N}{2}+ \frac{1}{R}\left(x-\frac{e_N}{2}\right)$ by $\psi_R$. So we found  
\[
\int_{C_R} \abs{D c(k')\circ \psi_{R}}^2 = R^{N-2} \int_{B_{1+}} \abs{Dc(k')}^2 \le \int_{B_{1+}} \abs{Dc(k')}^2
\]
with $C_R=\psi_{R}^{-1}(B_{1+})\subset B_{1+}$. We define $C(k') \in W^{1,2}(B_{1+}, \A_Q(\R^n))$ considering $G(B_R)=C_R$ by
\[
	C(k')=\begin{cases}
		w(k')\circ G^{-1}, &\text{ if } x\in B_{1+}\setminus C_{R_\epsilon} = G(A_{1,R_\epsilon})\\
		c(k')\circ \psi_{R_\epsilon}. &\text{ if } x \in C_{R_\epsilon}.
	\end{cases}
\]
$C(k')\circ G(k') \in W^{1,2}(\Omega_{F(k)}\cap B_1, \A_{Q}(\R^n))$ is now an admissible competitor to $u(k')$ and therefore
\begin{align*}
	(1-o(1)) \int_{B_{1+}} \abs{Dv(k')}^2 &\le \int_{\Omega_{F(k)}\cap B_1} \abs{Du(k)}^2 \le (1+o(1)) \int_{B_{1+}} \abs{DC(k')}^2\\
	&\le (1+o(1)) C \int_{A_{1,R_\epsilon}} \abs{Dw(k')}^2 + (1+o(1)) \int_{B_{1+}} \abs{Dc(k')}^2\\& \le C\left(\epsilon + C o(1)\right) + (1+o(1)) \sum_{j=1}^J \int_{B_{1+}} \abs{Dc_j}^2.
\end{align*}
Pass to the $\liminf$ and apply the lower semicontinuity ensured by the concentration compactness Lemma \ref{lem_I:A1.1} to conclude
\[
	\sum_{j=1}^J \int_{B_{1+}} \abs{Db_j}^2 \le \liminf_{k'\to \infty} (1-o(1)) \int_{B_{1+}} \abs{Dv(k')}^2 \le C\epsilon + \sum_{j=1}^J \int_{B_{1+}} \abs{Dc_j}^2.
\]
$\epsilon$ can be chosen arbitrary small and $C$ is a dimensional constant so that $b_j$ has to be Dirichlet minimizing for every $j=1, \dotsc, J$. The strong convergence in energy follows choosing $c_j=b_j$ for every $j$ in the inequality above.\\

The maps $b_j$ constructed above with the properties (i),(ii),(iii) contradict corollary \ref{cor_B:2.5}. Firstly we found due to (iii), that
\begin{align*}
	&\sum_{j=1}^J \int_{B_{1+}} \abs{Db_j}^2 = \lim_{k' \to \infty} \int_{\Omega_{F(k')}\cap B_1} \abs{Du(k')}^2 \\
	 &\ge \lim_{k' \to \infty} \left( \frac{1}{N-1} - \frac{1}{k'}\right) \int_{\Omega_{F(k)}\cap \Sp^{N-1}} \abs{D_\tau u(k')}^2 \\
	&= \lim_{k' \to \infty} \left( \frac{1}{N-1} - \frac{1}{k'}\right) \int_{S_+^{N-1}} \abs{D_\tau v(k')}^2\\
	&\ge \frac{1}{N-2} \sum_{j=1}^J \int_{\Sp^{N-1}_+} \abs{D_\tau b_j}^2.
\end{align*}
Combining this with (ii) gives, for $j=1, \dotsc, J$
\[
	\int_{B_{1+}} \abs{Db_j}^2 = \frac{1}{N-2} \int_{\Sp^{N-1}_+} \abs{D_\tau b_j}^2.
\]
Corollary \ref{cor_B:2.5} states now that $Db_j=0$ on $B_{1+}$ because $b_j\tr{\Gamma_0}=const.$ by (i). This contradicts (iii), because $1= \int_{\Omega_{F(k')} \cap B_1} \abs{Du(k')}^2$ for all $k'$.\\
This contradiction proves that the proposition must hold.
\end{proof}


%

%
%
%

%



\section{Boundary regularity in dimension $N=2$} 
\label{sec:boundary_regularity_in_dimension_n_2_}

\subsection{Global H\"older regularity} 
\label{sub:global_h"older_regularity}
In this section we will show that Theorem \ref{theo_B:2.1} extends directly to two dimensions. We can consider the two dimensional case as a special case of a certain minimizer on a three dimensional domain.

\begin{lemma}\label{lem_B:3.1} 
Let $u \in W^{1,2}(\Omega, \A_Q(\R^n))$ be a minimizer on a domain $\Omega \subset \R^N$, $N\ge 1$, then $U(x,t)=u(x)$ is an element of $W^{1,2}(\Omega \times I, \A_Q(\R^n))$ for any bounded open interval $I\subset \R$. $U$ is Dirichlet minimizing. 
\end{lemma}

\begin{proof} Assuming the contrary there exists $V \in W^{1,2}(\Omega\times I, \A_Q(\R^n))$ with $V=U$ on the boundary of $\Omega \times I$ i.e. $(x,t) \mapsto \G(U(x,t),V(x,t)) \in W^{1,2}_0(\Omega\times I)$ and
\begin{equation}\label{eq_B:3.1}
	\int_{\Omega \times I} \abs{DV}^2 < \int_{\Omega \times I} \abs{DU}^2= \abs{I} \int_{\Omega} \abs{Du}^2;
\end{equation}
the second equality actually shows that $U\in W^{1,2}(\Omega \times I, \A_Q(\R^n))$.\\
Consider the subset $J\subset I$ 
\[
	J=\{ t \in I \colon x \mapsto v_t(x)=V(x,t) \in W^{1,2}(\Omega, \A_Q(\R^n))\text{ and } v_t\tr{\partial \Omega}= u\tr{\partial \Omega} \};
\] 
then by Fubini's theorem $\abs{I\setminus J}=0$.\\

Furthermore there must be a $t \in J$ with 
\begin{equation}\label{eq_B:3.2}
	\int_{\Omega} \abs{Dv_t}^2 dx < \int_\Omega \abs{Du}^2;
\end{equation}
non existence would contradict \eqref{eq_B:3.1} because then
\begin{align*}
	\abs{I} \int_{\Omega} \abs{Du}^2 = \int_{J} \int_{\Omega} \abs{Du}^2 \, dt \le \int_J \int_\Omega \abs{Dv_t}^2 \, dx\, dt = \int_{\Omega \times I} \abs{DV}^2.
\end{align*}

$v_t$ for $t\in J$ satisfying \eqref{eq_B:3.2} is an admissible competitor to $u$, but \eqref{eq_B:3.2} violates the minimality of $u$.
\end{proof}

\begin{remark}\label{rem_B:3.1}
The converse of this lemma holds as well in the following sense, if $u(x)\in W^{1,2}(\Omega, \A_Q(\R^n))$ and $U(x,t)=u(x)$ is Dirichlet minimizing on $\Omega \times \R$ then $u$ itself is minimizing in $\Omega$, in the sense of compact perturbations:
\[
\int_{\{U \neq V \}} \abs{DU}^2 \le \int_{\{ U \neq V \} } \abs{DV}^2
\]
for all $V\in W^{1,2}(\Omega \times \R, \A_Q(\R^n))$ with $\overline{\{U \neq V\}}$ compact.\\
This had been proven in \cite{Lellis}, but for the sake of completeness we recall their proof in the appendix, Lemma \ref{lem:A6.1}. 
\end{remark}

From now on $\Omega$ denotes a $C^1$ regular domain in $\R^2$.

\begin{theorem}\label{theo_B:3.2}
For any $\frac{1}{2}<s \le 1$, there are constants $C>0$ and $\alpha_1>0$ depending on $n,Q,s$ with the property that,
\begin{itemize}
	\item[(a1)] $u \in W^{1,2}(\Omega, \A_Q(\R^n))$ Dirichlet minimizing;
	\item[(a2)] $u\tr{\partial \Omega} \in W^{s,2}(\partial \Omega, \A_Q(\R^n))$;
\end{itemize}
then the following holds
\begin{itemize}
	\item[(i)] $\abs{Du}$ is an element of the Morrey space $L^{2,2\alpha}$ for any $0< \alpha < \min\{\alpha_1, s-\frac{1}{2}\}$, more precisely the following estimate holds
	\begin{equation}\label{eq_B:3.3}
		r^{-2\alpha} \int_{B_r(x)\cap \Omega} \abs{Du}^2 \le 2^7 R_0^{-2\alpha} \int_{B_{2R_0}(x)\cap \Omega} \abs{Du}^2 + C \frac{R_0^{2s-1-2\alpha}}{2s-1-2\alpha} \llfloor u \rrfloor_{\partial \Omega}^2
	\end{equation}
	for any $r<\frac{R_0}{4}$. The positive $R_0$ depends only on $n,Q,s,\Omega$ but not on the specific $u$;\\
	\item[(ii)] $u \in C^{0,\alpha}(\overline{\Omega})$.
\end{itemize}
\end{theorem}

\begin{proof}
Set $\Omega_I= \Omega\times ]-2L, 2L[ \subset \R^3$ for some large $L>0$. The boundary portion $\partial \Omega \times ]-L,L[$ is $C^1$ regular by assumption on the regularity of $\partial \Omega$. $U(x,t)=u(x)$ is an element of $W^{1,2}(\Omega_I, \A_Q(\R^n))$ and Dirichlet minimizing as seen in lemma \ref{lem_B:3.1}.
For any $(z,t_0) \in \partial \Omega \times ]-L,L[$ and $0<r<L$ we found
\begin{align*}
	&r^{2(s-\beta)-2} \llfloor U \rrfloor^2_{s, B_r(z,t_0)\cap \partial \Omega_I} \le r^{2(s-\beta)-2} \llfloor U \rrfloor^2_{s, (B_r(z)\cap \partial \Omega)\times]t_0-r, t_0+r[}\\
	&= r^{2(s-\beta)-2} \int_{B_r(z)\cap \partial \Omega \times B_r(z)\cap \partial \Omega} \int_{t_0-r}^{t_0+r} \frac{\G(u(x),u(y))^2}{(\abs{x-y}^2 +(t_1-t_2)^2)^{\frac{2+2s}{2}}}\,dt_1dt_2 dxdy  \\
	&\le C 2r^{2(s-\beta)-1} \int_{B_r(z)\cap \partial \Omega\times B_r(z)\cap \partial \Omega} \frac{\G(u(x),u(y))^2}{\abs{x-y}^{1+2s}} dx dy \le 2C \,r^{2(s-\beta)-1} \llfloor u \rrfloor^2_{s,\partial \Omega}.
\end{align*}
(We have applied above the following auxiliary calculation. Let $\alpha>0$ and $J=[a, a+\delta]$.  After the change of variables $t_1=a+rx$, $t_2=a+ry$, we have
\begin{align*}
	&\int_{J\times J} \frac{1}{(r^2+(t_1-t_2)^2)^{\frac{\alpha+1}{2}}} \, dt_1dt_2 = 2 r^{1-\alpha} \int_{\substack{[0,\frac{\delta}{r}]\times[0,\frac{\delta}{r}]\\ x \ge y}} \frac{1}{(1+ (x-y)^2)^{\frac{\alpha+1}{2}}}\, dx dy\\
	&= 2 r^{1-\alpha} \int_{0}^\frac{\delta}{r} \int_{0}^{\frac{\delta}{r}-y} \frac{1}{(1+z^2)^{\frac{\alpha+1}{2}}} \,dz dy \le 2r^{-\alpha} \delta \int_{0}^\infty \frac{1}{(1+z^2)^{\frac{\alpha+1}{2}}} \\
	&= C \abs{J} r^{-\alpha}.
\end{align*}
The dimensional constant $C=2 \int_{0}^\infty \frac{1}{(1+z^2)^{\frac{\alpha+1}{2}}} \le \frac{\alpha+1}{\alpha}$ is therefore finite.)\\
Combining all obtained estimates we found that $U$ satisfies the assumption of theorem \ref{theo_B:2.1} with $\beta= s-\frac{1}{2}$ and $M_U = \llfloor u \rrfloor_{s, \partial \Omega}$ in (a2).

Apply Theorem \ref{theo_B:2.1}, in particular \eqref{eq_B:2.1}, to $U$ on a point $(x,0) \in \Omega\times]-L,L[$ with $r<\frac{R_0}{4}<L$. This gives the desired \eqref{eq_B:3.3}, because
\begin{align*}
	&r^{-2\alpha} \int_{B_r(x)\cap\Omega} \abs{Du}^2 = \frac{r^{-2\alpha}}{2r} \int_{-r}^r \int_{B_r(x)\cap \Omega} \abs{DU}^2 \le 2^2 (2r)^{-1-2\alpha} \int_{B_{2r}((x,0))\cap \Omega_I} \abs{DU}^2\\
	&\le 2^5 \left( R_0^{-1-2\alpha}\int_{B_{2R_0}((x,0))\cap \Omega_I} \abs{DU}^2 + C\frac{R_0^{2(\beta-\alpha)}}{\beta-\alpha} M_U^2 \right)\\
	&\le 2^7 R_0^{-2\alpha} \int_{B_{2R_0}(x)\cap \Omega} \abs{Du}^2 +  C\frac{R_0^{2s-1-2\alpha}}{2s-1-2\alpha} \llfloor u \rrfloor^2_{s,\partial \Omega}.
\end{align*}
(ii) i.e. $u \in C^{0,\alpha}(\overline{\Omega})$ now follows as outlined in the proof to theorem \ref{theo_B:2.1}.
\end{proof}


\subsection{Continuity up to boundary} 
\label{sub:continuity_up_to_boundary}
That continuity extends up to the boundary for 2-dimensional ball has been proven by W.Zhu in \cite{Zhu}. His idea is based on the Courant-Lebesgue lemma and can be modified to work on Lipschitz regular domains as well. We will give here a different proof, that on a first glimpse doesn't seem to be so restricted to the 2-dimensional setting as it is for Zhu's proof due to the Courant-Lebesgue lemma. Our proof uses an interplay of classical trace estimates and energy decay. We shortly recall the classical trace estimates and their proof. The proof here is taken from \cite[Lemma 13.5]{Tartar}. As introduced in the general assumptions, section \ref{sec:general_assumptions_and_notation}, we use the notation $\Omega_F=\{(x',x_N) \colon x_N > F(x') \}$ for $F: \R^{N-1} \to \R$.

\begin{lemma}\label{lem_B:3.3}
For $F$ Lipschitz continuous and $1<p<\infty$, one has
\begin{equation}\label{eq_B:3.4}
\norm{\frac{f(x',x_N)- f\tr{\partial\Omega_F}(x')}{x_N- F(x')}}_{L^p(\tilde{\Omega}} \leq \frac{p}{p-1} \norm{\frac{\partial f}{\partial x_N}}_{L^p(\tilde{\Omega})} \quad \forall f\in W^{1,p}(\Omega_F,\R);
\end{equation}
and any subset $\tilde{\Omega}\subset \Omega_F$ of the following type:
\[\tilde{\Omega}=\{ (x',x_N)\colon x' \in \Omega', F(x')<x_N< G(x')\}\]
$\tilde{\Omega} \subset \R^{N-1}$ and $G \ge F$ continuous.\\
Equivalently one has 
\begin{equation}\label{eq_B:3.5}
\norm{\frac{\G(u(x',x_N), u\tr{\partial\Omega_F}(x'))}{x_N- F(x')}}_{L^p(\tilde{\Omega})} \leq \frac{p}{p-1} \norm{\abs{D_N u}}_{L^p(\tilde{\Omega})} \quad \forall u \in W^{1,p}(\Omega_F, \A_Q(\R^n)).
\end{equation}
\end{lemma}
\begin{proof}
For $p>1$ Hardy's inequality, compare for instance with \cite[Lemma 13.4]{Tartar}, states that, if $h \in L^p(\R_+)$, $g(t):= \frac{1}{t} \int_0^t h(s) ds \in L^p(\R_+)$ satisfies 
\begin{equation}\label{eq_B:3.6}
	\norm{g}_p \le \frac{p}{p-1} \norm{f}_p.
\end{equation}
For $f \in C_c^1(\overline{\Omega_F})$ set
\[
	h(t):= \textbf{1}_{[0,G(x')-F(x')]}(t)\, \frac{\partial f}{\partial x_N} (x', F(x')+t).
\]
Apply Hardy's inequality to it and observe that for $0<t<G(x')-F(x')$ and $t=x_N-F(x')$
\[
	g(t)= \frac{f(x', F(x')+t)- f(x',F(x'))}{t}= \frac{f(x',x_N)-f\tr{\partial \Omega_F}(x')}{x_N-F(x')}.
\]
Hence take the power $p$ and integrate in $x' \in \Omega'$ to conclude \eqref{eq_B:3.5}. By a density argument the inequality extends to all of $W^{1,p}(\Omega_F)$.\\
For a Lipschitz continuous $u \in W^{1,p}(\Omega_F)$, we have $u\tr{\partial\Omega_F}(x')= u(x',F(x'))$. $k(t):= \G(u(x'), F(x')+t)$ is Lipschitz continuous in $t$. Furthermore $k'(t) \le \abs{D_Nu}(x',F(x')+t)$ for a.e. $x'$. Apply Hardy's inequality this time to $h(t)= \textbf{1}_{[0,G(x')-F(x')]}(t)\, k'(t)$, take the power $p$ and integrate in $x' \in \Omega'$. This shows \eqref{eq_B:3.5} under the additional assumption that $u$ is Lipschitz. It extends by density to all of $W^{1,p}(\Omega_F)$. 
\end{proof}

\begin{proposition}\label{prop_B:3.4}
Given a Dirichlet minimizer $u \in W^{1,2}(\Omega, \A_Q(\R^n))$ on a Lipschitz regular domain $\Omega \subset \R^N$ that satisfies
\begin{itemize}
	\item[(a1)] $u\tr{\partial \Omega}$ is continuous;\\
	\item[(a2)] $N=2$ or \begin{equation}\label{eq_B:3.7}
		r^{2-N} \int_{B_r(z)\cap \Omega} \abs{Du}^2 \to 0 \text{ as } r \to 0 \text{ uniformly for all } z\in \partial \Omega; 
	\end{equation}
\end{itemize}	
then $u$ is continuous on $\overline{\Omega}$.
\end{proposition}

\begin{proof}
Observe that in case of $N=2$, $r^{2-N} \int_{B_r(z)\cap \Omega} \abs{Du}^2=\int_{B_r(z)\cap \Omega} \abs{Du}^2 \to 0$ uniformly due to the absolute continuity of the integral and $\abs{Du}^2 \in L^1(\Omega)$. Hence it is sufficient to prove the proposition under the assumption that \eqref{eq_B:3.7} holds. $u$ is H\"older continuous in the interior (theorem \ref{theo_I:1.103}) and so it remains to check that continuity extends up to the boundary. This is a local question so we assume that $\Omega=\Omega_F$ for some Lipschitz continuous $F$, with Lipschitz norm $Lip(F)<L$. Furthermore let $z_0=(z', F(z')) \in \partial \Omega_F$  be fixed. \\
Consider a generic sequence $x_k=(x_k',x_{N,k})$ converging to $z_0$ from the interior. Set $r_k = x_{N,k}-F(x_k')>0$ and $\epsilon= \frac{1}{2 \sqrt{1+L^2}}$. Then $B_{2\epsilon r_k}(x_k) \subset \Omega_F$ for all $k$ and 
\begin{equation}\label{eq_B:3.8}
	r_k^2\le 2 (x_{N,k}- z_N)^2 + 2 (F(z')-F(x'_k))^2\le \frac{1}{2\epsilon^2} \abs{x_k -z_0}^2.
\end{equation}
To show continuity we have to check that $\G(u(x_k),u\tr{\partial\Omega_F}(z_0))$ is of order $o(1)$. The triangle inequality and convexity gives
\begin{align*}
	&\frac{1}{3}\G(u(x_k), u\tr{\partial\Omega_F}(z_0))^2 \le \G(u(x_k), u(x))^2  \\&+ \G(u(x), u\tr{\partial\Omega_F}(x'))^2 + \G(u\tr{\partial\Omega_F}(x'),u\tr{\partial\Omega_F}(z_0))^2.
\end{align*}
Integration in $x \in B_{\epsilon r_k}(x_k)$ gives
\begin{align*}
	&\frac{1}{3}\G(u(x_k), u\tr{\partial\Omega_F}(z_0))^2 \le \fint_{B_{\epsilon r_k}(x_k)}\G(u(x_k), u(x))^2 \\
	& + \fint_{B_{\epsilon r_k}(x_k)} \G(u(x), u\tr{\partial\Omega_F}(x'))^2 + \fint_{B_{\epsilon r_k}(x_k)} \G(u\tr{\partial\Omega_F}(x'),u\tr{\partial\Omega_F}(z_0))^2.
\end{align*}
It is sufficient to check that all integrals are of order $o(1)$.
\[
 \fint_{B_{\epsilon r_k}(x_k)} \G(u\tr{\partial\Omega_F}(x'),u\tr{\partial\Omega_F}(z_0))^2 \le \sup_{x \in B_{\abs{x_k-z_0}}(z_0)} \G(u\tr{\partial\Omega_F}(x'),u\tr{\partial\Omega_F}(z_0))^2=o(1)
\]
where we used \eqref{eq_B:3.8} and assumption (a1).\\
For a fixed $k$ set $\tilde{\Omega}=\{ (x',x_N)\colon x' \in \Omega', F(x')< x_N < G(x')\}$ with $\Omega' = B_{\epsilon r_k}(x_k')\subset \R^{N-1}$, $G(x')=x_{N,k}+\epsilon r_k$. The trace estimate, Lemma \ref{lem_B:3.3} states
\[
 \frac{1}{r_k^2} \int_{\tilde{\Omega}} \G(u(x), u\tr{\partial\Omega_F}(x'))^2 \le 4 \int_{\tilde{\Omega}} \frac{ \G(u(x), u\tr{\partial\Omega_F}(x') )^2 }{\abs{x_N-F(x')}^2} \le 16 \int_{\tilde{\Omega}} \abs{Du}^2;
\]
where we used $\frac{1}{r_k} \le \frac{2}{x_N- F(x')}$ because of $x_N-F(x')= x_N-x_{N,k} + r_k + F(x_k')- F(x') \le \epsilon r_k + r_k + L \epsilon r_k \le 2 r_k$. We may combine it with $B_{\epsilon r_k}(x_k) \subset \tilde{\Omega} \subset B_{2r_k}(z_k) \cap \Omega_F$ and assumption (a2) to deduce
\[
	\fint_{B_{\epsilon r_k}(x_k)} \G(u(x), u\tr{\partial\Omega_F}(x'))^2 \le \frac{16}{\omega_N \epsilon^N} r_k^{2-N} \int_{B_{2r_k}(z_k)\cap \Omega_F} \abs{Du}^2=o(1).
\]
Finally the first integral is estimated using the internal H\"older continuity result: since $B_{2\epsilon r_k}(x_k) \subset \Omega_F$ for positive $C, \beta$ 
\[
	\G( u(x), u(x_k))^2 \le C \left( \frac{\abs{x-x_k}}{\epsilon r_k}\right)^{2\beta} \, (\epsilon r_k)^{2-N} \int_{B_{2\epsilon r_k}(x_k)} \abs{Du}^2 \text{ for all } x\in B_{\epsilon r_k}(x_k).
\]
Integration in $x$ and $B_{2\epsilon r_k}(x_k)\subset B_{2r_k}(z_k)$ gives
\[
	\fint_{B_{\epsilon r_k}(x_k)} \G(u(x),u(x_k))^2 \le \frac{C}{(\epsilon r_k)^{N-2}} \int_{B_{2\epsilon r_k}(x_k)} \abs{Du}^2 \le \frac{C}{\epsilon^{N-2}} r_k^{2-N} \int_{B_{2r_k}(z_k)} \abs{Du}^2;
\]
that is of order $o(1)$ by assumption (a2).
\end{proof}

\begin{remark}\label{rem_B:3.2}
	$u \in W^{1,2}(\Omega,\A_Q(\R^n))$ implies that $u\tr{\partial \Omega} \in W^{\frac{1}{2},2}(\partial \Omega, \A_Q(\R^n))$ but this is just not sufficient to ensure continuity. $W^{\frac{1}{2},2}(\R)=H^{\frac{1}{2}}(\R)$ does not embed into $L^\infty(\R)$ but only the slightly smaller space $(H^1(\R),L^2(\R))_{\frac{1}{2},1}$ embeds into $C^0(\R)$, compare for instance \cite[chapter 25]{Tartar}.
\end{remark}

\subsection{Partial improvement of the H\"older exponent} 
\label{sub:partial_improvement_of_the_h"older_exponent}
In the introduction we mentioned already that it would be desirable to extend the optimal H\"older exponent $\frac{1}{Q}$ in the interior up to the boundary. We want to present in this subsection a partial improvement of theorem \ref{theo_B:3.2}:\\
Let $\Omega \subset \R^2$ be a $C^1$-regular domain the following holds:\\
$u\in W^{1,2}(\Omega, \A_Q(\R^n))$ Dirichlet minimizing with $u\tr{\partial \Omega} \in C^{0, \beta}(\partial \Omega)$ for some $\beta>\frac{1}{2}$ then $u \in C^{0, \alpha}(K)$, $\alpha = \frac{1}{Q}$ for $Q>2$ and any $0<\alpha < \frac{1}{2}$. Furthermore $K\subset \overline{\Omega}$ closed and touches $\partial \Omega$ in at most 1 point $z$ non-tangential.\\
To every closed set $K$ of this type there is a cone $C_{z,\theta}=\{ x \in \R^2 \colon \abs{x}\cos(\theta) < - \langle \nu_{\partial \Omega}(z), x \rangle \}$, $0<\theta < \frac{\pi}{2}$ ( $\nu_{\partial \Omega}(z)$ denotes the outward pointing normal to $\partial \Omega$ at $z$ ) and a radius $0<R$ s.t. $K\cap \overline{B_R(z)} \subset \overline{ C_{z,\theta}\cap B_R(z)}$. Shrinking $R>0$ if necessary we may even assume w.l.o.g. that $C_{z,\theta}\cap B_R(z) \subset \Omega$.  This is sketched in the figure.
\begin{wrapfigure}{r}{0.5\textwidth}\label{fig}
    \includegraphics[width=0.5\textwidth]{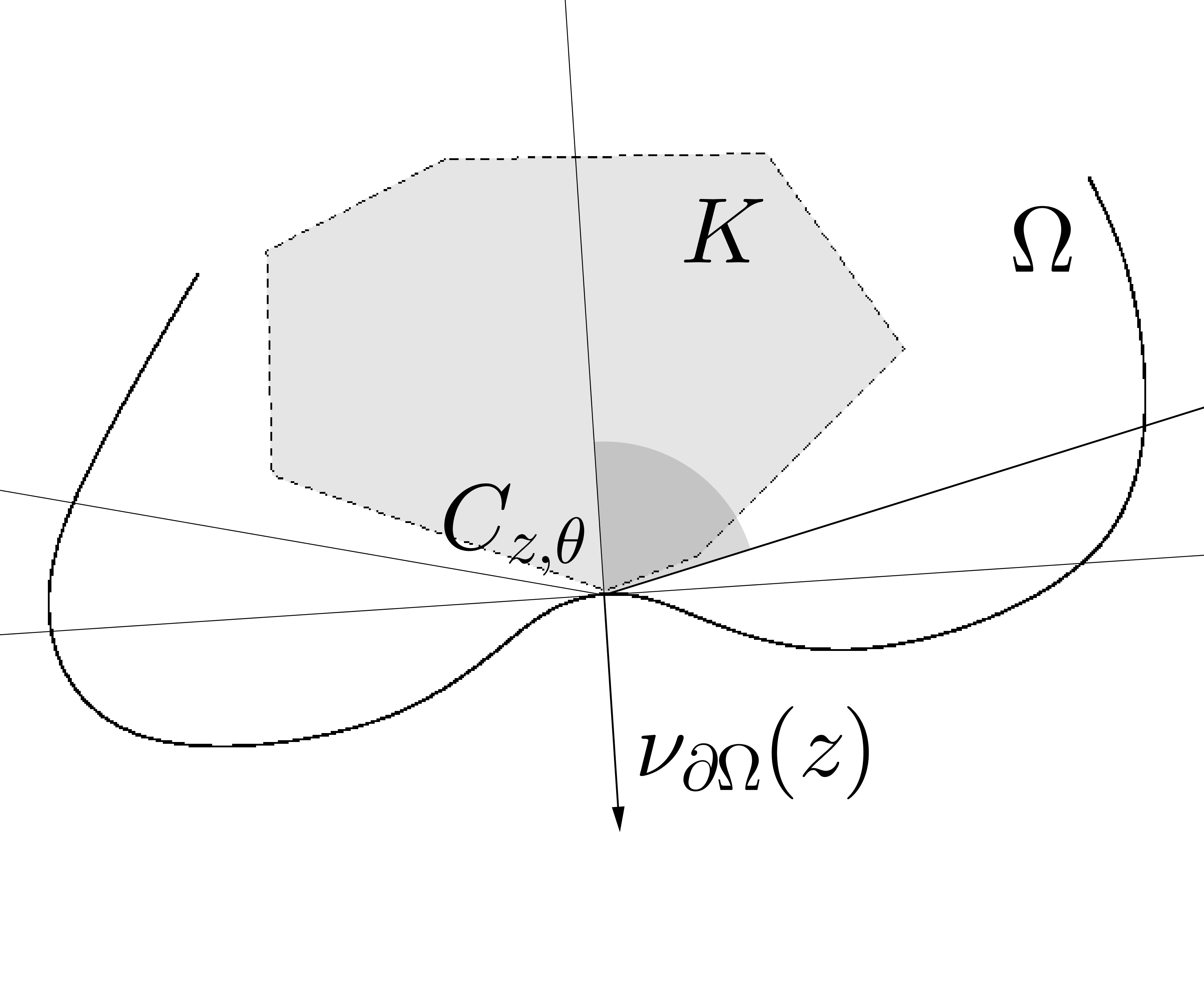}
\end{wrapfigure}
 $K\setminus B_R(z)$ is a compact subset of $\Omega$ hence the interior regularity theory holds. It remains to prove regularity for conical subsets $C_{z,\theta}\cap B_R(z)$. The precise statement is: 
\begin{corollary}\label{cor_B:5.8}
Let $\frac{1}{2}< s \le 1$ and $C_\theta=\{x =(x_1,x_2)\colon \abs{x} \cos(\theta)\le x_2 \}$ with $0<\theta <\frac{\pi}{2}$ (a cone). Under the assumptions
\begin{itemize}
	\item[(a1)] $u \in W^{1,2}(\Omega_F\cap B_1, \A_Q(\R^n))$ Dirichlet minimizing
	\item[(a2)] $ u\tr{\partial\Omega_F} \in W^{s,2}(\Gamma_F, \A_Q(\R^n))$ and for some $0<\gamma$ there is a constant $M_u>0$ s.t. 
	\begin{equation*}
		r^{2(s-\gamma)-1} \llfloor u \rrfloor^2_{s, B_r\cap \Gamma_F} \le M_u^2,
	\end{equation*}
\end{itemize}
then there exists $0<R<1$ depending on $u(0)$ and $\theta$ s.t., for any $\alpha<\min\{\gamma, \frac{1}{2}\}$ and $\alpha \le \frac{1}{Q}$ the following holds
\begin{itemize}
	\item[(i)] $\abs{Du}$ is an element of the Morrey space $L^{2,2\alpha}(\Omega_F\cap B_{\frac{R}{2}}\cap C_\theta)$, more precisely 
	\begin{equation}\label{eq_B:3.12}
		r^{-2\alpha} \int_{B_r(x)\cap \Omega_F} \abs{Du}^2 \le \frac{4}{\delta^{2\alpha}} \left(\int_{B_R\cap \Omega_F} \abs{Du}^2 + \frac{CR^{2(\gamma-\alpha)}}{\gamma-\alpha} M_u^2\right)
	\end{equation}
	where $\delta= \cos(\theta)- \cos(\frac{2\theta+ \pi}{4})$;
	\item[(ii)] $u\in C^{0,\alpha}(\Omega_F\cap B_{\frac{R}{2}}\cap C_\theta)$. 
\end{itemize}
\end{corollary}
Concerning the optimality of the achieved H\"older exponent and assumption (a2) consider the following:
\begin{remark}\label{rem_B:5.3}
(a2) is obviously always satisfied for $\gamma=s-\frac{1}{2}$.\\
(a2) is satisfied for $\gamma>\frac{1}{2}$ and any $s <\gamma$ if $u\tr{\Gamma_F} \in C^{0,\gamma}(\Gamma_F)$ as we have seen in lemma \ref{lem_B:2.101}. Furthermore this implies that
\[
	u \in C^{0,\alpha}(\overline{\Omega_F\cap B_R\cap C_\theta}) \text{ with $\alpha=\frac{1}{Q}$ for $Q>2$ and any $\alpha< \frac{1}{2}$ for $Q=2$};
\]
i.e. the optimal exponent extends on cones up to the boundary.
\end{remark}
The proof of the corollary follows similar lines as in the higer dimeinsional case. We will prove an improve estimate in the spirit of proposition \ref{prop_B:2.3}, that will lead eventually to corollary \ref{cor_B:5.8}. Before we present this final argument we prove the preliminary lemmas.
As in the previous sections: $B_{1+} = B_1 \cap \{ x_2>0\}$, $\Sp^1=\partial B_1, \Sp^{1}_{+}=\Sp^1 \cap \{x_2>0\}$, and $\Gamma_0= B_1\cap \{ x_2 = 0 \}$.

\begin{lemma}\label{lem_B:3.5}
	Let $\frac{1}{2}<s\le 1$ be given, then there is a constant $C=C(s)$ s.t. any single valued harmonic function $f \in W^{1,2}(B_{1+})$ satisfies
	\begin{equation}\label{eq_B:3.9}
		\int_{B_{1+}} \abs{Df}^2 \le (1+\epsilon) \int_{\Sp^1_+} \abs{D_\tau f}^2 + \frac{C}{\epsilon} \int_{\Gamma_0} \llfloor f \rrfloor_{s,\Gamma_0}^2 \quad \forall \epsilon >0.
	\end{equation}
\end{lemma}
\begin{proof}
In a first step we show the existence of $C=C(s)$ s.t. any classical single-valued harmonic $h \in W^{1,2}(B_{1+})$ satisfies
\begin{equation}\label{eq_B:3.10}
	\int_{B_{1+}} \abs{Dh}^2 \le C \left( \int_{\Sp^1_+} \abs{D_\tau h}^2 + \llfloor h \rrfloor_{s, \Gamma_0}^2 \right).
\end{equation}
If $h \notin W^{s,2}(\Gamma_0)$ the RHS is $+\infty$ so there is nothing to check. $G: \overline{B_1} \to \overline{B_{1+}}$ denotes the bilipschitz map of Lemma \ref{lem:A4.1}. Let $\sum_{k \in \Z} a_k e^{ik\theta}$ be the Fourier series of $h\circ G\tr{\Sp^1}= h\tr{\Sp^1}\circ G$. Its harmonic extension is then
\[
	\tilde{h}(r\, e^{i\theta}) = \sum_{k \in \Z} a_k r^k e^{ik\theta}.
\]
$h$ is harmonic, hence minimizing the Dirichlet energy, and $\tilde{h}\circ G^{-1}$ is an admissible competitor, so that 
\[
	\int_{B_{1+}} \abs{Dh}^2 \le \int_{B_{1+}} \abs{D(\tilde{h}\circ G^{-1})}^2 \le C \int_{B_1}\abs{D\tilde{h}}^2 = C 2\pi \sum_{k \in \Z} \abs{k} \abs{a_k}^2.
\]
For $s=1$ we estimate (the constant $C$ depends only on the Lipschitz norms of $G, G^{-1}$)
\begin{align*}
	2\pi \sum_{k \in \Z} \abs{k} \abs{a_k}^2 &\le 2\pi \sum_{k\in \Z} k^2 \abs{a_k}^2 = \int_{S^1_+} \abs{D_\tau \tilde{h}}^2 + \int_{S^1_-} \abs{D_\tau \tilde{h}}^2 \\&\le C \left(\int_{S^1_+} \abs{D_\tau h}^2 + \int_{\Gamma_0} \abs{D_\tau h}^2\right).
\end{align*}
for $\frac{1}{2} <s<1$:\\
(A short auxiliary argument: Lemma \ref{lem:A2.6} implies the equivalence of the norms $\abs{b_0}+\sum_{k\in \Z} \abs{k}^{2s}\abs{b_k}^2$ and $\norm{f}^2_{L^2(\Sp^1)}+ \llfloor f \rrfloor^2_{s,\Sp^1}$ for a function $f(\theta)=\sum_{k\in \Z} b_k e^{ik\theta}$. In the case of $\Sp^1$ this follows more directly. $f(\theta+\tau)-f(\theta)=\sum_{k\in \Z} (e^{ik\tau}-1)a_k e^{ik\theta}$ and therefore
\[
	\int_{0}^{2\pi} \abs{f(\theta+\tau)-f(\theta)}^2 \;d\theta= \sum_{k \in \Z} 4 \sin^2(\frac{k}{2}\tau) \abs{a_k}^2.
\] 
This implies 
\begin{align*}
	&\int_{[0,2\pi]^2} \frac{\abs{f(\theta)-f(\varphi)}^2}{\abs{\theta-\varphi}^{1+2s}} \, d\theta d\varphi = \int_{0}^{2\pi} \frac{1}{\tau^{1+2s}} \int_{0}^{2\pi} \abs{f(\theta+\tau)-f(\theta)}^2 \,d\theta d\tau\\
	 &= \sum_{k \in \Z} \abs{a_k}^2 \left(4 \int_{0}^{2\pi} \frac{\sin^2(\frac{k}{2}\tau)}{\tau^{1+2s}}\,d\tau\right)= \sum_{k \in \Z} c_k \abs{k}^{2s} \abs{a_k}^2;
\end{align*}
where
\[
	\int_{0}^{2\pi} \frac{\sin^2(\frac{k}{2}\tau)}{\tau^{1+2s}} \,d\tau = \abs{k}^{2s} \, 4^{1-s} \int_0^{k\pi} \frac{\sin^2(\tau)}{\tau^{1+2s}} \,d\tau = \abs{k}^{2s} c_k
\]
and $0<c_1\le c_k \le c_\infty <\infty$.)\\
Firstly the auxiliary argument gives
\[
	2\pi \sum_{k \in \Z} \abs{k} \abs{a_k}^2 \le 2\pi \sum_{k \in \Z} \abs{k}^2s \abs{a_k}^2 \le C \llfloor \tilde{h} \rrfloor^2_{s, \Sp^1};
\]
secondly Corollary  \ref{cor:A1.8} gives
\[
	\llfloor \tilde{h} \rrfloor_{s,S^1}^2 \le C \left( \llfloor \tilde{h}\rrfloor_{s, \Sp^1\cap\{x_2 > \frac{1}{5}\}}^2 + \llfloor  \tilde{h}\rrfloor_{s, \Sp^1\cap\{x_2 < \frac{1}{5}\}}^2\right);
\]
thirdly $G$ is Lipschitz continuous and $G(\Sp^1\cap\{x_2 > \frac{1}{5}\})=\Sp^1_+, G(\Sp^1\cap\{x_2 < \frac{1}{5}\})=\Gamma_0$ so that
\[
\llfloor \tilde{h}\rrfloor_{s, \Sp^1\cap\{x_2 > \frac{1}{5}\}}^2 + \llfloor  \tilde{h}\rrfloor_{s, \Sp^1\cap\{x_2 < \frac{1}{5}\}}^2 \le C \left(\llfloor h\rrfloor_{s, \Sp^1_+}^2 + \llfloor  h\rrfloor_{s, \Gamma_0}^2\right);
\]
finally combining these with the interpolation property $\llfloor \cdot \rrfloor_{s, \Sp^1_+} \le C \llfloor \cdot \rrfloor_{s, \Sp^1_+}$ we estimate
\[
	2\pi \sum_{k \in \Z} \abs{k} \abs{a_k}^2 \le C \left(\int_{\Sp^1_+} \abs{D_\tau h}^2  + \llfloor  h\rrfloor_{s, \Gamma_0}^2\right).
\]
Hence \eqref{eq_B:3.10} holds.\\

Now we are able to improve \eqref{eq_B:3.10} to \eqref{eq_B:3.9}. Let $f$ be the harmonic function as assumed. We may assume $f \in W^{s,2}(\Gamma_0)$ otherwise the RHS is $+\infty$ and \eqref{eq_B:3.9} holds trivially. Define the linear function 
\[
	l(x_1,x_2)= \frac{f(1,0)-f(-1,0)}{2} x_1 + \frac{f(1,0)+f(-1,0)}{2}.
\]
The same calculations as in lemma \ref{lem_B:2.101} give a constant $C=C(s)$ with
\[
	\llfloor l \rrfloor_{s,\Gamma_0}^2 \le C \norm{ \grad{l}}_\infty = C \abs{f(1,0)-f(-1,0)}.
\]
We achieved that $f(1,0)-l(1,0)=0=f(-1,0)-l(-1,0)$ and hence the glueing lemma \ref{lem:A1.7} provides that 
\[
	\tilde{h}(x)= \begin{cases}
		0, &\text{ if } x\in \Sp^1_+\\
		f(x)-l(x), &\text{ if } x\in \Gamma_0
	\end{cases}
\]
is an element of $W^{s,2}(\Sp^1_+ \cup \Gamma_0)$. Hence there is a unique harmonic $h \in W^{1,2}(B_{1+})$ with $h\tr{\Sp^1_+ \cup \Gamma_0} = \tilde{h}$. $g=f-(h+l)$ is harmonic in $B_{1+}$ and satisfies $g(x)=0$ on $\Gamma_0$. The antisymmetric reflexion 
\[
	\tilde{g}(x_1,x_2)=\begin{cases}
		g(x_1,x_2), &\text{ if } x_2 \ge 0\\
		-g(x_1,-x_2), &\text{ if } x_2 \le 0
	\end{cases}
\]
is by means of the Schwarz reflexion principle harmonic in $B_1$ with
\[
	2 \int_{B_{1+}} \abs{Dg}^2 = \int_{B_1} \abs{D\tilde{g}}^2 \le \int_{S^1} \abs{D_\tau \tilde{g}}^2 = 2 \int_{S^1_+} \abs{D_\tau g}^2.
\]
Young's inequality for $2 \langle D_\tau f, D_\tau l\rangle \le \epsilon \abs{D_\tau f}^2 + \frac{1}{\epsilon} \norm{\grad{l}}_\infty^2$ gives
\begin{align*}
	\int_{S^1_+} \abs{D_\tau g}^2 &\le (1+\epsilon) \int_{S^1_+} \abs{D_\tau f}^2 + (1+\frac{1}{\epsilon})\pi \norm{\grad{l}}^2_\infty\\
	&(1+\epsilon) \int_{S^1_+} \abs{D_\tau f}^2 + \frac{C}{\epsilon} \llfloor f \rrfloor_{s,\Gamma_0}^2
\end{align*}
where we used $\grad{l}=\frac{f(1,0)-f(-1,0)}{2}$ and $W^{s,2}(\Gamma_0) \subset C^{0,s-\frac{1}{2}}(\Gamma_0)$.
Young's inequality for $2 \langle D_if, D_i(h+l)\rangle \ge - \epsilon \abs{D_if}^2 - \frac{1}{\epsilon} \abs{D_i (h+l)}^2$ gives 
\[
	\int_{B_{1+}} \abs{Dg}^2 \ge (1-\epsilon) \int_{B_{1+}} \abs{Df}^2 - \frac{1}{\epsilon} \int_{B_{1+}}\abs{D(h+l)}^2;
\]
applying \eqref{eq_B:3.10} we may conclude
\begin{align*}
	\int_{B_{1+}} \abs{D(h+l)}^2 \le C \left( \int_{S^1_+} \abs{D_\tau (h+l)}^2 + \llfloor h+l \rrfloor_{s,\Gamma_0}^2\right)\\
	\le C \left( \pi \norm{\grad{l}}_\infty^2 + \llfloor f \rrfloor _{s, \Gamma_0}^2 \right) \le C \llfloor f \rrfloor_{s, \Gamma_0}^2.
\end{align*}
\end{proof}

Lemma \ref{lem_B:3.5} behaves well under perturbations of $B_{1+}$, as made quantitive in the following corollary.

\begin{corollary}\label{cor_B:3.6}
Let $\frac{1}{2}<s \le 1$. There is a constant $C>0$ s.t. to any $\epsilon>0$ there is $\epsilon_F=\epsilon_F(\epsilon)>0$ s.t. any single valued harmonic function $f \in W^{1,2}(\Omega_F\cap B_1)$ satisfies
\[
	\int_{\Omega_F \cap B_1} \abs{Df}^2 \le (1+\epsilon) \int_{\Omega_F \cap S^1} \abs{D_\tau f}^2 + \frac{C}{\epsilon} \llfloor f \rrfloor_{s, \Gamma_F}^2.
\]
\end{corollary}
\begin{proof}
This follows as a perturbation of the previous lemma making use of the bilipschitz equivalence of $\Omega_F\cap B_1$ and $B_{1+}$ i.e. fix
\[
	G_F: \overline{B_{1+}} \to \overline{\Omega_F \cap B_1}
\]
as given by lemma \ref{lem:A4.2}. Hence $\norm{DG_F - \textbf{1}}_\infty, \norm{DG_{F}^{-1}- \textbf{1}}_\infty < 10 \norm{\grad{F}}_\infty \le 10 \epsilon_F$. Let $f$ as assumed with finite RHS, otherwise there is nothing to prove. $f\circ G_F \in W^{1,2}(B_{1+})$ hence there is an unique harmonic $\tilde{f} \in W^{1,2}(B_{1+})$ with $\tilde{f}\tr{\Sp^1\cup\Gamma_0}=  f\circ G_F\tr{\Sp^1\cup\Gamma_0}$. $f,\tilde{f}$ are Dirichlet minimizer on their domains so that 
\[
	\int_{\Omega_F\cap B_1} \abs{Df}^2 \le \int_{\Omega_F\cap B_1} \abs{D(\tilde{f}\circ G_F^{-1})}^2 \le (1+10\epsilon_F)^4 \int_{B_{1+}} \abs{D\tilde{f}}^2.
\]
The previous lemma showed that, for some constant $C>0$,
\begin{align*}
	\int_{B_{1+}} \abs{D\tilde{f}}^2 &\le (1+\epsilon_1) \int_{S^1_+} \abs{D_\tau \tilde{f}}^2 + \frac{C}{\epsilon_1} \llfloor \tilde{f} \rrfloor_{s,\Gamma_0}^2\\
	&\le (1+\epsilon_1)(1+10 \epsilon_F)^3 \int_{S^1\cap \Omega_F} \abs{D_\tau f}^2 + \frac{C}{\epsilon_1} (1+10 \epsilon_F)^5 \llfloor f \rrfloor_{s, \Gamma_F}^2.
\end{align*}
We conclude choosing $\epsilon_1= \frac{\epsilon}{2}$ and then $\epsilon_F>0$ sufficient small for $(1+\frac{\epsilon}{2})(1+ 10 \epsilon_F)^7 \le 1+\epsilon$.
\end{proof}

We can use the obtained results to get an estimate for Dirichlet minimizers in the spirit of proposition \ref{prop_B:2.3}.

\begin{lemma}\label{lem_B:3.7}
	For $\frac{1}{2}<s\le 1$ and $\epsilon>0$, there is a constant $C=C(s)>0$ with the property that if (A2) holds with $\epsilon_F=\epsilon_F(\epsilon)>0$ then
	\[
		\int_{B_r\cap \Omega_F} \abs{Du}^2 \le (1+\epsilon ) \int_{\partial B_r \cap \Omega_F} \abs{D_\tau u}^2 + \frac{C}{\epsilon} r^{2s-1} \llfloor u \rrfloor_{s, B_r\cap \Omega_F}^2 \forall 0<r<R_0
	\]
	for any Dirichlet minimizing $u \in W^{1,2}(\Omega_F \cap B_1, \A_Q(\R^n))$ and $R_0=R_0(u(0))>0$.
\end{lemma}
\begin{proof}
As usual we may assume that the RHS is finite. Let $\epsilon_F>0$ be the constant of the previous corollary \ref{cor_B:3.6} and $\norm{\grad{F}}_{\infty, B_1}< \epsilon_F$.\\
Suppose $s(u(0))=0$ i.e. $u(0)=Q\llbracket p \rrbracket$ for some $p \in \R^n$. Since we assumed the RHS is finite $u \in W^{1,2}(\partial B_r\cap \Omega_F, \A_Q(\R^n))$. Fix for such a radius $t_-<0<t_+$ and $- \frac{\pi}{2} < \theta_+ < \theta_- < \frac{3\pi}{2}$ s.t.
\[
	\partial B_r \cap \Omega_F=\{ x_+= (rt_+, F(rt_+))= re^{i\theta_+}, x_-= (rt_-, F(rt_-))= re^{i\theta_-} \}.
\]
There is $b=(b_1, \dotsc , b_Q) \in W^{1,2}([\theta_+, \theta_-], \R^{nQ})$ s.t. $[b(\theta)]=u_{0,r}(e^{i\theta})=u(re^{-\theta})$ for $\theta_+ \le \theta \le \theta_-$ due to the 1-dim. $W^{1,2}$-selection criterion \cite[proposition 1.2]{Lellis}. There are $a(t)=(a_1, \dotsc, a_Q) \in W^{s',2}([0,t_+], \R^{nQ})$ and $b(t)= (b_1, \dotsc, b_Q) \in W^{s',2}([t_-, 0], \R^{nQ})$ for any $s'<s$ with $[a(t)]=u(rt, F(rt))$, $[b(t)]=u(rt, F(rt))$ respectively due to the $W^{s,2}$-selection, lemma \ref{lem:A2.401}. Permuting $a$ and $c$ if necessary we may assume that $a(t_+)=b(\theta_+)$, $c(t_-)=b(\theta_-)$. We may define
\[
	g(x)= \begin{cases}
		a(x_1), &\text{ if } rx \in B_r\cap \Gamma_F, x_1 \ge 0\\
		b(\theta), &\text{ if } rx=re^{i\theta} \in \partial B_r \cap \Omega_F\\
		c(x_1), &\text{ if } rx \in B_r\cap \Gamma_F, x_1 \le 0.
	\end{cases}
\]
$g=(g_1, \dotsc, g_Q) \in W^{s',2}(\partial (B_1, (\Omega_F)_{0,r}), \R^{nQ})$ as a consequence of the glueing lemma \ref{lem:A1.7}. $[g(x)]=\sum_{i=1}^Q \llbracket g_i(x) \rrbracket = u_{0,r}(x)$ for all $x\in \partial (B_1\cap (\Omega_F)_{0,r})$. Hence there is $h=(h_1, \dotsc h_Q) \in W^{1,2}(B_1\cap (\Omega_F)_{0,r}, \R^{nQ})$ harmonic with $g$ as boundary values. $[h]=\sum_{i=1}^Q \llbracket h_i \rrbracket $ is a competitor to $u_{0,r}$ so that
\[
	\int_{B_r\cap \Omega_F} \abs{Du}^2 = \int_{B_1 \cap (\Omega_F)_{0,r}} \abs{Du_{0,r}}^2 \le \int_{B_1 \cap (\Omega_F)_{0,r}} \abs{D[h]}^2= \int_{B_1\cap (\Omega_F)_{0,r}} \abs{Dh}^2.
\]
The previous corollary \ref{cor_B:3.6} applies to $h$ since $\norm{\grad{F_{0,r}}}_{\infty, B_1}= \norm{\grad F}_{\infty, B_r}< \epsilon_F$. So, we find for a fixed $\frac{1}{2}<s'<s$, e.g. $s' =\frac{1+2s}{4}$,
\begin{align*}
	\int_{B_1\cap(\Omega_F)_{0,r}} \abs{Dh}^2 &\le (1+\epsilon) \int_{S^1 \cap (\Omega_F)_{0,r}} \abs{D_\tau h}^2 + \frac{C}{\epsilon} \llfloor h \rrfloor_{s',(\Gamma_F)_{0,r}}^2\\ &\le (1+\epsilon) r\int_{\partial B_r \cap \Omega_F} \abs{D_\tau u}^2 + \frac{C}{\epsilon} r^{2s-1}\llfloor u \rrfloor^2_{s, \Omega_F \cap B_r}
\end{align*}
considering in the last line $[h(x)]=[g(x)]=u_{0,r}(x)$ for $x\in \partial (B_1\cap (\Omega_F)_{0,r})$ and $\llfloor h \rrfloor_{s',(\Gamma_F)_{0,r}}\le C \llfloor u_{0,r} \rrfloor_{s,(\Gamma_F)_{0,r}}=C r^{2s-1}\llfloor u \rrfloor^2_{s, \Omega_F \cap B_r}$ from the $W^{s,2}$-selection,  lemma \ref{lem:A2.401}.\\
If $s(u(0))>0$, i.e. $u(0)=\sum_{j=1}^J Q_j \llbracket p_j \rrbracket$, $\abs{p_i-p_j}\ge s(u(0))$ for $i \neq j$. Fix $R_0>0$ s.t.
\[
	R_0^{\tilde{\alpha}} [u]_{\tilde{\alpha}, \Omega_F \cap B_{R_0}} < \frac{1}{3} s(u(0))
\]
where $[\cdot]_{\tilde{\alpha}, \Omega_F \cap B_{R_0}}$ denotes the H\"older semi-norm on $\Omega_F\cap B_{R_0}$ with exponent $\tilde{\alpha}>0$ provided by theorem \ref{theo_B:3.2}. Hence there are Dirichlet minimizing $u_j\in W^{1,2}(\Omega_F\cap B_{R_0}, \A_{Q_j}(\R^n))$ with 
\begin{equation}\label{eq_B:3.11}
	\G(u_j(x), Q_j \llbracket p_j \rrbracket )< \frac{1}{3} s(u(0)) \text{ for all } x\in \Omega_F\cap B_{R_0}.
\end{equation} 
To each $u_j$ the assumption $s(u_j(0))=0$ is satisfied. So, by the previous considerations for a.e. $0<r\le R_0$
\begin{align*}
	&\int_{B_r\cap \Omega_F} \abs{Du}^2 = \sum_{j=1}^J \int_{B_r\cap \Omega_F} \abs{Du_j}^2 \\
	&\le \sum_{j=1}^J (1+\epsilon) r\int_{\partial B_r \cap \Omega_F} \abs{D_\tau u_j}^2 + \frac{C}{\epsilon} r^{2s-1}\llfloor u_j \rrfloor^2_{s, \Omega_F \cap B_r}\\
	&= (1+\epsilon) r\int_{\partial B_r \cap \Omega_F} \abs{D_\tau u}^2 + \frac{C}{\epsilon} r^{2s-1}\llfloor u \rrfloor^2_{s, \Omega_F \cap B_r}
\end{align*}
where we used in the last step that $\G(u(x), u(y))^2 =\sum_{j=1}^J \G(u_j(x),u_j(y))^2 $ to to \eqref{eq_B:3.11}.
\end{proof}

As theorem \ref{theo_B:2.1} follows from proposition \ref{prop_B:2.3}, we can now use lemma \ref{lem_B:3.7} to give the final argument leading to the H\"older estimate of corollary \ref{cor_B:5.8}.

\begin{proof}[Proof of corollary \ref{cor_B:5.8}]
Let $\alpha>0$ be given as stated. Fix $\epsilon>0$ s.t. $1+\epsilon \le \frac{1}{2\alpha}$ and $0<R<1$ sufficient small s.t. 
\begin{itemize}
	\item[(1)] $R\le R_0$ when $R_0$ is the radius of the previous lemma, \ref{lem_B:3.7};
	\item[(2)] $\norm{\grad F}_{\infty, B_R\cap \Omega_F} < \cos(\frac{2\theta+\pi}{4})$.
\end{itemize}
(2) ensures that $C_\theta \cap B_R \subset C_{\frac{2\theta+\pi}{4}}\cap B_R \subset \Omega_F\cap B_1$. Following the steps in the proof of theorem \ref{theo_B:2.1} for a.e. $0<r\le R$
\begin{align*}
	- \frac{\partial}{\partial r} r^{-2\alpha} \int_{B_r\cap \Omega_F}\abs{Du}^2 &= - r^{-2\alpha} \int_{\partial B_r\cap \Omega_F} \abs{Du}^2 + 2\alpha r^{-2\alpha -1} \int_{B_r\cap \Omega_F} \abs{Du}^2\\
	&\le \frac{C}{\epsilon} r^{(2s-1-2\alpha)-1} \llfloor u  \rrfloor_{s, B_r\cap \Gamma_F}^2 \le \frac{C}{\epsilon} r^{2(\gamma-\alpha)-1} M_u^2.
\end{align*}
Integration in $0<r\le R$ gives 
\begin{equation}\label{eq_B:3.13}
	r^{-2\alpha} \int_{B_r\cap \Omega_F} \abs{Du}^2 \le R^{-2\alpha} \int_{B_R\cap \Omega_F} \abs{Du}^2 + \frac{CR^{2(\gamma-\alpha)}}{\gamma-\alpha} M_u^2.
\end{equation}
By definition of $\delta= \cos(\theta)- \cos(\frac{2\theta+ \pi}{4})$, for all $x\in B_{\frac{R}{2}}\cap C_\theta$ we have $B_{\delta\abs{x}}(x)\subset C_{\frac{2\theta+\pi}{4}}\cap B_R$. Let $x\in B_{\frac{R}{2}}\cap C_\theta$ and $0<r<\frac{R}{2}$ be given, set $r_1=\max\{r,\delta\abs{x}\}$ and $r_2=r_1+\abs{x}\le \frac{2}{\delta}r_1$. We found
\begin{align*}
	&r^{-2\alpha} \int_{B_r(x)\cap \Omega_F} \abs{Du}^2 \le r_1^{-2\alpha} \int_{B_{r_1}(x)\cap \Omega_F} \abs{Du}^2 \le \frac{2^{2\alpha}}{\delta^{2\alpha}} r_2^{-2\alpha}\int_{B_{r_2}(x)\cap\Omega_F} \abs{Du}^2\\
	& \le \frac{4}{\delta^{2\alpha}} \left(\int_{B_R\cap \Omega_F} \abs{Du}^2 + \frac{CR^{2(\gamma-\alpha)}}{\gamma-\alpha} M_u^2\right).
\end{align*}
where we applied at first the internal estimate since $\alpha \le \frac{1}{Q}$ and finally the just established \eqref{eq_B:3.13}. Having established (i), (ii) follows as indicated in the proof of theorem \ref{theo_B:2.1}.
\end{proof}



\newpage

\begin{appendix}
\section{Fractional Sobolev spaces} 
\label{sec:fractional_sobolev_spaces}
We will restrict our overview to the special case of $W^{s,2}=H^s$ for $0<s<1$.

\subsection{General facts} 
\label{sub:general_facts}
At first let us consider the spaces on $\R^N$, there are several ways to define them:
\begin{itemize}
	\item[(a)] using Fourier transform: 
	\[
		H^s(\R^N)=\{ u \in L^2(\R^N) \:\, \abs{\xi}^s\F u(\xi) \in L^2(\R^N)\};
	\]
	\item[(b)] using real interpolation:
	\[
		W^{s,2}(\R^N)=\left(W^{1,2}(\R^N), L^2(\R^N)\right)_{1-s,2};
	\]
	\item[(c)] using the the Gagliardo semi-norm $\llfloor \cdot \rrfloor_{s,\R^N}$
	\[
		W^{s,2}(\R^N)=\left\{ u \in L^2(\R^N) \colon \llfloor u \rrfloor_{s,\R^N}^2=\int_{\R^N \times \R^N} \frac{\abs{u(x)-u(y)}^2}{\abs{x-y}^{N+2s}} \, dxdy;< \infty \right\};
	\]
\end{itemize}
All of these definitions define the same Banach space as can found for instance in \cite{Tartar}: (a)=(c) corresponds to Lemma 16.3 or Lemma 35.2, (a)=(b) can be found in Lemma 23.1.\\

We will be mostly interested in the case of an open domain $\Omega \subset \R^N$.
In this case several definitions are possible, compare \cite[section 34 and section 36]{Tartar}:
\begin{itemize}
	\item[(a)] as restriction 
	\[
		W^{s,2}(\Omega)= \text{ space of restrictions of functions in } W^{s,2}(\R^N);
	\]
	\item[(b)] using interpolation
	\[
		W^{s,2}(\Omega)=\left(W^{1,2}(\Omega), L^2(\Omega)\right)_{1-s,2};
	\]
	\item[(c)] using the Gagliardo norm
	\[
		W^{s,2}(\Omega)=\{ u \in L^2(\Omega) \colon \llfloor u \rrfloor^2_{s,\Omega}=\int_{\Omega \times \Omega} \frac{\abs{u(x)-u(y)}^2}{\abs{x-y}^{N+2s}} \, dxdy < \infty \};
	\]
\end{itemize}

For $\Omega$ with Lipschitz boundary one has the existence of an extension operator that is linear and continuous:
\[
	E: W^{1,2}(\Omega) \to W^{1,2}(\R^N);
\]
$E$ extends to a continuous linear operator mapping $\left(W^{1,2}(\Omega), L^2(\Omega)\right)_{1-s,2}$ into $\left(W^{1,2}(\R^N), L^2(\R^N)\right)_{1-s,2}$; therefore $(a)$ and $(b)$ agree in these cases, compare \cite[section 34]{Tartar}.\\
For Lipschitz domains one can show the existence of a linear continuous extension operator $\tilde{E}: L^2(\Omega) \to L^2(\R^N)$ with $\llfloor\tilde{E}u\rrfloor_{s, \R^N} \le \llfloor u\rrfloor_{s, \Omega}$, so that all definitions agree; compare \cite[Lemma 36.1]{Tartar}.\\

$W^{1,2}(\R^N)$ is dense in $W^{s,2}(\R^N)$ and $W^{1,2}(\Omega)$ in $W^{s,2}(\Omega)$. Since $C_0^\infty(\R^N)$ is dense in $W^{1,2}(\R^N)$ and $C^\infty(\overline{\Omega})$ in  $W^{1,2}(\Omega)$, if $\Omega$ is Lipschitz regular, the same holds true for the interpolation spaces $W^{s,2}(\R^N)$ and $W^{s,2}(\Omega)$.\\

The trace spaces are our main concern. Using the characterisation via the Fourier transform one finds the following, \cite[Lemma 16.1]{Tartar}:\\
For $s>\frac{1}{2}$ functions in $H^s(\R^N)$ have a trace on the hyperplane $x_N=0$ belonging to $H^{s-\frac{1}{2}}(\R^{N-1})$ and this mapping is surjective.\\
But our concern is the trace on $\partial \Omega$ which will be a $C^1$ or Lipschitz manifold. We would like to have a statement as follows:
For $s>\frac{1}{2}$ functions in $W^{s,2}(\Omega)$ have a trace $u\tr{\partial \Omega}$ belonging to $W^{s-\frac{1}{2},2}(\partial \Omega)$ and this mapping is surjective.\\

How can we best describe $W^{s,2}(\partial \Omega)$? The definitions (a),(b),(c) for $W^{s,2}(\Omega)$, $\Omega \subset \R^N$ an open Lipschitz regular domain are all non-local. One can check that all definitions share the following property: Let $U_1, U_2 \subset \Omega$ be an open cover of $\Omega$ and $u \in L^2(\Omega)$ satisfies $u\tr{U_i} \in W^{s,2}(U_i)$ for $i=1,2$ then $u \in W^{s,2}(\Omega)$. We are looking now for an general approach to localize that works for all three definitions. This is desirable to define $W^{s,2}(\partial \Omega)$ for a $C^1$- or Lipschitz regular domain $\Omega \subset \R^N$ since  $\Omega$ has the defining property that locally $\Omega$ looks like $\Omega_F=\{ x \in \R^N \, :\, x_N > F(x')\}$, for a $C^1$ or Lipschitz continuous function $F$, where $x'=(x_1, \dotsc, x_{N-1})$. We would like to reduce our analysis to such a local description.\\
For this aim the following two observations are useful:
\begin{itemize}
	\item[(i)] equivalence under bilipschitz transformations;\\
	\item[(ii)] one can "localise" and a "local" description controls the global one.
\end{itemize}

Concerning (i): let $\psi: \Omega' \to \Omega$ be bilipschitz, $\Omega$ $N$-dimensional; then we may define a linear operator $u \mapsto \psi^\sharp u = u\circ \psi$ with 
\begin{align*}
	\norm{\psi^\sharp u}_{L^2(\Omega')} &\le Lip(\psi^{-1})^{\frac{N}{2}} \norm{u}_{L^2(\Omega)}\\
	\norm{\grad(\psi^\sharp u)}_{L^2(\Omega')} &= \norm{D\psi^t \grad(u)\circ \psi}_{L^2(\Omega')} \le Lip(\psi) Lip(\psi^{-1})^{\frac{N}{2}} \norm{\grad u}_{L^2(\Omega)};
\end{align*}
therefore $\psi^\sharp$ extends to a continuous linear operator on the interpolation spaces $\left(W^{1,2}(\Omega'), L^2(\Omega')\right)_{1-s,2} \to \left(W^{1,2}(\Omega), L^2(\Omega)\right)_{1-s,2}$.\\
For the Gagliardo semi-norm,we define the constant $C_\psi= Lip(\psi^{-1})^{2N} Lip(\psi)^{N+2s}$ and use $\abs{x-y}\le Lip(\psi) \abs{\psi^{-1}(x)-\psi^{-1}(y)}$ with a change of variables to conclude that,
\[
	\int_{\Omega' \times \Omega'} \frac{\abs{\psi^\sharp u (x)- \psi^\sharp u(y)}^2}{\abs{x-y}^{N+2s}} dydx \le C_\psi \int_{\Omega \times \Omega} \frac{\abs{u(x)-u(y)}^2}{\abs{x-y}^{N+2s}} dydx.
\]

Concerning (ii): Interpolation behaves well for finite tensor products in the sense that
\begin{equation}\label{eq:A1.1}
	\left(\bigotimes_{i=1}^L E_{0,i},\bigotimes_{i=1}^L E_{1,i}\right)_{\theta,p}= \bigotimes_{i=1}^L (E_{0,i},E_{1,i})_{\theta,p}.
\end{equation}
We will show that below. Assuming \eqref{eq:A1.1} holds true we can check (ii). Given any finite open cover $\{U_i\}_{i=1, \dotsc, L}$ of $\Omega$ with subordinate partition of unity $(\theta_i)_{i=1, \dotsc, L}$ we define
\[
	R: W^{1,2}(\Omega) \to \bigotimes_{i=1}^L W^{1,2}(U_i) \quad Ru=(u_1,\dotsc ,u_L), 
\]
where $u_i$ is the restriction of $u$ to $U_i$, and
\[ 
 T: \bigotimes_{i=1}^L W^{1,2}(U_i) \to W^{1,2}(\Omega) \quad T(u_1, \dotsc, u_L) = \sum_{i=1}^L \theta_i u_i.
\]
Both operators are linear and continuous, because
\begin{align*}
	\sum_{i=1}^L \norm{u_i}_{L^{2}(U_i)} &\le L \norm{u}_{L^2(\Omega)} \\
	\sum_{i=1}^L \norm{\grad(u_i)}_{L^2(U_i)} &\le L \norm{\grad(u)}_{L^2(\Omega)}\\
	\norm{\sum_{i=1}^L \theta_i u_i}_{L^2(\Omega)} &\le \sum_{i=1}^L \norm{u_i}_{L^2(U_i)}\\
	\norm{\sum_{i=1}^L \grad(\theta_i u_i)}_{L^2(\Omega)} &\le \sum_{i=1}^L \abs{\grad(\theta_i)}_\infty \norm{u_i}_{L^2(U_I} + \abs{\theta_i}_{\infty}\norm{\grad(u_i)}_{L^2(U_i)}. 
\end{align*}
Using \eqref{eq:A1.1} they extend to linear continuous operators
\begin{align*}
	R: W^{s,2}(\Omega) &\to \bigotimes_{i=1}^L W^{s,2}(U_i)\\
	T: \bigotimes_{i=1}^L W^{s,2}(U_i) &\to W^{s,2}(\Omega).
\end{align*}
By definition $T\circ R = \mathbf{1}_{W^{s,2}(\Omega)}$ since the equality is obvious on $W^{1,2}(\Omega)$. This shows (ii) in the interpolation case.\\

It remains to check \eqref{eq:A1.1}. Let $\{(E_{0,i}, E_{1,i})\}_{i=1,\dotsc, L}$ be finitely many tuples of Banach spaces admissible for interpolation. We can consider the interpolation of their tensor product:
\begin{align*}
	E_0= \bigotimes_{i=1}^L E_{0,i} \text{ equipped with the norm } \norm{a}_0= \sum_{i} \norm{a_i}_{0,i}\\
	E_1= \bigotimes_{i=1}^L E_{1,i} \text{ equipped with the norm } \norm{a}_1= \sum_{i} \norm{a_i}_{1,i}
\end{align*}
Hence for the $K$ functional in real interpolation we have
\[
	K(t,a)=\inf_{\substack{a=a_0+a_1 \Leftrightarrow \\
	a_i=a_{0,i}+a_{1,i}; i=1, \dotsc, L}} \norm{a_0}_0 + t \norm{a_1}_1 = \sum_{i=1}^L K_i(t,a_i) \ge K_j(t,a_j)
\]
and this establishes \eqref{eq:A1.1} because
\begin{align*}
	\frac{1}{L} \sum_{i=1}^L \norm{t^{-\theta} K_i(t,a_i)}_{L^p(\R_+; \frac{dt}{t})} &\le \norm{t^{-\theta} K(t,a)}_{L^p(\R_+; \frac{dt}{t})}\\
	& \le \sum_{i=1}^L \norm{t^{-\theta} K_i(t,a_i)}_{L^p(\R_+; \frac{dt}{t})}.
\end{align*}

To check (ii) in the case of the Gagliardo semi-norm we have for the restrictions
\[
	\sum_{i=1}^L \llfloor u_i \rrfloor_{s,U_i} \le L \llfloor u \rrfloor_{s,\Omega}.
\]
For an arbitrary Lipschitz function $f$ and $\Omega_1= \Omega \cap supp(f)$ write
\begin{align*}
	&\int_{\Omega \times \Omega} \frac{\abs{(fu)(x)-(fu)(y)}^2}{\abs{x-y}^{N+2s}} dy dx = \int_{\Omega_1 \times \Omega_1} \frac{\abs{(fu)(x)-(fu)(y)}^2}{\abs{x-y}^{N+2s}} dy dx\\
	&= \int_{\substack{\Omega_1\times \Omega_1\\\abs{x-y} <1 }} \frac{\abs{(fu)(x)-(fu)(y)}^2}{\abs{x-y}^{N+2s}} dy dx +  \int_{\substack{\Omega_1\times \Omega_1\\\abs{x-y} \ge 1 }} \frac{\abs{(fu)(x)-(fu)(y)}^2}{\abs{x-y}^{N+2s}} dy dx;
\end{align*}
for the second integral we have 
\[
	\int_{\substack{\Omega_1\times \Omega_1\\\abs{x-y} \ge 1 }} \frac{\abs{(fu)(x)-(fu)(y)}^2}{\abs{x-y}^{N+2s}} dy dx \le 4 \abs{f}^2_\infty \frac{N\omega_N}{2s} \int_{\Omega_1} \abs{u}^2
\]
where we used symmetry in $x,y$ and 
\[
	\int_{\Omega\setminus B_1(x)} \frac{1}{\abs{x-y}^{N+2s}} dy \le N\omega_N \int_1^{\infty} r^{-1-2s} dr = \frac{N\omega_N}{2s};
\]
for the first integral we have
\begin{align*}
	&\int_{\substack{\Omega_1\times \Omega_1\\\abs{x-y} <1 }} \frac{\abs{(fu)(x)-(fu)(y)}^2}{\abs{x-y}^{N+2s}} dy dx \\
	&\le 2\abs{f}^2_\infty \int_{\Omega_1 \times \Omega_1} \frac{\abs{u(x)-u(y)}^2}{\abs{x-y}^{N+2s}} dy dx + Lip(f)^2 \frac{2N\omega_N}{2-2s} \int_{\Omega_1} \abs{u}^2
\end{align*}
where we used $\abs{(fu)(x)-(fu)(y)} \le \abs{f}_\infty \abs{u(x)-u(y)} + \abs{f(x)-f(y)} \abs{u(x)} \le \abs{f}_\infty \abs{u(x)-u(y)} + \abs{u(x)} Lip(f) \abs{x-y}$ and
\[
	\int_{\Omega\cap B_1(x)} \frac{\abs{x-y}^2}{\abs{x-y}^{N+2s}} \le N \omega_N \int_{0}^{1} r^{2-2s-1} dr  = \frac{N\omega_N}{2-2s}.	
\]
Hence we got the desired estimate with the constant $C_f= 2 \abs{f}^2_\infty + \frac{2N\omega_N}{s(1-s)}(Lip(f)^2 +\abs{f}_\infty^2)$
\[
	\int_{\Omega\times \Omega} \frac{\abs{(fu)(x)-(fu)(y)}^2}{\abs{x-y}^{N+2s}} dy dx \le C_f \left( \llfloor u \rrfloor_{s,\Omega_1}^2 + \norm{u}^2_{L^2(\Omega_1)}\right).
\]
Using this estimate we can conclude (ii) in case of using the Gagliardo semi-norm since
\[
	\llfloor \sum_{i=1}^L \theta_i u_i \rrfloor_{s,\Omega} \le \sum_{i=1}^L \llfloor \theta_i u_i \rrfloor_{s,\Omega} \le C \left( \sum_{i=1}^L \llfloor u_i \rrfloor_{s,U_i} + \norm{u_i}_{L^2(U_i)}\right).
\]

Due to (ii) it is sufficient to consider the case $\Omega_F$, Furthermore using (i) with the bilipschitz mapping $(x',x_N) \mapsto (x', x_N + F(x'))$ between $\R^N_+$ and $\Omega_F$, it is sufficient to understand $\R^N_+$.
Hence as definition for the spaces on the boundary we may use
\[
	W^{s,2}(\partial \Omega_F) = \{ u(x', x_N - F(x')) \, : \, u \in W^{s,2}(\R^N_+)\};
\]
for the Gagliardo seminorm we may use as well the global version
\[
	\llfloor u \rrfloor_{s, \partial \Omega}^2 = \int_{\partial \Omega \times \partial \Omega} \frac{\abs{u(x)-u(y)}^2}{\abs{x-y}^{N-1+2s}} dydx.
\]

\begin{corollary}\label{cor:A1.1}
For $s>\frac{1}{2}$ functions of $W^{s,2}(\R_+^N)$ have a trace on the hyperplane $x_N=0$ belonging to $W^{s-\frac{1}{2},2}(\R^{N-1})$ and this linear continuos mapping $\tr{\partial \R^{N}_+}$ is surjective.
\end{corollary}

\begin{proof}
	$u \in W^{s,2}(\R_+^N)$ if and only if the extension
	\[
		Eu(x)=
		\begin{cases}
			u(x',x_N), &\text{ if } x_N > 0 \\
			u(x',-x_N), &\text{ if } x_N < 0
		\end{cases}
	\]
	is an element of $W^{s,2}(\R^N)=H^s(\R^N)$. Composing this operator with the continuous linear trace operator defined on the whole space using the Fourier transform shows existence. Furthermore it inherits all its properties and hence concludes the proof. 
\end{proof}

The following characterisation for the trace of a function provides a tool to check that a function $u \in W^{s,2}(\Omega)$ can be patched together with a function $v \in W^{s,2}(\R^N \setminus \Omega)$ to a function $U \in W^{s,2}(\R^N)$ if their traces coincide. 
As introduced before: $\Omega_F=\{x \in \R^N \colon x_N > F(x')\}$,$F$ Lipschitz continuous
\begin{lemma}\label{lem:A1.2} For $u \in W^{s,2}(\Omega_F)$, one has
\begin{equation}\label{eq:A1.2}
	\norm{\frac{u(x',x_N)- u\tr{\partial\Omega_F}(x')}{\abs{x_N- F(x')}^s}}_{L^2(\Omega_F)} \le C \llfloor u \rrfloor_{s,\Omega_F}
\end{equation}
\end{lemma}

\begin{proof}
Using the bilipschitz mapping $(x',x_N) \mapsto (x', x_N - F(x'))$ and $v(x',x_N)=u(x', F(x')+x_N) \in W^{s,2}(\R^N_+)$ together with
\[
	\int_{\Omega_F} \frac{\abs{u(x',x_N)- u\tr{\partial\Omega_F}(x')}^2}{\abs{x_N- F(x')}^{2s}} dx = \int_{R^N_+} \frac{\abs{u(x',x_N+F(x'))- u\tr{\partial\Omega_F}(x')}^2}{\abs{x_N}^{2s}} dx;
\]
one has only to consider the case $F=0$, i.e. $\R^N_+$.\\
We may extend $u$ by $u(x',-x_N)$ for $x_N<0$ to obtain $u \in W^{s,2}(\R^N)=H^s(\R^N)$. We define $v_{x_N}(x')=u(x',x_N)$, then $\F v_{x_N}(\xi')= \int_{\R} e^{2i\pi \xi_Nx_N} \F u(\xi',\xi_N) d\xi_N$ and $\F 'u\tr{\partial \R^{N}_+}(\xi')=\F v_0(\xi')= \int_{\R} \F u(\xi',\xi_N) d\xi_N$; hence by Cauchy inequality
\begin{align*}
	&\abs{\F v_{x_N}(\xi')- \F v_0(\xi')}^2 = \left( \int_\R (e^{2i\pi \xi_N x_N}-1) \F u (\xi',\xi_N) d\xi_N\right)^2\\
	&\le 4 \left( \int_\R \frac{\abs{\sin( \pi\xi_N x_N)}}{\abs{\xi_N x_N}^\alpha } x_N d\xi_N \right) x_N^{\alpha-1} \left( \int_\R \abs{\sin(\pi \xi_N x_N)} \abs{\xi_N}^\alpha \abs{\F u }^2(\xi',\xi_N) d\xi_N\right);
\end{align*}
Multiply this by $\abs{x_N}^{-2s}$ and integrate in $x_N$ to conclude
\begin{align*}
	&\int_{\R} \abs{x_N}^{-2s} \abs{\F v_{x_N}(\xi')- \F v_0(\xi')}^2 dx_N\\
	 &\le 4 C(\alpha) \int_\R \left(\int_\R \frac{\abs{\sin( \pi\xi_N x_N)}}{\abs{\xi_N x_N}^{1+2s-\alpha }} \abs{\xi_N} dx_N \right) \abs{\xi_N}^{2s} \abs{\F u }^2(\xi',\xi_N) d\xi_N\\
	&= 4C(\alpha)^2 \int_\R \abs{\xi_N}^{2s} \abs{\F u }^2(\xi',\xi_N) d\xi_N\
\end{align*}
where $C(\alpha)= \int_{\R} \frac{\sin(\pi t)}{\abs{t}^\alpha} dt< \infty$ for $\alpha=1+2s-\alpha$ (note that $1< \frac{1}{2} + s = \alpha < 2$). This gives the desired result by integrating in $\xi'$, since
\begin{equation*}
	\int_{\R^N} \frac{\abs{u(x',x_N)- u\tr{\partial\Omega_F}(x')}^2}{\abs{x_N}^{2s}} dx = \int_{\R} \abs{x_N}^{-2s}\int_{\R^{N-1}} \abs{\F v_{x_N}(\xi') - \F v_0(\xi')}^2 d\xi' dx_N.
\end{equation*}
\end{proof}

For $s=1$ compare lemma \ref{lem_B:3.3}, that corresponds to \cite[Lemma 13.5]{Tartar}.
We can conclude the following corollary

\begin{corollary}\label{cor:A1.3}
$v \in L^2(\R^{N-1})$ is the trace of $u$ 
(and so in $W^{s-\frac{1}{2},2}(\R^{N-1})$) if 
\begin{equation}\label{eq:A1.3}
	\norm{\frac{u(x',x_N)- v(x')}{\abs{x_N- F(x')}^s}}_{L^2(\Omega_F)} < \infty
\end{equation}
\end{corollary}

\begin{proof}
	\begin{align*}
		&\int_{\R^{N-1}} \abs{v(x')-u\tr{\partial\Omega_F}(x')}^2 dx'
		\le 2 \epsilon^{2s} \frac{1}{\epsilon}\int_0^\epsilon\int_{\R^{N-1}} \frac{\abs{v(x') - u(x',F(x')+x_N)}^2}{\abs{x_N}^{2s}} \\
		&+  2 \epsilon^{2s} \frac{1}{\epsilon}\int_0^\epsilon\int_{\R^{N-1}} \frac{\abs{u(x',F(x')+x_N)-u\tr{\partial\Omega_F}(x')}^2}{\abs{x_N}^{2s}} dx' dx_N\\
		&\le 2 \epsilon^{2s-1} \left(\norm{\frac{u(x',x_N)- v(x')}{\abs{x_N- F(x')}^s}}^2_{L^2(\Omega_F)} + \norm{\frac{u(x',x_N)- u\tr{\partial\Omega_F}(x')}{\abs{x_N- F(x')}^s}}^2_{L^2(\Omega_F)}\right);
	\end{align*}
	converging to $0$ as $\epsilon \to 0$ hence $v=u\tr{\partial\Omega_F}$.
\end{proof}

\begin{corollary}\label{cor:A1.4}
Let $u\in W^{s,2}(\Omega_F)$ and $v\in W^{s,2}(\R^N \setminus \Omega_F)$ for $s>\frac{1}{2}$ satisfying $u\tr{\partial\Omega_F} =v\tr{\partial\Omega_F}$ then 
\begin{equation}\label{eq:A1.4}
	U(x)=\begin{cases}
		u(x), &\text{ if } x \in \Omega_F\\
		v(x), &\text{ if } x \in \R^N\setminus \Omega_F
	\end{cases}
\end{equation}
defines an element in $W^{s,2}(\R^N)$ satisfying
\begin{equation}\label{eq:A1.5}
	\llfloor U\rrfloor_{s,\R^N} \le C \left(\llfloor u \rrfloor_{s,\Omega_F} + \llfloor v \rrfloor_{s,\R^N \setminus \Omega_F}\right)
\end{equation}
\end{corollary}

\begin{proof}
As before using the bilipschitz mapping $(x',x_N) \mapsto (x', x_N - F(x'))$ one has only to consider the case $F=0$; then
\begin{align*}
	&\norm{U}^2_{L^2(\R^N)} = \norm{u}^2_{L^2(\R^N_+)}+ \norm{v}^2_{L^2(\R^N_-)}\\
	&\int_{\R^N \times \R^N} \frac{\abs{U(x)-U(y)}^2}{\abs{x-y}^{N+2s}} dy dx
	= 2 \int_{\R_+^N \times \R_-^N} \frac{\abs{u(x)-v(y)}^2}{\abs{x-y}^{N+2s}} dy dx\\&+\int_{\R_+^N \times \R_+^N} \frac{\abs{u(x)-u(y)}^2}{\abs{x-y}^{N+2s}} dy dx +\int_{\R_-^N \times \R_-^N} \frac{\abs{v(x)-v(y)}^2}{\abs{x-y}^{N+2s}} dy dx.
\end{align*}
The first two summands are obviously bounded and the third is bounded because
\begin{align}\nonumber
	&\int_{\R_+^N \times \R_-^N} \frac{\abs{u(x)-v(y)}^2}{\abs{x-y}^{N+2s}} dy dx\le\\ 
	& 3 \int_{\R_+^N \times \R_-^N} \frac{\abs{u\tr{\partial\Omega_F}(x')-v\tr{\partial\Omega_F}(y')}^2}{\abs{x-y}^{N+2s}} dy dx \label{eq:int1}\\
	&+ 3 \int_{\R_+^N \times \R_-^N} \frac{\abs{u(x)- u\tr{\partial\Omega_F}(x')}^2}{\abs{x-y}^{N+2s}} dy dx + 3 \int_{\R_+^N \times \R_-^N} \frac{\abs{v\tr{\partial\Omega_F}(y')-v(y)}^2}{\abs{x-y}^{N+2s}} dy dx.\label{eq:int2}
\end{align}
For the first integral, \eqref{eq:int1}, we have 
\begin{align*}
&\int_{\R_+^N \times \R_-^N} \frac{\abs{u\tr{\partial\Omega_F}(x')-v\tr{\partial\Omega_F}(y')}^2}{\abs{x-y}^{N+2s}} dy dx\\
\le C_1 &\int_{\R^{N-1}\times \R^{N-1}} \frac{\abs{u\tr{\partial\Omega_F}(x')-v\tr{\partial\Omega_F}(y')}^2}{\abs{x'-y'}^{N-2+2s}} dy' dx' \le C \llfloor u \rrfloor^2_{s,\R^N_+},
\end{align*}
where we used firstly 
\[
\int_{\R_+\times \R_-} \frac{1}{\abs{x-y}^{N+2s}} dx_N dy_N =\int_{\R_+\times \R_+} \frac{(1+ (t+\tau)^2)^{-\frac{N}{2}-s}}{\abs{x'-y'}^{N-2 + 2s}} d\tau dt = \frac{C_1}{\abs{x-y}^{N-2 + 2s}}
\]
by means of the change of variables $x_N= \abs{x'-y'}t, y_N=-\abs{x'-y'}\tau$ and then $u\tr{\partial\Omega_F}=v\tr{\partial\Omega_F}$ together with the continuity of the trace operator $\tr{\partial\Omega_F}: W^{s,2}(\R^N_+) \to W^{s-\frac{1}{2},2}(\R^{N-1})$, compare \cite[lemma 16.1, lemma 16.3]{Tartar}.\\
For the second and third integral, \eqref{eq:int2}, we proceed equivalently. For instance for the the second
\[
\int_{\R_+^N \times \R_-^N} \frac{\abs{u(x)- u\tr{\partial \R^N_+}(x')}^2}{\abs{x-y}^{N+2s}} dy dx \le C_2 \int_{\R_+^N} \frac{\abs{u(x',x_N)- u\tr{\partial\Omega_F}(x')}^2}{\abs{x_N}^{2s}} dx \le C \llfloor u \rrfloor_{s,\R^N_+}^2
\]
where we used
\[
\int_{\R^N_-} \frac{1}{\abs{x-y}^{N+2s}} dy = x_N^{-2s} \int_{\R^N_+} \frac{1}{\abs{z+e_N}^{N+2s}} dz = x_N^{-2s} C_2.
\]
by means of the change of variables $(y',y_N)=(x'-x_N z', -x_Nz_N)$, $x_N>0$ and afterwards we apply lemma \ref{lem:A1.2}.\\
The constants $C_1,C_2$ are indeed finite since $(t+\tau)^2 \ge t^2 +\tau^2$
\begin{align*}
	C_1 \le \int_0^\infty \int_0^{\frac{\pi}{2}} \frac{r dr d\theta}{(1+r^2)^{\frac{N}{2}+s}}=\frac{\pi}{2N-4+4s}\\
	C_2 \le \int_{\R^N\setminus B_1(-e_N)}\frac{1}{\abs{z+e_N}^{N+2s}} dz = \frac{N\omega_N}{2s}.
\end{align*}
\end{proof}

A further nice consequence is the following characterisation of $W^{s,2}_0(\Omega)$, defined as the closure of $C^\infty_c(\Omega)$ in $W^{s,2}(\R^N)$. The "classical" case, $s=1$, is considered in \cite[Lemma 13.6]{Tartar}.

\begin{corollary}\label{cor:A1.5}
If $F$ is Lipschitz continuous and $s>\frac{1}{2}$ then $W_0^{s,2}(\Omega_F)$ is the subspace of $u \in W^{s,2}(\Omega_F)$ satisfying $u\tr{\partial\Omega_F}=0$.
\end{corollary}

\begin{proof}
If $u\in W^{s,2}_0(\Omega_F)$ there exists a sequence $u_n \in C^\infty_c(\Omega_F)$ s.t. $u_n \to u$ in $W^{s,2}(\Omega_F)$; as $\tr{\partial\Omega_F}$ is a continuous operator on $W^{s,2}(\Omega_F)$ we have $0=u_n\tr{\partial\Omega_F} \to u\tr{\partial\Omega_F}$ in $L^2(\R^{N-1})$.\\
We may extend $u$ by $0$ outside of $\Omega_F$ and denote the extension by $U$. The corollary above shows that $U \in W^{s,2}(\R^N)$. One chooses $0\le \theta \le 1 \in C^\infty_c(\R^N)$ s.t. $\theta(x)=1$ for $\abs{x}<1$. One approaches $U$ by the sequence $u_n(x',x_N)= U(x, x_N-\frac{1}{n}) \theta(\frac{x}{n}) \in W^{s,2}(\R^N)$. $u_n$ converges to $U$ by Lebesgue dominated convergence. The support of these $u_n$ is compactly supported within $\Omega_F$. Finally regularise $u_n$ by convolution.  
\end{proof}

Using interpolation theory there is an elegant way to obtain a statement on compact embeddings:
\begin{lemma}\label{lem:A1.6}
If $\Omega\subset \R^N$ and bounded, then the injection of $W_0^{s,2}(\Omega)$ into $L^2(\Omega)$ is compact.
\end{lemma}

\begin{proof}
We have to show that for a bounded sequence $u_n \in W_0^{s,2}(\Omega)$, there is a subsequence converging strongly in $L^2(\Omega)$. To do so it is sufficient to check that for every $\epsilon >0 $ there is a compact subset $K_\epsilon$ of $L^2(\Omega)$ s.t. we can decompose $u_n= v_{n,\epsilon}+ w_{n,\epsilon}$ with $\norm{w_{n,\epsilon}}_{L^2(\Omega)} \le \epsilon C$ and $v_{n,\epsilon} \in K_\epsilon $ for all $n$.\\
Firstly we may extend each $u_n$ by $0$ outside of $\Omega$. For a special smoothing sequence $\rho_\epsilon(x)= \frac{1}{\epsilon^N}\rho_0(\frac{x}{\epsilon})$ with $\rho_0$ radial we can consider the linear operators $u \mapsto u - \rho_\epsilon \star u $. For them we clearly have
\begin{align*}
	\norm{ u - \rho_\epsilon \star u}_{L^2(\R^N)} &\le 2 \norm{u}_{L^2(\R^N)}\\
	\norm{ u - \rho_\epsilon \star u}_{L^2(\R^N)} &\le \int_{\R^N} \rho_\epsilon(y) \norm{u(\cdot) - u(\cdot - y)}_{L^2(\R^N)} dy\\
	& \le \norm{\grad(u)}_{L^2(\R^N)} \int_{\R^N} \abs{y} \rho_\epsilon(y) dy \le A\epsilon \norm{\grad(u)}_{L^2(\R^N)}. 
\end{align*}
$(\mathbf{1}- \rho_\epsilon \star)$ extends to a continuous linear operator on $W^{s,2}(\R^N)$. It therefore satisfies $\norm{u-\rho_\epsilon \star u}_{L^2(\R^N)} \le 2^{1-s} A^s \epsilon^s \norm{u}_{W^{s,2}(\R^N)}$. The choice $w_{n,\epsilon}= u_n-\rho_\epsilon\star u_n$ has $\norm{w_{n,\epsilon}}_{L^2(\R^N)} \le C \epsilon^s$ for all $n$ and since $\norm{\frac{\partial \rho_\epsilon\star u_n}{\partial x_j}}_{L^\infty(\R^N)} \le \norm{\frac{\partial \rho_\epsilon}{\partial x_j}}_{L^2(\R^N)} \norm{u_n}_{L^2(\R^N)}$, the sequence $v_{n,\epsilon}$ stays in a bounded set of Lipschitz functions and keeps their support in a fixed compact set of $\R^N$. The Arzel\'a-Ascoli theorem provides a subsequence converging strongly in $L^\infty$ and hence $L^2$, concluding the statement.
\end{proof}
The existence of a continuous linear extension operator $E: W^{s,2}(\Omega) \to W^{s,2}(\R^N)$ for Lipschitz regular domains extends the result to bounded domains i.e. the injection of $W^{s,2}(\Omega)$ into $L^2(\Omega)$ is compact for $\Omega \subset \R^N$ bounded and Lipschitz regular. 

As usual the compact embedding can be used to prove Poincar\'e inequalities:

\begin{lemma}\label{lem:A1.7}
For a bounded, Lipschitz regular domain $\Omega \subset \R^N$ and $0<s\le 1$ there is a constant $C_1$ s.t. for each $u\in W^{s,2}(\Omega)$
\begin{equation}\label{eq:A1.6}
	\norm{u - \fint_{\Omega} u }_{L^2(\Omega)} \le C_1 \llfloor u \rrfloor_{s,\Omega};
\end{equation}
for $\frac{1}{2}< s \le 1$ there is a constant $C_2$ s.t. for each $u \in W^{s,2}(\Omega)$
\begin{equation}\label{eq:A1.7}
	\norm{u - \fint_{\partial \Omega} u\tr{\partial\Omega} }_{L^2(\Omega)} \le C_2 \llfloor u \rrfloor_{s,\Omega}.
\end{equation}
\end{lemma}

\begin{proof}
Both proofs are along the same lines. For the second we need the continuity of the trace operator $\tr{\partial\Omega}$ and so $s>\frac{1}{2}$. Nonetheless we will only present the second case and it will be obvious how to argue in the first. We argue by contradiction; so we assume that there exists a sequence $u_k \in W^{s,2}(\Omega)$ with 
\[
	\norm{u_k - \fint_{\partial \Omega} u_k\tr{\partial\Omega}}_{L^2(\Omega)} > k \llfloor u_k \rrfloor_{s,\Omega}.
\]
Normalising via
\[
	v_k= \frac{ u_k - \fint_{\partial \Omega} u_k\tr{\partial\Omega}}{\norm{u_k - \fint_{\partial \Omega} u_k\tr{\partial\Omega}}_{L^2(\Omega)}} \text{ for all } k
\]
we may assume that $\norm{v_k}_{L^2(\Omega)}=1$, $\fint_{\partial \Omega}v_k\tr{\partial\Omega}=0$ and by assumption $\llfloor v_k \rrfloor_{s,\Omega}< \frac{1}{k}$ for all $k$. In particular the sequence stays in  a fixed bounded set of $W^{s,2}(\Omega)$. We may pass to a subsequence $v_{k'}$ converging strongly in $L^{2}(\Omega)$ to a function $v \in L^2(\Omega)$, due to the just obtained compact embedding of $W^{s,2}(\Omega)$ into $L^2(\Omega)$. $v$ needs to be constant since $\llfloor v_k \rrfloor_{s,\Omega}< \frac{1}{k}$. Thus $v_{k'} \to v$ strongly in  $W^{s,2}(\Omega)$. The continuity of the trace operator provides
\[
	\fint_{\partial \Omega} v\tr{\partial\Omega} = \lim_{k' \to \infty} \fint_{\partial \Omega} v_{k'}\tr{\partial\Omega}=0.
\]
This contradicts $\norm{v}_{L^2(\Omega)}=1$ because $v = const.$ implies $v\tr{\partial\Omega}= const. = 0$.
\end{proof}

For our purpose a particular version of corollary \ref{cor:A1.4} is needed:
\begin{corollary}\label{cor:A1.8}
To any given $-1<a<1$ and $\frac{1}{2} <s \le 1$ there is a constant $C>$ with the property, that if $u \in W^{s,2}(\Sp^{N-1}\cap \{ x_N> a\}), v \in W^{s,2}(\Sp^{N-1}\cap \{x_N<a\})$ with $u\tr{\Sp^{N-1}\cap \{ x_N= a\}}=v\tr{\Sp^{N-1}\cap \{ x_N= a\}}$ then 
\begin{equation}\label{eq:A1.8}
	U(x)=\begin{cases}
		u(x), &\text{ if } x \in \Sp^{N-1}, x_N>a\\
		v(x), &\text{ if } x \in \Sp^{N-1}, x_N<a
	\end{cases}
\end{equation}
defines an element in $W^{s,2}(\Sp^{N-1})$ satisfying
\begin{equation}\label{eq:A1.9}
	\llfloor U\rrfloor_{s,\Sp^{N-1}} \le C \left( \llfloor u\rrfloor_{s,\Sp^{N-1}\cap \{ x_N> a\}} + \llfloor v\rrfloor_{s,\Sp^{N-1}\cap \{ x_N< a\}}\right)
\end{equation}
\end{corollary}

\begin{proof}
We can apply corollary \ref{cor:A1.4} locally using a partition of unity $\{\theta_i\}_{i=1}^L$ subordinate to a coordinated atlas $(U_i, \varphi_i)_{i=1, \dotsc, L}$. More detailed, we may choose a smooth atlas $(U_i, \varphi_i)_{i=1, \dotsc, L}$ with the additional property that every chart $\varphi_i\colon U_i \subset \Sp^{N-1} \to V_i \subset \R^{N-1}$ satisfies $\varphi_i( U_i \cap \{ x_N \ge a \} ) = V_i \cap \{y_{N-1} \ge a \}$. We may now apply corollary \ref{cor:A1.4} to each pair $u \vert_{U_i} \circ \varphi_i^{-1}, v \vert_{U_i} \circ \varphi_i^{-1}$ and obtain functions $U_i \in W^{s, 2}(V_i)$. Using a subordinated partition of unity $\{\theta_i\}_{i=1}^L$, the function $U(x)= \sum_{i=1}^L \theta_i(x) U_i\circ\varphi_i(x)$ agrees by construction with $u$ on $S^+=\Sp^{N-1}\cap \{ x_N> a\}$ and with $v$ on $S^-=\Sp^{N-1}\cap \{ x_N< a\}$. Furthermore it satisfies for a constant $C>0$
\[
	\llfloor U \rrfloor_{s,\Sp^{N-1}} \le \norm{U}_{W^{s,2}(\Sp^{N-1})} \le C \left( \norm{u}_{W^{s,2}(S^+)}+ \norm{v}_{W^{s,2}(S^-)} \right).
\]
because every $U_i$ does. 
To pass to the desired inequality \eqref{eq:A1.9} we proceed as follows: Given $u,v$ satisfying the assumption, we can apply the above construction to 
\[ \tilde{u} = u - \fint_{\partial S^+} u\tr{\partial S^+},  \;\tilde{v} = v - \fint_{\partial S^-} v\tr{\partial S^-}, \]
because $\tilde{u}$, $\tilde{v}$ still satisfy the assumptions as a consequence of $ u\tr{\partial S^+} = v\tr{\partial S^{-}}$. We obtain $\tilde{U}$ and $U$ with $\tilde{U} = U - \fint_{\partial S^+} u\tr{\partial S^+}$. We can now conclude \eqref{eq:A1.9} by applying the Poincar\'e inequality \eqref{eq:A1.7}, since
\begin{align*}
	\llfloor \tilde{U}\rrfloor_{s,\Sp^{N-1}} &= 	\llfloor U\rrfloor_{s,\Sp^{N-1}}\\
	\norm{\tilde{u}}_{W^{s,2}(S^+)} &= \norm{u - \fint_{\partial S^+} u\tr{\partial S^+} }_{L^2(S^+)} + \llfloor u \rrfloor_{s,S^+} \le C \llfloor u \rrfloor_{s,S^+}\\
	\norm{\tilde{v}}_{W^{s,2}(S^-)} &= \norm{v - \fint_{\partial S^-} v\tr{\partial S^+} }_{L^2(S^-)} + \llfloor v \rrfloor_{s,S^-} \le C \llfloor v \rrfloor_{s,S^-}.
\end{align*}
\end{proof}

\subsection{Interpolation for fractional Sobolev functions} 
\label{sub:interpolation_for_fractional_sobolev_functions}
Commonly one can use a version of the Luckhaus' lemma to interpolate between two functions on the sphere. If an $L^\infty$-estimate is not needed it states:\\

To any $0<\epsilon<\frac{1}{2}$ and $u,v \in W^{1,2}(\Sp^{N-1})$ there is $w \in W^{1,2}(B_1\setminus B_{(1-\epsilon)})$ with $w(x)=u(x)$ and $w((1-\epsilon)x)=v(x)$ for all $x \in \Sp^{N-1}$, satisfying
\begin{equation}\label{eq:A2.1}
\int_{B_1\setminus B_{1-\epsilon}} \abs{Dw}^2 \le 2 \epsilon \int_{\Sp^{N-1}} \abs{D_\tau u}^2 +\abs{D_\tau v}^2 + \frac{1}{\epsilon} \int_{\Sp^{N-1}} \abs{u-v}^2
\end{equation}
Define a linear interpolation on the cylinder $\Sp^{N-1} \times [0,\epsilon]$ by
\[
	\tilde{w}(y,t)=\left(1- \frac{t}{\epsilon}\right) u(y) + \left(\frac{t}{\epsilon}\right) v(y) \text{ for } y \in \Sp^{N-1}, t \in [0,\epsilon]
\]
and then making use of polar coordinates $x=ry, r\in [1-\epsilon, 1], y \in \Sp^{N-1}$ the annulus $A_{1, 1-\epsilon}= B_1\setminus B_{1-\epsilon}$ is close to the cylinder i.e.
\[
	w(ry)=\tilde{w}(y, 1-r) \text{ for } r \in [1-\epsilon, 1], y \in \Sp^{N-1} \text{ i.e. } ry \in A_{1,1-\epsilon}.
\]
One checks that $w$ defined in that way satisfies \eqref{eq:A2.1}.

Our extension of this result to "boundary" functions in a fractional Sobolev space is:
\begin{lemma}\label{lem:A2.1}
Let $\frac{1}{2}<s<1$ and $\epsilon >0 $ be given then there exists $R_\epsilon>0$ with the property: for any $R_\epsilon\le R < 1$ there is $C=C(\epsilon, R)$ s.t. given $u, v \in W^{s,2}(\Sp^{N-1})$ one can find $w \in W^{1,2}(A_{1,R})$ on the annulus $A_{1,R}=B_1\setminus B_R$ with $w(x)=u(x)$ and $w(Rx)=v(x)$ for $x \in \Sp^{N-1}$ that satisfies 
\begin{equation}\label{eq:A2.2}
	\int_{A_{1,R}} \abs{Dw}^2 \le \epsilon  \left( \llfloor u \rrfloor^2_{s,\Sp^{N-1}} + \llfloor v \rrfloor^2_{s,\Sp^{N-1}} \right) + C \norm{u-v}^2_{\Sp^{N-1}}.
\end{equation}
\end{lemma}

Our proof uses heavily the theory of homogenous harmonic polynomials. This is not a surprise since they build, together with their Kelvin transforms, a natural basis for solving the Dirichlet problem on an annulus. As a reference for classical results one may consult \cite[chapter 5]{Axler}.\\
We will use the same notation introduced there: 
\begin{itemize}
\item $\p_m(\R^N)$ denotes the complex vector space of all homogeneous polynomials on $\R^N$ of degree $m$;
\item $\h_m(\R^N)\subset \p_m(\R^N)$ the subspace of all harmonic homogeneous polynomials of degree $m$.
\end{itemize}
We want to emphazise that we do not equip $\p_m(\R^N)$ and $\h_m(\R^N)$ with specific norms or inner products.\\ 
Furthermore we need the Kelvin transform for a map $u: \Omega \subset \R^N \setminus \{0\}$
\begin{equation}\label{eq:A2.3}
	K[u]= \abs{x}^{2-N} u\left(\frac{x}{\abs{x}^2}\right) \text{ for } x \in \Omega^*=\left\{ x \, : \, \frac{x}{\abs{x}^2} \in \Omega\right\}.
\end{equation}
A key feature of the Kelvin transform is $\Delta( K[u] )= K[\abs{x}^4\Delta u]$, compare \cite[Proposition 4.6]{Axler}. Hence the Kelvin transform is a homeomorphism on harmonic functions, \cite[Theorem 4.7]{Axler}. Furthermore for $p \in \p_m(\R^N)$ we have the simple formula $K[p](x)= \frac{p(x)}{\abs{x}^{N+2m -2}}$. $K[p]$ is therefore homogeneous of degree $2-N-m$.\\

The proof of lemma \ref{lem:A2.1} splits into two parts.\\
In the first we characterise $W^{s,2}(\Sp^{N-1})$ using a Fourier decomposition into harmonic homogeneous polynomials. In the second we use this characterisation to estimate the solution of the Dirichlet problem on the annulus $A_{1,R} = B_1\setminus B_R$.

Recall the classical theorem, e.g. \cite[Theorem 5.7]{Axler}
\begin{theorem}\label{theo:A2.2}
Every $p \in \p_m(\R^N)$ can be uniquely written in the form 
\[
	p=p_m + \abs{x}^2p_{m-2} + \dotsb + \abs{x}^{2k}p_{m-2k},
\]
where $k= \lfloor \frac{m}{2} \rfloor $ and each $p_n \in \h_n(\R^N)$.
\end{theorem}

\begin{lemma}\label{lem:A2.3}
If $p\in \h_m(\R^N)$ and $q$ is a polynomial with strictly less degree then
\begin{equation}\label{eq:A2.4}
	\int_{\Sp^{N-1}} pq = 0 = \int_{\Sp^{N-1}} D_\tau p \cdot D_\tau q
\end{equation}
($ D_\tau p\cdot D_\tau q = Dp \cdot Dq - \frac{\partial p}{\partial r}\frac{\partial q}{\partial r}=\sum_{i=1}^N\frac{\partial p}{\partial x_i} \frac{\partial q}{\partial x_i} - \frac{\partial p}{\partial r}\frac{\partial q}{\partial r}$)\\
If $p,q \in \h_m(\R^N)$ then
\begin{align}\label{eq:A2.5}
	m ( N-2+2m) \int_{\Sp^{N-1}} pq &= \int_{\Sp^{N-1}} Dp\cdot Dq \\\nonumber
	&= \int_{\Sp^{N-1}} D_\tau p\cdot D_\tau q + m^2 \int_{\Sp^{N-1}} \frac{\partial p}{\partial r}\frac{\partial q}{\partial r}.
\end{align}
\end{lemma}

\begin{proof}
By linearity and the decomposition of theorem \ref{theo:A2.2} we may assume that $q\in \h_n(\R^N)$ for some $n<m$. Recall that if $u\in C^1$ is homogenous of degree $\lambda$, it satisfies the Euler formula $\abs{x} \frac{\partial u}{ \partial r}(x)= Du(x)\cdot x = \lambda u(x)$. Furthermore observe that $\frac{\partial p}{\partial x_i} \in \h_{m-1}(\R^N)$ and $\frac{\partial q}{\partial x_i} \in \h_{n-1}(\R^N)$ for any $i=1, \dotsc, N$. Hence
\begin{align*}
	n \int_{\Sp^{N-1}} pq &= \int_{\Sp^{N-1}} p \frac{\partial q}{\partial r} = \int_{\Sp^{N-1}}  \frac{\partial p}{\partial r} q + \int_{B_1} p \Delta q- \Delta p \,q\\
	&=m \int_{\Sp^{N-1}} pq;
\end{align*}
\begin{align*}
	\int_{\Sp^{N-1}} D_\tau p\cdot D_\tau q &= \int_{\Sp^{N-1}} D_\tau p\cdot D_\tau q + n m \,pq = \int_{\Sp^{N-1}} D_\tau p\cdot D_\tau q + \frac{\partial p}{\partial r}\frac{\partial q}{\partial r}\\
	&=\sum_{i=1}^N \int_{\Sp^{N-1}} \frac{\partial p}{\partial x_i} \frac{\partial q}{\partial x_i} =0;
\end{align*}
where we applied the (just obtained) orthogonality of $\h_m(\R^N)$ to $\h_n(\R^N)$ for $m\neq n$.\\
To show \eqref{eq:A2.5} observe that $pq$ is homogenous of degree $2m$ hence
\begin{align*}
	& m(N-2+2m) \int_{\Sp^{N-1}} pq = \frac{1}{2}(N-2+2m) \int_{\Sp^{N-1}} \frac{\partial (pq)}{\partial r}\\
	&= (N-2+2m) \int_{B_1} Dp\cdot Dq = (N-2+2m) \int_0^1 \int_{\Sp^{N-1}} (Dp \cdot Dq)(rx) r^{N-1} dr \\
	& = (N-2+2m) \int_{0}^1 r^{2m-2 + N-1} dr \int_{\Sp^{N-1}} Dp \cdot Dq =\int_{\Sp^{N-1}} Dp \cdot Dq \\
	&= \int_{\Sp^{N-1}} D_\tau p\cdot D_\tau q + \frac{\partial p}{\partial r}\frac{\partial q}{\partial r} = \int_{\Sp^{N-1}} D_\tau p\cdot D_\tau q + m^2 \int_{\Sp^{N-1}} pq.
\end{align*}
\end{proof}

On the base of some Hilbert space theory we recover the following classical result and a small extension, compare e.g. \cite[Theorem 5.12]{Axler}:
\begin{theorem}\label{theo:A2.4}
\begin{align}\label{eq:A2.6}
	L^2(\Sp^{N-1}) = \bigoplus_{m=0}^\infty \h_m(\R^N) \\ \nonumber
	W^{1,2}(\Sp^{N-1}) = \bigoplus_{m=0}^\infty \h_m(\R^N)
\end{align}
\end{theorem}
We are here a bit imprecise in the chosen notation. As a direct sum of vector space both direct sums are the same, but we consider them with different topologies. Furthermore to be precise the equality should be understood restricting each element of the righthand side to the sphere, $\Sp^{N-1}$. In the first case we equip each $\h_m(\R^N)$, with the $L^2$ inner product on the sphere, $\langle p, q \rangle = \int_{\Sp^{N-1}} pq$. $\h_m(\R^N)$ with this topology is a Hilbert subspace of $L^2(\Sp^{N-1})$. In the second equality we equip $\h_m(\Sp^{N-1}$ with the inner product of $W^{1,2}(\Sp^{N-1})$,  $\langle p, q \rangle_1 = \int_{\Sp^{N-1}} pq + \int_{\Sp^{N-1}} D_\tau p \cdot D_\tau q$. With this topology $\h_m(\R^N)$ is a Hilbert subspace of $W^{1,2}(\Sp^{N-1})$.

\begin{proof}
The finite dimensional linear subspaces $\h_m(\R^N), \h_n(\R^N)$ are orthogonal with respect to both inner products $\langle \cdot, \cdot \rangle$, $\langle \cdot, \cdot \rangle_1$ for $m \neq n$. This is a consequence of \eqref{eq:A2.4}. \\
Finally the restriction of polynomials to the sphere are dense in $L^2(\Sp^{N-1})\supset W^{1,2}(\Sp^{N-1})$ due to the Stone-Weierstrass theorem. This proves the theorem since the right hand side is dense in the left. 
\end{proof}

Combining \eqref{eq:A2.5} together with theorem \ref{theo:A2.4} shows that every $u \in L^2(\Sp^{N-1})$ has a unique decomposition $u= \sum_{m=0}^\infty p_m$ with $p_m \in \h_m(\R^N)$ and 
\begin{equation}\label{eq:A2.7}
	\norm{u}^2 = \sum_{m=0}^\infty \norm{p_m}^2.
\end{equation}
Furthermore $u$ is an element of $W^{1,2}(\Sp^{N-1})$ if and only if
\begin{equation}\label{eq:A2.8}
	\infty> \int_{\Sp^{N-1}} \abs{D_\tau u}^2 = \sum_{m=0}^\infty \int_{\Sp^{N-1}} \abs{D_\tau p_m}^2 = \sum_{m=0}^{\infty} m^2 \left(1 + \frac{N-1}{m}\right) \, \norm{p_m}^2.
\end{equation}

This suggests an extension for defining Sobolev spaces on $\Sp^{N-1}$ with noninteger order.
\begin{definition}\label{def:A2.1}
	For a real $s\ge 0$ 
	\begin{equation}\label{eq:A2.9}
		H^s(\Sp^{N-1})=\left\{ u= \sum_{m=0}^\infty p_m \in L^2(\Sp^{N-1})\colon \sum_{m=0}^\infty m^{2s} \norm{p_m}^2 < \infty \right\}.
	\end{equation}
\end{definition}
Now \eqref{eq:A2.8} reads:
\begin{corollary}\label{cor:A2.5}
\begin{equation}\label{eq:A2.10}
	H^1(\Sp^{N-1})=W^{1,2}(\Sp^{N-1}).
\end{equation}
\end{corollary}
As a consequence of corollary \ref{cor:A2.5} we will see that \eqref{eq:A2.9} provides an equivalent characterisation of the fractional Sobolev spaces:
\begin{lemma}\label{lem:A2.6}
	\begin{equation}\label{eq:A2.11}
		H^s(\Sp^{N-1})=W^{s,2}(\Sp^{N-1})= \left(W^{1,2}(\Sp^{N-1}), L^2(\Sp^{N-1})\right)_{1-s,2}
	\end{equation}
\end{lemma}
We postpone the proof after the next lemma.\\
Identifying interpolation spaces between $W^{1,2}(\Sp^{N-1})$ and $L^2(\Sp^{N-1})$ is now the same question as interpolating between some direct sums of Hilbert spaces with weights. This can be settled easily in a more general setting. Our presentation follows the $L^2$ equivalent of L. Tartar in  \cite[chapter 23]{Tartar}.\\
We consider the situation of a direct sum of Hilbert spaces:
\begin{equation}\label{eq:A2.12}
	H=\bigoplus_{m=0}^\infty H_m
\end{equation}
\begin{lemma}\label{lem:A2.7}
	For a sequence of positive numbers $w=\{w_m\}_{m=0}^\infty$, let 
	\begin{equation}\label{eq:A2.13}
		E(w)=\left\{ a=(a_m)_{m} \in H \colon \sum_{m=0}^\infty w_m \norm{a_m}^2 < \infty \right\} \text{ with } \norm{a}^2_w = \sum_{m=0}^\infty w_m \norm{a_m}^2.
	\end{equation}
	If $w(0)=\{w_m(0)\}_m, w(1)=\{w_m(1)\}$ are two such sequences, then for $0<\theta <1$ one has
	\begin{equation}\label{eq:A2.14}
		\left(E(w(0)),E(w(1))\right)_{\theta,2} = E(w(\theta)) \text{ where } w_m(\theta)=w_m(0)^{1-\theta}w_m(1)^\theta.
	\end{equation}
\end{lemma}
\begin{proof}
We use a variant of the $K$-functional, namely 
\[
	K_2(t,a)= \inf_{a=b+c} \left( \norm{b}^2_{w(0)} + t^2 \norm{c}^2_{w(1)}\right)^{\frac{1}{2}};
\]
hence $K_2(t,a) \le K(t,a) \le \sqrt{2} K_2(t,a)$. Now for $a=\sum_{m} a_m$ we have $K_2(t,a)^2= \inf_{a_m=b_m+c_m} \sum_{m=0}^2 w_m(0) \norm{b_m}^2_{w(0)}+ t^2 w_m(1) \norm{c_m}^2_{w(1)}$. We can calculate $K_2(t,a)$ explicitly, because one is led to choose $b_m= \lambda_m a_m + d_m$ with $d_m \in H_m\cap span(a_m)^\perp$. Then $c_m= (1-\lambda_m) a_m - d_m$ and so $\norm{b_m}^2= \lambda_m^2 \norm{a_m}^2 + \norm{d_m}$, $\norm{c_m}^2= (1-\lambda_m)^2 \norm{a_m}^2 + \norm{d_m}$. Hence $d_m =0$ and one is led to choose for $b_m$ the value $\lambda_m$ that minimises $w_m(0)\lambda_m^2 \norm{a_m}^2+ t^2 w_m(1) (1-\lambda_m)^2 \norm{a_m}^2$. One finds
\[
	\lambda_m = \frac{t^2 w_m(1)}{w_m(0)+ t^2 w_m(1)} \text{ and } 1-\lambda_m = \frac{w_m(0)}{w_m(0)+ t^2 w_m(1)};
\]
so $K_2(t,a)$ is computed explicitly by
\[
	K_2(t,a)^2 = \sum_{m=0}^\infty \norm{a_m}^2 t^2 \frac{w_m(0)w_m(1)}{w_m(0)+ t^2 w_m(1)}.
\]
Finally Lebesgue's monotone convergence theorem provides
\[
	\norm{t^{-\theta}K_2(t,a)}^2_{L^2(\R_+, \frac{dt}{t})}= \sum_{m=0}^\infty \norm{a_m}^2  \int_{0}^\infty t^{2(1-\theta)} \frac{w_m(0)w_m(1)}{w_m(0)+ t^2 w_m(1)} \frac{dt}{t},
\]
making the change of variables $t= \sqrt{\frac{w_m(0)}{w_m(1)}} s$, one finds
\[
	\int_{0}^\infty t^{2(1-\theta)} \frac{w_m(0)w_m(1)}{w_m(0)+ t^2 w_m(1)} \frac{dt}{t}= w_m(0)^{1-\theta} w_m(1)^\theta \int_{0}^\infty \frac{s^{1-2\theta}}{1+s^2} ds.
\]
Since $C= \int_{0}^\infty \frac{s^{1-2\theta}}{1+s^2} ds = \frac{\pi}{2\sin(\pi \theta)}$, this gives
\[
	\norm{t^{-\theta}K_2(t,a)}^2_{L^2(\R_+, \frac{dt}{t})}= C \sum_{m=0}^\infty w_m(\theta) \norm{a_m}^2.
\]
\end{proof}
\begin{proof}[Proof of lemma \ref{lem:A2.6}]
There is unique decomposition $L^2(\Sp^{N-1}) \to \bigoplus_m \h_m(\R^N)$  with $u\mapsto \{ p_m \}_m $ and $u=\sum_{m} p_m$ as seen in theorem \ref{theo:A2.4}.  This map is an isometry between $L^2(\Sp^{N-1})$ and $H^0(\Sp^{N-1})$ and continuously linear between $W^{1,2}(\Sp^{N-1})$ and $H^1(\Sp^{N-1})$. Thus lemma \ref{lem:A2.7} showed that the decomposition is a linear homeomorphism between
\[
	W^{s,2}(\Sp^{N-1})= \left(W^{1,2}(\Sp^{N-1}), L^2(\Sp^{N-1})\right)_{1-s,2}
\]
and 
\[
\left(H^1(\Sp^{N-1}), H^0(\Sp^{N-1})\right)_{1-s,2}= H^s(\Sp^{N-1});	
\]
that is the statement of lemma \ref{lem:A2.6}.
\end{proof}

Now we come to the second part estimating the energy of the solution to the Dirichlet problem on $A_{1,R}=B_1\setminus B_R$ for a fixed $0<R<1$. We start with estimating them for polynomials and after that we will use these estimates to conclude it for general functions.\\

Consider the following Dirichlet problem:\\
Let $p,q \in \h_m(\R^N)$ be given, and let $P: A_{1,R} \to \R$ be the unique solution of
\begin{equation}\label{eq:A2.15}
	\begin{cases}
		\Delta P = 0 , &\text{ on  } A_{1,R}\\
		P(x)=p(x) \text{ and } P(Rx)= q(x) &\text{ for all } x \in \Sp^{N-1}
	\end{cases}	
\end{equation}

\begin{lemma}\label{lem:A2.8}
Let $p,q$ be two given constants, i.e. $p,q \in \h_0(\R^N)$, then there are $\tilde{p},\tilde{q} \in \h_0(\R^N)$ s.t. the solution $P$ of \eqref{eq:A2.15} is
\begin{equation}\label{eq:A2.16}
P(x) = \begin{cases}
	\tilde{p}+ \tilde{q} \ln(r), &\text{ if } N=2\\
	\tilde{p}+ \frac{\tilde{q}}{\abs{x}^{N-2}}, &\text{ if } N>2;
\end{cases} 	
\end{equation}
furthermore we have the estimate
\begin{equation}\label{eq:A2.17}
	\int_{A_{1,R}} \abs{DP}^2 = \begin{cases}
		\frac{2\pi}{-\ln(R)} \abs{p-q}^2, &\text{ if } N=2\\
		\frac{N(N-2)\omega_N}{R^{2-N}-1} \abs{p-q}^2, &\text{ if } N>2;
	\end{cases}
\end{equation}
\end{lemma}
\begin{proof}
It is a standard calculation that $\ln(r)$ for $N=2$ and $\abs{x}^{2-N}$ for $N>2$ are harmonic on $\R^N\setminus \{0\}$, hence the $P(x)=P(r)$ defined by \eqref{eq:A2.16} are harmonic. The boundary conditions in \eqref{eq:A2.15} translate to 
\begin{align*}
	P(1)&=p \text{ hence } \tilde{p}=p \text{ for } N=2 \text{ and } \tilde{p}+\tilde{q}=p \text{ for } N>2\\
	P(R)&=q \text{ hence } \tilde{p}+ \tilde{q} \ln(R)=q \text{ for } N=2 \text{ and } \tilde{p}+\frac{\tilde{q}}{R^{2-N}}=q \text{ for } N>2.
\end{align*}
In the case of $N=2$ one solves for $\tilde{q}= \frac{q-p}{\ln(R)}$, in the case of $N>3$ for $\tilde{q}= \frac{q-p}{R^{2-N}-1}$.\\
Apply Green's formula on the annulus and then insert the boundary conditions in the second to obtain:
\begin{align}\label{eq:A2.18}
	\int_{A_{1,R}} \abs{DP}^2 = \int_{\partial A_{1,R}} P\frac{\partial P}{\partial \nu} &=\int_{\Sp^{N-1}} P(x) \frac{\partial P}{\partial r}(x) - \int_{\partial B_R} P(x) \frac{\partial P}{\partial r}(x) \\\nonumber
	&=\int_{\Sp^{N-1}} p(x)\frac{\partial P}{\partial r}(x) - \int_{\partial B_R} q(R^{-1}x) \frac{\partial P}{\partial r}(x).
\end{align}
For $N=2$, $\frac{\partial P}{\partial r}(r)= \frac{\tilde{q}}{r}$ otherwise $\frac{\partial P}{\partial r}(r)= \frac{(2-N)\tilde{q}}{r^{N-1}}$, hence in two dimensions we found
\[
	\int_{\partial A_{1,R}} P \frac{\partial P}{\partial \nu} = 2\pi \left( p\frac{\partial P}{\partial r}(1)- q\frac{\partial P}{\partial r}(R)R\right)= \frac{2\pi}{-\ln{R}} \abs{p-q}^2;
\]
in higher dimensions
\[
	\int_{\partial A_{1,R}} P \frac{\partial P}{\partial \nu} = N\omega_N\left( p\frac{\partial P}{\partial r}(1)- q\frac{\partial P}{\partial r}(R)R^{N-1}\right)= \frac{N(N-2)\omega_N}{R^{2-N}-1} \abs{p-q}^2.
\]
\end{proof}

For the estimates in the case $m\ge 1$ we introduce two functions:
\begin{align}\label{eq:A2.19}
	f(t)&= \frac{\cosh((N-1)t)-1}{\sinh(t)}\\\nonumber
	\tilde{f}(t)&=\frac{t}{t-t^{-1}}. 
\end{align}

\begin{lemma}\label{lem:A2.9}
	Let $p,q \in \h_m(\R^N)$, $m>0$, be given. Then there are $\tilde{p},\tilde{q} \in \h_m(\R^N)$ s.t. that the solution to \eqref{eq:A2.15} has the form
	\begin{equation}\label{eq:A2.20}
	P(x) = \tilde{p}(x) + K[\tilde{q}](x) = \tilde{p}(x) + \frac{\tilde{q}(x)}{\abs{x}^{N+2m-2}};	
	\end{equation}
	furthermore we can estimate the energy either by
	\begin{equation}\label{eq:A2.21}
		\int_{A_{1,R}} \abs{DP}^2 - \frac{2m + N -2}{ R^{-m-N+2}-R^m }\, \norm{p-q}^2  \le f(\ln(R^{-m})) \, m \left(\norm{p}^2 +\norm{q}^2\right);
	\end{equation}
	or by
	\begin{equation}\label{eq:A2.22}
		\int_{A_{1,R}} \abs{DP}^2 \le 4N \tilde{f}(R^{-m}) \,m\left(\norm{p}^2 +\norm{q}^2\right).
	\end{equation}
\end{lemma}

\begin{proof}
The Kelvin transform maps harmonic polynomials $\tilde{q} \in \h_m(\R^N)$ to harmonic functions on $\R^N\setminus\{0\}$, homogeneous of degree $2-N-m$. Hence $P$ defined by \eqref{eq:A2.20} is harmonic on $\R^N \setminus \{0\}$. The boundary conditions impose $\tilde{p}(x)+ \tilde{q}(x)= p(x)$ and $R^{m}\tilde{p}(x)+ R^{2-N-m} \tilde{q}(x)=q(x)$. Solving this for $\tilde{p}$ and $\tilde{q}$ gives
\[
	\tilde{p}(x)= \frac{R^{2-N-m} p(x)-q(x)}{R^{2-N-m}-R^m} \text{ and } \tilde{q}(x) = \frac{q(x)- R^m p(x)}{R^{2-N-m}-R^m}.
\]
As before we can use the Euler formula for homogenous function $u$ of degree $\lambda$, $r \frac{\partial u(x)}{\partial r}= \lambda u(x)$, to simplify the integrals and inserting $P(x)=p(x), P(Rx)=q(x)$ for all $x \in \Sp^{N-1}$ we obtain
\begin{align*}
	&\int_{\partial A_{1,R}} P \frac{\partial P}{\partial \nu} = \int_{\Sp^{N-1}} p(x) DP(x)\cdot x - R^{N-2} \int_{\Sp^{N-1}} q(x) DP(Rx)\cdot Rx\\
	&= \int_{\Sp^{N-1}} p(x)\left(m \tilde{p}(x)+ (2-N-m) \tilde{q}(x) \right) \\
	&\quad- R^{N-2} \int_{\Sp^{N-1}} q(x)\left(m R^m \tilde{p}(x)+ (2-N-m)R^{2-N-m} \tilde{q}(x) \right)\\
	&=\frac{m}{R^{2-N-m}-R^m} \Bigl(\left[ R^{2-N-m} + \left(1+ \frac{N-2}{m}\right) R^m \right] \norm{p}^2 \\
	&\quad+ \left[ R^{m+N-2} + \left(1+ \frac{N-2}{m}\right) R^{-m} \right] \norm{q}^2 - \left(2+ \frac{N-2}{m}\right) 2 \langle p, q\rangle \Bigr).
\end{align*}
To obtain the first estimate \eqref{eq:A2.21}, subtract $\frac{2m + N -2}{ R^{-m-N+2}-R^m }\, \norm{p-q}^2$ from the integral above and use $-2\langle p, q \rangle = \norm{p-q}^2 - \norm{p}^2 - \norm{q}^2$, which gives
\begin{align*}
- &\left(2+ \frac{N-2}{m}\right) 2 \langle p, q\rangle = \frac{1}{m}(2m + N-2) \norm{p-q}^2\\&
 - \norm{p}^2- (1+\frac{N-2}{m}) \norm{p}^2 - \norm{q}^2- (1+\frac{N-2}{m}) \norm{q}^2.
\end{align*}
We then conclude
\begin{align*}
	&\int_{A_{1,R}} \abs{DP}^2 - \frac{2m + N -2}{ R^{-m-N+2}-R^m }\, \norm{p-q}^2 =\\
	&\frac{m}{R^{-m-N+2}-R^m} \left(\left(R^{-m-N+2} -1\right)+ \left(1+ \frac{N-2}{m}\right)\left(R^m-1\right)\right) \norm{p}^2\\
	&+\frac{m}{R^{-m-N+2}-R^m} \left(\left(R^{m+N-2} -1\right)+ \left(1+ \frac{N-2}{m}\right)\left(R^{-m}-1\right)\right) \norm{q}^2.
\end{align*}
One easily checks that the function $g(y)=(y^a-1)-a(y-1)$ ( defined for $y >0 $ and $a>1$ ) attains its minimum at $y=1$: $g(1)=0$ i.e. $a(y-1)\le y^a -1$. In our case that gives $\left( 1+ \frac{N-2}{m} \right) \left(R^m-1\right) \le \left(R^{m+N-2}-1\right)$ and $\left(1+ \frac{N-2}{m}\right)\left(R^{-m}-1\right)\le \left(R^{-m-N+2}-1\right)$. Hence we can simplify to 
\begin{align*}
	&\int_{A_{1,R}} \abs{DP}^2 - \frac{2m + N -2}{ R^{-m-N+2}-R^m }\, \norm{p-q}^2\\
	& \le m \frac{R^{2-N-m}+R^{m+N-2}-2}{R^{2-N-m}-R^m}\left(\norm{p}^2+\norm{q}^2\right)\le m f(\ln(R^{-m})) \left(\norm{p}^2 + \norm{q}^2\right);
\end{align*}
where we used 
\[
	\frac{R^{2-N-m}+R^{m+N-2}-2}{R^{2-N-m}-R^m} \le \frac{\cosh\left(\left(1+\frac{N-2}{m}\right)\ln(R^{-m})\right)-1}{\sinh(\ln(R^{-m}))} \le f(\ln(R^{-m}))
\]
with $R^{2-N-m}-R^m\ge R^{-m}-R^m$.\\

To deduce \eqref{eq:A2.22}, we estimate quite brutally $-2\langle p, q \rangle \le \norm{p}^2 + \norm{q}^2$. As coefficient in front of $\norm{p}^2$ we get
\begin{align*}
	&\frac{R^{2-N-m}+\left(1+ \frac{N-2}{m}\right)R^m +\left(2+ \frac{N-2}{m}\right)}{R^{2-N-m}-R^m}\\
	&\le \frac{2\left(2+ \frac{N-2}{m}\right)R^{-m}}{R^{-m}-R^{m+N-2}} \le 4N \frac{R^{-m}}{R^{-m}- R^m}.
\end{align*}
In the last inequality we used that $R^{-m}-R^{m+N-2}\ge \frac{1}{2}(R^{-m}- R^m)$. This can be checked as follows: $y\in]0,1] \mapsto (y^{-1}-y^a)- \frac{1}{2}(y^{-1}-y)$ for $a\ge 1$ is nonincreasing and vanishes for $y=1$; the inequality follows inserting $y=R^m$ and $a=1+\frac{N-2}{m}$.\\
The coefficient in front of $\norm{q}^2$ is
\begin{align*}
	&\frac{R^{m+N-2}+\left(1+ \frac{N-2}{m}\right)R^{-m} +\left(2+ \frac{N-2}{m}\right)}{R^{2-N-m}-R^m}\\
	&\le \frac{2\left(2+ \frac{N-2}{m}\right)R^{-m}}{{R^{2-N-m}-R^m}}\le 4N \frac{R^{-m}}{R^{-m}- R^m}.
\end{align*}
This completes the proof.
\end{proof}
To conclude the interpolation theorem we need shortly to analyse the behaviour of the two functions $f$ and $ \tilde{f}$ in \eqref{eq:A2.19}.

\begin{lemma}\label{lem:A2.10}
$f$ is monotone increasing, hence $f(\ln(R^{-m}))$ is increasing in $m$ and decreasing in $R\in]0,1]$. Furthermore we have $\lim_{y \searrow 0} f(y)=0$;\\
$\tilde{f}$ is monotone decreasing, hence for $\delta>0$,  $m^{-2\delta} \tilde{f}(R^{-m})$ is decreasing in $m$ and increasing in $R\in ]0,1]$. Furthermore we have $m^{-2\delta}\tilde{f}(e^{m^{-\delta}})\le \frac{2}{m^\delta} \to 0$ as $m \to \infty$. \\
\end{lemma}

\begin{proof}
$f'$ is given by
\[
	f'(y) 
	=\frac{g(N-1,y)}{\sinh^2(y)};
\]
where we introduced the function 
\[
	g(a,y)=a\sinh(ay)\sinh(y)-\cosh(y)(\cosh(ay)-1) \text{ for } a\ge 1, y >0
\]
$f'$ is strictly positive because firstly we have $g(1,y)= \sinh^2(y)- \cosh^2(y)+\cosh(y)=\cosh(y)-1> 0$ for $y>1$ and secondly
\begin{align*}
	\frac{\partial g}{\partial a}(a,y)&=\sinh(ay)\sinh(y)+ ay \cosh(ay)\sinh(y)- y \cosh(y)\sinh(ay)\\
	&\ge \sinh(ay)\sinh(y) + y (\cosh(ay)\sinh(y)-\cosh(y)\sinh(ay))\\
	&= \sinh(ay)\sinh(y) - y \sinh((a-1)y)\\
	&\ge y (\sinh(ay)-\sinh((a-1)y))\ge 0.
\end{align*}
We used the addition theorem and $\sinh(y)\ge y$ for $y\ge 0$. Therefore we found $g((N-1),y)\ge g(1,y)>0$.
Using L'Hospital's rule we have \[\lim_{y \searrow 0} f(y)= \frac{(N-1)\sinh((N-1)0)}{\cosh(0)}=0.\]\\
$\tilde{f}'(y)=\frac{-2y^{-1}}{(y-y^{-1})^2}< 0$, hence $\tilde{f}$ is monotone decreasing and so is $m \mapsto m^{-2\delta}$.
Finally the conclusions on the behaviour of $f(\ln(R^{-m}))$ and $m^{-2\delta}\tilde{f}(R^{-m})$ follow because for $0<R\le 1$ we have $m \mapsto \ln(R^{-m})$ is monotone increasing and $R\in ]0,1] \mapsto \ln(R^{-m})$ monotone decreasing. The last estimate just follows from $\sinh(y)\ge y$:
\[
	m^{-2\delta}\tilde{f}(e^{m^{-\delta}})= \frac{e^{m^{-\delta}}}{2m^{2\delta}\sinh(m^{-\delta})} \le \frac{2}{m^\delta}.
\]
\end{proof}

Now we are able to prove the interpolation lemma $\ref{lem:A2.1}$:
\begin{proof}[Proof of Lemma \ref{lem:A2.1}]
Recall that $\epsilon>0$ and $1>s>\frac{1}{2}$ are given. Fix $\delta=s-\frac{1}{2} >0 $.\\
Lemma \ref{lem:A2.6} stated that $W^{s,2}(\Sp^{N-1})=H^s(\Sp^{N-1})$ and each element of $H^s(\Sp^{N-1})$, a subset of the vector space $\bigoplus_{m=0}^\infty \h_m(\R^N)$. Therefore it is sufficient to proof \eqref{eq:A2.2} under the additional assumption that for some finite large $M$ we have $u=\sum_{m=0}^M p_m, v=\sum_{m=0}^M q_m$ for $p_m, q_m \in \h_m(\R^N)$. But we have to ensure that the constant in \eqref{eq:A2.2} is independent of $M$.\\
Firstly observe, that if $P_m, P_n$ are the solutions to \eqref{eq:A2.15} corresponding to pairs $p_m, q_m \in \h_m(\R^N)$, $p_n,q_n \in \h_n(\R^N)$ constructed in the preparatory lemmas \ref{lem:A2.8}, \ref{lem:A2.9}. Hence we deduce (as in the proofs to lemma \ref{lem:A2.8}, \ref{lem:A2.9}, using the Euler formula)
\begin{align*}
\int_{\partial A_{1,R}} P_n \frac{\partial P_m}{\partial \nu} =& \int_{\Sp^{N-1}} p_n(x) DP_m(x)\cdot x - R^{N-2} \int_{\Sp^{N-1}} q_n(x) DP_m(Rx)\cdot Rx\\
=& m \langle p_n, \tilde{p}_m \rangle + (2-N-m) \langle p_n, \tilde{q}_m \rangle\\
& - R^{N-2} \left( m R^m \langle q_n, \tilde{p}_m \rangle + (2-N-m)R^{2-N-m}\langle q_n, \tilde{q}_m \rangle\right)\\=&0
\end{align*}
due to the orthogonality \eqref{eq:A2.4}. 
To every $0\le m \le M$ let $P_m$ be the solution of \eqref{eq:A2.15} to the pair $p_m, q_m \in \h_m(\R^N)$ given by the decompositions $u=\sum_{m=0}^M p_m, v=\sum_{m=0}^M q_m$. For $P=\sum_{m=0}^M P_m$ we have just shown that 
\[
	\int_{A_{1,R}} \abs{DP}^2 = \int_{\partial A_{1,R}} P \frac{\partial P}{\partial \nu} = \sum_{m=0}^M \int_{\partial A_{1,R}} P_m \frac{\partial P_m}{\partial \nu}.
\]

Let us define $R_\epsilon = e^{-m_\epsilon^{-1-\delta}}$ for some sufficiently large $m_\epsilon> 1$ with the property that $f(y)<\epsilon$ for $0<y<(m_\epsilon -1)^{-\delta}$ and $4N \frac{2}{m_\epsilon^\delta}<\epsilon$. Such an $m_\epsilon$ exists as a consequence of lemma \ref{lem:A2.10}.\\
Finally for any $R_\epsilon \le R <1$ we may fix $m_R\ge m_\epsilon$ s.t. $e^{-(m_R-1)^{-1 - \delta}} < R\le e^{-m_R^{-1 - \delta}}$. Using the results of lemma \ref{lem:A2.10} we conclude for $m \ge m_R$ 
\begin{align*}
	m^{-2\delta} \tilde{f}(R^{-m}) \le m_R^{-2\delta} \tilde{f}(R^{-m_R})\le m_R^{-2\delta} \tilde{f}((e^{-m_R^{-1-\delta}})^{-m_R}) < \frac{\epsilon}{4N}.
\end{align*}
And for $m <m_R$ i.e. $m\le m_R-1$ we deduce
\begin{align*}
	f(\ln(R^{-m})) &\le f(\ln(R^{-(m_R-1)})) \\
	&\le f(-(m_R-1)\ln (e^{-(m_R-1)^{-1 - \delta}}))= f((m_R-1)^{-\delta})< \epsilon.
\end{align*}\\
Finally we fix the constant $C=C(\epsilon, R)$ to be the maximum of the constants of lemma \ref{lem:A2.8} i.e. $\frac{2\pi}{\ln(R)}$ for $N=2$, $\frac{N(N-2)\omega_N}{R^{2-N}-1}$ for $N>2$ and the one of \eqref{eq:A2.21} i.e. $\frac{2m + N -2}{ R^{-m-N+2}-R^m }$ for $m \le m_R$.\\
We have shown that 
\[
	\int_{\partial A_{1,R}} P_m \frac{\partial P_m}{\partial \nu} \le \epsilon \, m^{2s} \left( \norm{p_m}^2 + \norm{q_m}^2\right) \text{ for } m\ge m_R;
\]
and 
\[
	\int_{\partial A_{1,R}} P_m \frac{\partial P_m}{\partial \nu} \le \epsilon m \left( \norm{p_m}^2 + \norm{q_m}^2\right) + C \norm{p_m-q_m}^2 \text{ for } m< m_R.
\]
This proves a first version of the interpolation since we found
\begin{align*}
	\int_{A_{1,R}} \abs{DP}^2 &= \sum_{m=0}^{m_R-1} \int_{\partial A_{1,R}} P_m \frac{\partial P_m}{\partial \nu} + \sum_{m=m_R}^M \int_{\partial A_{1,R}} P_m \frac{\partial P_m}{\partial \nu}\\
	&\le \epsilon \sum_{m=0}^M m^{2s} \left( \norm{p_m}^2 + \norm{q_m}^2\right)  + C \sum_{m=0}^{m_R-1} \norm{p_m-q_m}^2\\
	&\le  \epsilon \sum_{m=0}^\infty m^{2s} \left( \norm{p_m}^2 + \norm{q_m}^2\right)  + C \sum_{m=0}^{\infty} \norm{p_m-q_m}^2;
\end{align*}
the right hand side is independent of $M$, so that we can pass to the limit as $M \to \infty$. Although $\sum_{m=1}^\infty m^{2s} \norm{p_m}^2$ does not contain the $0$th. order lemma, \ref{lem:A2.6} provides only equivalence for complete norms. Choosing $\epsilon>0$ ( a priory smaller, if necessary, to absorb the constants) we got, for any admissible $W^{s,2}$-norm:
\begin{equation*}
	\int_{A_{1,R}} \abs{Dw}^2 \le \epsilon \left( \norm{u}^2_{W^{s,2}(\Sp^{N-1})}+\norm{v}^2_{W^{s,2}(\Sp^{N-1})} \right) + C \norm{u-v}^2_{\Sp^{N-1}}.
\end{equation*}
To pass actually to \eqref{eq:A2.2} we can use a small oberservation and the Poincar\'e inequality \eqref{eq:A1.6}. Let $u,v \in W^{s,2}(\Sp^{N-1})$ be given, apply the so far obtained interpolation to 
$\tilde{u}= u - \frac{1}{2} (\fint_{\Sp^{N-1}} u + \fint_{\Sp^{N-1}} v)$ and $\tilde{v}= v - \frac{1}{2} (\fint_{\Sp^{N-1}} u + \fint_{\Sp^{N-1}} v)$ providing $\tilde{w} \in W^{1,2}(A_{1,R})$. $\tilde{w}= w + \frac{1}{2} (\fint_{\Sp^{N-1}} u + \fint_{\Sp^{N-1}} v)$ has the desired properties because 
\begin{align*}
	\norm{\tilde{u}}^{2}_{W^{s,2}(\Sp^{N-1})} &= \norm{\tilde{u}}^2_{L^2(\Sp^{N-1})} + \llfloor \tilde{u} \rrfloor^2_{s,\Sp^{N-1}}\\
	&=\norm{\tilde{u}}^2_{L^2(\Sp^{N-1})} + \llfloor u \rrfloor^2_{s,\Sp^{N-1}}
\end{align*}
and by the Poincar\'e inequality \eqref{eq:A1.6} and $2\tilde{u}=(u -\fint_{\Sp^{N-1}}u) + (v -\fint_{\Sp^{N-1}} v) + (u-v)$
\begin{align*}
	2\norm{\tilde{u}}_{L^2(\Sp^{N-1})} 
	&\le  C \left(\llfloor u \rrfloor_{s,\Sp^{N-1}} + \llfloor v \rrfloor_{s,\Sp^{N-1}} \right)+ \norm{u-v}_{L^2(\Sp^{N-1})}.
\end{align*}
We argue similarly for $\tilde{v}$. In conclusion we obtained
\begin{align*}
	\int_{A_{1,R}} \abs{Dw}^2 = \int_{A_{1,R}} \abs{D\tilde{w}}^2 &\le \epsilon \left( \norm{\tilde{u}}^2_{W^{s,2}(\Sp^{N-1})}+\norm{\tilde{v}}^2_{W^{s,2}(\Sp^{N-1})} \right) + C \norm{\tilde{u}-\tilde{v}}^2_{\Sp^{N-1}}\\
	&\le C\epsilon  \left( \llfloor u \rrfloor^2_{s,\Sp^{N-1}} + \llfloor v \rrfloor^2_{s,\Sp^{N-1}} \right) + C \norm{u-v}^2_{\Sp^{N-1}}.
\end{align*}
\end{proof}



\section{$Q$-valued functions} 
\label{sec:_q_valued_functions}

\subsection{Fractional Sobolev spaces for $Q$-valued functions} 
\label{sub:fractional_sobolev_spcaes_for_q_valued_functions}
As before we restrict ourself to $0< s \le 1$. Since $\A_Q(\R^n)$ fails to be a linear space, $L^2(\Omega, \A_Q(\R^n))$ is not a Banach space. Hence we are not in a setting for classical interpolation methods. Nonetheless there are two ways to define $W^{s,2}(\Omega,\A_Q(\R^n))$ in a natural way:
\begin{itemize}
	\item[(a)] using Almgren's bilipschitz embedding $\boldsymbol{\xi}: \A_Q(\R^n) \to \R^m$, theorem \ref{theo_I:1.101},
	\[
		W^{s,2}(\Omega,\A_Q(\R^n))=\{ u \in L^2(\Omega,\A_Q(\R^n)) \colon \boldsymbol{\xi} \circ u \in W^{s,2}(\Omega, \R^{m}) \};
	\]
	\item[(b)]  using the Gagliardo norm
	\[
		W^{s,2}(\Omega)=\{ u \in L^2(\Omega, \A_Q(\R^n)) \colon \llfloor u \rrfloor^2_{s,\Omega}=\int_{\Omega \times \Omega} \frac{\G(u(x),u(y))^2}{\abs{x-y}^{N+2s}} \, dxdy < \infty \}.
	\]
\end{itemize}
The equivalence of both definitions follows from the bilipschitz property of $\boldsymbol{\xi}$ i.e $c \abs{\boldsymbol{\xi}\circ u (x) - \boldsymbol{\xi}\circ u(y)} \le \G(u(x),u(y)) \le \abs{\boldsymbol{\xi}\circ u (x) - \boldsymbol{\xi}\circ u(y)}$ for some $c=c(n,Q)$. This implies
\begin{equation}\label{eq:A2.101}
	c \llfloor \boldsymbol{\xi} \circ u \rrfloor^2_{s,\Omega} \le \llfloor u \rrfloor^2_{s,\Omega} \le \llfloor \boldsymbol{\xi} \circ u \rrfloor^2_{s,\Omega}.
\end{equation}
We had seen that all definitions of $W^{s,2}(\Omega, \R^m)$ are equivalent in case of a Lipschitz regular domain $\Omega\subset \R^N$. \\
Combining the definition of $W^{s,2}(\Omega,\A_Q(\R^n))$ as suggested in (a) with \eqref{eq:A2.101} we obtain nearly all statements for single valued functions as well for multiple valued functions. For the sake of completeness we state them now for $Q$-valued functions:

\begin{corollary}\label{cor:A2.101}
To any given $-1<a<1$ and $\frac{1}{2} <s \le 1$ there is a constant $C>$ with the property, that if $u \in W^{s,2}(\Sp^{N-1}\cap \{ x_N> a\}, \A_Q(\R^n)), v \in W^{s,2}(\Sp^{N-1}\cap \{x_N<a\}, \A_Q(\R^n))$ with $u\tr{\Sp^{N-1}\cap \{ x_N= a\}}=v\tr{\Sp^{N-1}\cap \{ x_N= a\}}$ then 
\begin{equation}\label{eq:A2.102}
	U(x)=\begin{cases}
		u(x), &\text{ if } x \in \Sp^{N-1}, x_N>a\\
		v(x), &\text{ if } x \in \Sp^{N-1}, x_N<a
	\end{cases}
\end{equation}
defines an element in $W^{s,2}(\Sp^{N-1},\A_Q(\R^n))$ satisfying
\begin{equation}\label{eq:A2.103}
	\llfloor U\rrfloor_{s,\Sp^{N-1}} \le C \left( \llfloor u\rrfloor_{s,\Sp^{N-1}\cap \{ x_N> a\}} + \llfloor v\rrfloor_{s,\Sp^{N-1}\cap \{ x_N< a\}}\right)
\end{equation}
\end{corollary}

\begin{lemma}\label{lem:A2.102}
Let $\frac{1}{2}<s\le 1$ and $\epsilon >0 $ be given then there exists $R_\epsilon>0$ with the property: for any $R_\epsilon\le R < 1$ there is $C=C(\epsilon, R, n, Q)$ s.t. given $u, v \in W^{s,2}(\Sp^{N-1},\A_Q(\R^n))$ one can find $w \in W^{1,2}(A_{1,R}, \A_Q(\R^n))$ on the annulus $A_{1,R}=B_1\setminus B_R$ with $w(x)=u(x)$ and $w(Rx)=v(x)$ for $x \in \Sp^{N-1}$ that satisfies 
\begin{equation}\label{eq:A2.104}
	\int_{A_{1,R}} \abs{Dw}^2 \le \epsilon  \left( \llfloor u \rrfloor^2_{s,\Sp^{N-1}} + \llfloor v \rrfloor^2_{s,\Sp^{N-1}} \right) + C \norm{\G(u,v)}^2_{\Sp^{N-1}}.
\end{equation}
\end{lemma}

\begin{proof}
	For $s=1$ we set $R_\epsilon = 1-\epsilon$. We obtain $\tilde{w} \in W^{1,2}(A_{1,R}, \R^m)$ applying observation \eqref{eq:A2.1} to $\boldsymbol{\xi}\circ u$, $\boldsymbol{\xi}\circ v$ with $\epsilon'= 1-R$, $R_\epsilon < R< 1$. We obtain $\tilde{w} \in W^{1,2}(A_{1,R}, \R^m)$. The retraction $w=\boldsymbol{\rho}\circ \tilde{w} \in W^{1,2}(A_{1,R}, \A_Q(\R^n))$ then has up to a constant the desired properties.\\
	For $\frac{1}{2}<s<1$ we proceed similarly. Firstly apply lemma \ref{lem:A2.1} to  $\boldsymbol{\xi}\circ u$, $\boldsymbol{\xi}\circ v$ that gives $\tilde{w} \in W^{1,2}(A_{1,R}, \R^m)$. As before the retraction $w=\boldsymbol{\rho}\circ \tilde{w} \in W^{1,2}(A_{1,R}, \A_Q(\R^n))$ fulfils up to a constant the desired properties.
\end{proof}

\subsection{Concentration compactness for $Q$-valued functions}\label{sec:concentration_compactness_for_q_valued_functions}
Let $\Omega \subset \R^N$ be given, then there is a concentration compactness lemma for sequences $u(k) \in W^{1,2}(\Omega,\A_Q(\R^n))$ with uniformly bounded energy.

\begin{lemma}\label{lem_I:A1.1}
Given a sequence $u(k) \in W^{1,2}(\Omega, \A_Q(R^n))$ and a sequence of means $T(k) \in \A_Q(\R^n)$ with
\[
	\limsup_{k \to \infty} \int_{\Omega} \abs{Du(k)}^2 \le \infty \text{ and } \int_{\Omega} \G(u(k),T(k))^2 \le C \int_{\Omega} \abs{Du(k)}^2
\]
for a subsequence, not relabelled, we can find:
\begin{itemize}
	\item[(i)] maps $b_l \in W^{1,2}(\Omega, \A_{Q_l}(\R^n))$ for $l=1, \dotsc , J$, $\sum_{l=1}^L Q_l =Q$;\\
	\item[(ii)] a splitting $T(k) = T_1(k) + \dotsm + T_L(k) $ with $T_l(k) \in \A_{Q_l}(\R^n)$ and
	\begin{itemize}
		\item $\limsup_{k} diam(spt(T_l(k)))< \infty$ for all $l=1, \dotsc, L$\\
		\item $\lim_{k \to \infty} \dist(spt(T_l(k)), spt(T_{m}(k)))= \infty$ for $l\neq m$;
	\end{itemize}
	\item[(iii)] a sequence $t_l(k) \in spt(T_l(k))$ such that $\G(u(k), b(k)) \to 0$ in $L^2$ with $b(k)= \sum_{l=1}^L  (b_l \oplus t_l(k))$.
\end{itemize}
Moreover, the following two additional properties hold:
\begin{itemize}
	\item[(a)] if $\Omega' \subset \Omega$ is open and $A_k$ is a sequence of measurable sets with $\abs{A_k} \to 0$, then
	\[
		\liminf_{k \to \infty} \int_{\Omega'\setminus A_k} \abs{Du(k)}^2 - \int_{\Omega'} \abs{Db(k)}^2 \ge 0.
	\]
	\item[(b)] $\liminf_{k \to \infty} \int_{\Omega} \left(\abs{Du(k)}^2 -\abs{Db(k)}^2\right) = 0$ if and only if\\ $\liminf_{k \to \infty} \int_{\Omega} \bigl(\abs{Du(k)} - \abs{Db(k)})^2 = 0$.
\end{itemize}
\end{lemma}

Before we give the proof we recall the definition of the separation $sep(T)$ of a $Q$-point $T=\sum_{i=1}^Q \llbracket t_i \rrbracket \in \A_Q(\R^n)$.
\[
sep(T)=\begin{cases}
0, &\text{ if } T=Q\llbracket t \rrbracket \\
\min_{t_i \neq t_j} \abs{t_i - t_j}, &\text{ otherwise }.
\end{cases}\]

The following results are of essential use in the context of the separation and needed for the proof of the concentration compactness lemma. The first gives a kind of relation between $diam(spt(T))$ and $sep(T)$, see \cite[lemma 3.8]{Lellis}; the second gives a retraction $\boldsymbol{\vartheta}=\boldsymbol{\vartheta}_T$ based on $sep(T)$, see \cite[lemma 3.7]{Lellis}

\begin{lemma}\label{lem_I:A1.2}
To every $\epsilon>0$ there exists $\beta=\beta(\epsilon, Q)>0$ with the property that to any $T \in \A_Q(\R^n)$ there exists $S=S(T) \in \A_Q(\R^n)$ with
\[ spt(S) \subset spt(T), \quad \G(T,S) < \epsilon\, sep(S) \text{ and } \beta \, diam(spt(T)) < sep(S). \]
(For example $\beta=\epsilon^Q \, 3^{4-Q^2}$ works.)
\end{lemma}

\begin{lemma}\label{lem_I:A1.3}
To a given $T\in \A_Q(\R^n$ and $0< 4s < sep(T)$ there exists a $1-$Lipschitz retraction 
\[\boldsymbol{\vartheta}=\boldsymbol{\vartheta}_T: \A_Q(\R^n) \to \overline{B_s(T)}=\{ S \in \A_Q(T) \colon \G(S,T) \le s \} \]
with the property that
\begin{itemize}
\item[(i)] $\boldsymbol{\vartheta}(S)=S$ if $\G(S,T) \le s$;\\
\item[(ii)] $\G(\boldsymbol{\vartheta}(S_1), \boldsymbol{\vartheta}(S_2))<\G(S_1,S_2)$ if $\G(S_1,T)> s$.
\end{itemize}

\end{lemma}

\begin{proof}[Proof of lemma \ref{lem_I:A1.1}]
We distinguish two cases. The second will be handled by induction on the first.\\

\emph{Case 1 and basis of the induction: $\liminf_{k \to \infty} diam(spt(T(k)))< \infty$\\ ( $diam(spt(T(k)))=0$ for $Q=1$):}\\
Passing to an appropriate subsequence, not relabelled $diam(spt(T(k)))< C$ for all $k$. Set $L=1$, and as splitting keep the sequence itself i.e. $T(k) =T_1(k)$. To every $k$ fix a $t_1(k) \in spt (T(k))$.\\
Hence we have
\begin{align*}
	&\limsup_k \int_{\Omega} \abs{ u(k) \oplus  (-t_{1}(k))}^2 = \limsup_k \int_{\Omega} \G(u(k), Q \llbracket t_1(k) \rrbracket )^2\\
	 \le &\limsup_k 2 \int_{\Omega} \G(u(k), T(k))^2 + 2 \abs{\Omega} \G(T(k), Q\llbracket t_1(k) \rrbracket)^2 < \infty.
\end{align*}
Hence passing to an appropriate subsequence there is $b=b_1 \in W^{1,2}(\Omega, \A_Q(\R^n))$ with $u(k)\oplus(-t_1(k)) \to b$ in $L^2$. This proves (i),(ii),(iii), since $\G(u(k)\oplus- t_1(k), b)= \G(u(k), b\oplus t_1(k)) = \G(u(k), b(k))$. Furthermore, the established properties imply that $\boldsymbol{\xi} \circ u(k) \rightharpoonup \boldsymbol{\xi}\circ b(k)$ in $W^{1,2}(\Omega, \R^m)$. The additional property (a) follows, because $\mathbf{1}_{\Omega'\setminus A_k} \to \mathbf{1}_{\Omega'}$ in $L^2(\Omega)$ and so $\mathbf{1}_{\Omega'\setminus A_k} D\boldsymbol{\xi} \circ u(k) \rightharpoonup \mathbf{1}_{\Omega'}D\boldsymbol{\xi}\circ b(k)$. Property (b) holds because $L^2(\Omega)$ is an Hilbert space. Therefore we have, that $f_k = D \boldsymbol{\xi} \circ u(k) \to f= D\boldsymbol{\xi}\circ b(k)$ in $L^2(\Omega)$ if and only if $f_k \rightharpoonup f$ and $\norm{f_k}^2_{L^2(\Omega)} \to \norm{f}^2_{L^2(\Omega)}$; compare $\liminf_{k} \norm{f_k-f}^2 = \liminf_{k} \norm{f_k}^2 + \norm{f}^2 - 2 \langle f_k, f \rangle = \liminf_{k} \norm{f_k}^2 - \norm{f}^2$. \\

\emph{Case 2 and the induction step: $\liminf_k diam(spt(T(k)))= +\infty$}\\
Suppose the lemma holds for $Q'<Q$. To every $T(k)$ pick $S(k) \in \A_Q(\R^n)$ using \ref{lem_I:A1.2} s.t. for $S(k)=\sum_{j=1}^{J(k)} Q_j(k) \llbracket s_j(k) \rrbracket \in \A_Q(\R^n)$ set $\sigma_k = sep(S(k))$, then $\beta(\frac{1}{10},Q)\, diam(spt(T(k)))< \sigma_k$ and $\G(T(k),S(k)) < \frac{\sigma_k}{10}$. Passing to an appropriate subsequence, not relabelled, we may further assume that $J(k)>1$ and$Q_j(k)$ do not depend on $k$. Fix the associated 1-Lipschitz retractions of \ref{lem_I:A1.3} $\boldsymbol{\vartheta}_k: \A_Q(\R^N) \to \overline{B_{\frac{1}{5} s(S(k))}( S(k))}$ i.e. $\mathcal{H}^0\left( spt(\boldsymbol{\vartheta}_k(T))\cap B_{\frac{\sigma_k}{5}(s_j)}\right)=Q_j$ for all $T \in \A_Q(\R^n)$ and $j=1, \dotsc, J$. Hence these retractions $\boldsymbol{\vartheta}_k$ defines new sequences $v_j(k)$ in $W^{1,2}(\Omega, \A_{Q_j}(\R^n))$ and a splitting of $T(k)$:
\begin{align*}
	&\boldsymbol{\vartheta}_k \circ u(k) = v_1(k)+ \dotsb v_J(k) \text{ with } v_j(k) \in B_{\frac{\sigma_k}{5}}(s_j); \\
	T(k)=&\boldsymbol{\vartheta}_k \circ T(k) = T_1(k)+ \dotsb + T_J(k) \text{ with } T_j(k) \in B_{\frac{\sigma_k}{5}}(s_j)
\end{align*}
Each sequence $v_j(k)$, $j=1, \dotsc, J$ satisfies itself the assumptions of the lemma, because $\boldsymbol{\vartheta_k}$ is a retraction and so
\begin{align}
	\sum_{j=1}^J \abs{Dv_j(k)}^2 &= \abs{D\boldsymbol{\vartheta_k} \circ u(k)}^2 \le \abs{Du(k)}^2\label{eq_I:A1.1}\\
	\sum_{j=1}^J\G(v_j(k), T_j(k))^2 &= \G(\boldsymbol{\vartheta}_k\circ u(k), \boldsymbol{\vartheta}_k \circ T(k))^2 \le \G(u(k),T(k))^2.\label{eq_I:A1.2}
\end{align}
Furthermore we record some properties:\\
Defining  $A_k=\{x\,:\, \boldsymbol{\vartheta}_k \circ u(k)(x) \neq u(k)(x) \} = \{ x\, :\, \G(u(k),S(k)) > \frac{\sigma_k}{5} \} \subset \{x\,:\, \G(u(k),T(k)) \ge \frac{\sigma_k}{10} \}= B_k$ (subsets of $\Omega$) we have
\begin{itemize}
\item[(1.)] $\abs{B_k} \to 0$ as $k \to \infty$, because
\begin{align*}
\abs{B_k} &\le \left( \frac{10}{\sigma_k)}\right)^{2^*} \int_{B_k} \G(u(k),T(k))^{2^*}\\
&\le \left( \frac{10}{\sigma_k}\right)^{2^*} C \left(\int_{\Omega} \abs{Du(k)}^2\right)^{\frac{2^*}{2}} \to 0;
\end{align*}
\item[(2.)] $\G(u(k), \boldsymbol{\vartheta}_k \circ u(k)) \to 0$ in $L^2$ as $k \to \infty$, since
\begin{align*}
	&\int_\Omega \G(u(k),\boldsymbol{\vartheta}_k\circ u(k))^2 = \int_{A_k} \G(u(k), \boldsymbol{\vartheta}_k \circ u(k))^2\\
	&\le 2 \int_{B_k} \G(v_k, T(k))^2 + \G(\boldsymbol{\vartheta}_k\circ u(k), \boldsymbol{\vartheta}_k \circ T(k))^2\\
	&\le 4 \left(\frac{10}{\sigma_k}\right)^{2^*-2} \int_{B_k} \G(u(k), T(k))^{2^*}\\
	&\le \frac{C}{\sigma_k^{2^*-2}} \left(\int_{\Omega} \abs{Du(k)}^2\right)^{\frac{2^*}{2}} \to 0;
\end{align*}
\item[(3.)] $\dist(spt(T_i), spt(T_j)) \ge \sigma_k - 2 \G(S(k),T(k)) \ge \frac{4}{5} \sigma_k \to +\infty$ for any $i \neq j$ as $k \to \infty$;
\item[(4.)] $\abs{\abs{Du(k)} - \abs{D\boldsymbol{\vartheta}_k\circ u(k)}} \to 0$ in $L^2$ as $k \to \infty$, because $\abs{B_k} \to 0$, $\abs{D\boldsymbol{\vartheta}_k\circ u(k)}\le \abs{Du(k)}$, $D \boldsymbol{\vartheta}_k \circ u(k) = D u(k)$ on $\Omega\setminus B_k$ and
\begin{align*}
	&\int_{\Omega} \left( \abs{Du(k)} - \abs{D\boldsymbol{\vartheta}_k \circ u(k)}\right)^2 \le \int_{\Omega} \abs{Du(k)}^2- \abs{D\boldsymbol{\vartheta}_k \circ u(k)}^2\\
	 &= \int_{B_k} \abs{Du(k)}^2- \abs{D\boldsymbol{\vartheta}_k \circ u(k)}^2 \le \int_{B_k} \abs{Du(k)}^2\to 0.
\end{align*}
\end{itemize}
Due to the induction hypothesis the lemma holds for each sequence $v_j(k)$ i.e. we can find $b_{j,l} \in W^{1,2}(\Omega, \A_{Q_{j,l}}(\R^n))$, with $\sum_{l=1}^{L_j} Q_{j,l}=Q_j$, a splitting $T_j(k)= T_{j,1}(k) + \dotsb + T_{j,L_j}(k)$ together with sequences $t_{j,l}(k) \in spt(T_{j,l}(k))$ satisfying the conditions (i), (ii), (iii). Furthermore the additional properties (a),(b) hold. Set $L= \sum_{j=1}^J L_j$, $K_j=\sum_{i=1}^{j-1} L_i$ and relabel $b_{K_j+l}=b_{j,l}$, $T_{K_j+l}(k)=T_{j,l}(k)$, $t_{K_j+l}(k)=t_{j,l}(k)$ and $Q_{K_j+l}=Q_{j,l}$ for $j \in \{1, \dotsc, J\}$ and $l \in \{1, \dotsc, L_j\}$. The induction hypothesis on the lemma states that the obtained sequences $b_l$, $T_l(k)$, $t_l(k)$ for $l=1, \dotsc, L$ satisfy
\begin{itemize}
	\item[(i)] $b_l \in W^{1,2}(\Omega, \A_{Q_l}(\R^n))$ for $l=1, \dotsc , L$ and $\sum_{l=1}^{L} Q_l =Q$;\\
	\item[(ii)] $T(k) = T_1(k) + \dotsb + T_{L}(k)$, $t_l(k) \in spt( T_l(k))$ and
	\begin{itemize}
		\item $\limsup_{k} diam(spt(T_l(k)))< \infty$ for all $l=1, \dotsc , L$\\
		\item $\lim_{k \to \infty} \dist(spt(T_l(k)), spt(T_{m}))= \infty$ for $l\neq m$ for any $K_j < l < m \le K_{j+1}$, $j=1, \dotsc, J$
	\end{itemize}
	\item[(iii)] $\G(v_j(k), b_j(k)) \to 0$ in $L^2$ with $b_j(k)= \sum_{l=K_j+1}^{K_{j+1}} (b_l \oplus t_l(k)) $ for each $j$.
\end{itemize}
Moreover, the following two additional properties hold for each $j$:
\begin{itemize}
	\item[(a)] if $\Omega' \subset \Omega$ is open and $A_k$ is a sequence of measurable sets with $\abs{A_k} \to 0$, then
	\[
		\liminf_{k \to \infty} \int_{\Omega'\setminus A_k} \abs{Dv_j(k)}^2 - \int_{\Omega'} \abs{Db_j(k)} \ge 0.
	\]
	\item[(b)] $\liminf_{k \to \infty} \int_{\Omega} \left(\abs{Dv_j(k)}^2 -\abs{Db_j(k)}^2\right) = 0$ if and only if\\ ${\liminf_{k \to \infty} \int_{\Omega} \bigl(\abs{Dv_j(k)} - \abs{Db_j(k)})^2 = 0}$.
\end{itemize}
Due to properties (1) to (4) we may sum in $j$ and replace $\sum_{j=1}^J v_j(k)$ by $u(k)$. This completes the proof. 
\end{proof}

\subsection{Dirichlet minimizers on cylinders, Remark \ref{rem_B:3.1}} 
\label{sub:dirichlet_minimizers_on_cylinders_remark_rem_B:3.1}
As announced in Remark \ref{rem_B:3.1} we present the proof given in \cite{Lellis} to the following observation.
\begin{lemma}\label{lem:A6.1}
$u(x)\in W^{1,2}(\Omega, \A_Q(\R^n))$ and $U(x,t)=u(x)$ is Dirichlet minimizing on $\Omega \times \R$ then $u$ itself is minimizing in $\Omega$
\end{lemma}

\begin{proof}
Given an arbitrary competitor $v(x)\in W^{1,2}(\Omega, \A_Q(\R^n))$ to $u$ i.e. $u\tr{\partial \Omega}=v\tr{\partial \Omega}$ on $\partial \Omega$. We fix an interpolation $w\in W^{1,2}(\Omega \times [0,1], \A_Q(\R^n))$ satisfying $w(x,0)=u(x)$, $w(x,1)=v(x)$ for all $x\in \Omega$ and $w(x,t)=u\tr{\partial \Omega}(x)=v\tr{\partial \Omega}(x)$ on $\partial \Omega \times [0,1]$.
\[
	V(x,t)=\begin{cases}
		w(x,L+1 - t) &\text{ if } L \le t \le L+1\\
		v(x) &\text{ if } -L \le t \le L\\
		w(x, L+1+t) &\text{ if } -L-1 \le t \le -L.
	\end{cases}
\]
defines an admissible competitor to $U$. Hence the minimality of $U$ ensures
\begin{align*}
	2(L+1) \int_{\Omega} \abs{Du}^2 &= \int_{\Omega \times [-L-1, L+1]} \abs{DU}^2\\
	 &\le \int_{\Omega \times [-L-1, L+1]} \abs{DV}^2 = 2L \int_{\Omega} \abs{Dv}^2 + 2 \int_{\Omega \times [0,1]} \abs{Dw}^2.
\end{align*}
This is equivalent to
\[
	\int_{\Omega} \abs{Du}^2 \le \left(1- \frac{1}{L+1}\right) \int_{\Omega} \abs{Dv}^2 + \frac{1}{L+1} \int_{\Omega \times [0,1]} \abs{Dw}^2.
\]
for all $L \ge 0$, proving the minimality of $u$.
\end{proof}

\subsection{$W^{s,p}$-selection for $s>\frac{1}{p}$} 
\label{sub:selection}
The proof of this lemma is due to Camillo De Lellis, but has not been published so far. 

\begin{lemma}\label{lem:A2.401}
	Let $s>\frac{1}{p}$, $Q\in \N$ be given, then for $u \in W^{s,p}([0,1], \A_Q(\R^n))$ we can find $v=(v_1, \dots, v_Q):[0,1]\to (\R^n)^Q$ with the property that 
	\begin{itemize}
		\item[(i)] \[
			[v(t)]=\sum_{i=1}^Q \llbracket v_i(t) \rrbracket = u(t) \text{ for all } t \in [0,1];
		\] 
		\item[(ii)] $v \in W^{s',p}([0,1],(\R^n)^Q)$ for any $s'<s$ i.e. there is a positive constant $C$ depending on $Q$ and $p,s,s'$ s.t.
		\[
			\int_{[0,1]\times [0,1]} \frac{ \abs{v(x)-v(y)}^p}{\abs{x-y}^{1+ps'}} dxdy \le C \int_{[0,1]\times [0,1]} \frac{\G(u(x),u(y))^p}{\abs{x-y}^{1+ps}} dxdy
		\]
	\end{itemize}
\end{lemma}

\begin{proof}
The lemma is a consequence of the results on regular selections of multivalued functions, \cite[theorem 1.1]{Lellis select}, and the following estimate
\begin{equation}\label{eq:A2.401}
	\int_{0\le x\le y \le 1} \frac{\max_{\sigma, \tau \in [x,y]}\abs{f(\sigma)-f(\tau)}^p}{\abs{x-y}^{1+ps'}}\;dxdy \le C \int_{0\le \sigma \le \tau\le 1} \frac{\abs{f(\sigma) - f(\tau)}^p}{\abs{\sigma-\tau}^{1+ps}} \;d\sigma d\tau
\end{equation}
for a constant $C$ depending only on $p, s'<s$. \\
We start with proving \eqref{eq:A2.401}. $W^{s,p}([0,1])\subset C^{0, s-\frac{1}{p}}([0,1])$ for $ps>1$ i.e. for any $\sigma, \tau \in [0,1]$ 
\begin{equation}\label{eq:A2.402}
\abs{f(\sigma)- f(\tau)} \le C \llfloor f \rrfloor_{s,p,[0,1]}
\end{equation}
where we used the abbreviation $\llfloor f \rrfloor_{s,p,[a,b]}^p = \int_{[a,b]\times [a,b]} \frac{\abs{f(x)-f(y)}^p}{\abs{x-y}^{1+ps}}\,dxdy$. This holds by standard theory. Or it may be concluded from lemma \ref{lem_B:2.101}. To do so extend $f$ to $\tilde{f} \in W^{s,p}([-1,3], \R^n)$ by 
\[
	\tilde{f}= \begin{cases}
		f(-t), &\text{ if }-1<t <0\\
		f(t), &\text{ if } 0<t<1\\
		f(1-t), &\text{ if } 1<t<2.
	\end{cases}
\]
The means $\tilde{f}(x,r)= \fint_{x-r}^{x+r} \tilde{f}$ are well-defined for all $x \in [0,1]$ and $r<1$. \eqref{eq:A2.402} for $\tilde{f}$  in the case of $p=2$ agrees with \eqref{eq:142} in lemma \ref{lem_B:2.101} since \eqref{eq:141} is satisfied with $\beta=\frac{1}{2}$; for general $p$ the calculations have to be adapted classically. We conclude: for all $\sigma, \tau \in [0,1]$
\[
	\abs{f(\sigma)-f(\tau)}=\abs{\tilde{f}(\sigma)-\tilde{f}(\tau)} \le C \llfloor \tilde{f} \rrfloor_{s,p,[-1,2]} \le C \llfloor f \rrfloor_{s,p,[0,1]}.
\]

For any $f \in W^{s,p}([a,b],\R^n)$ we may applying \eqref{eq:A2.402} to the rescaled function $f_{a,\rho}(t)= f(a+\rho t)$ with $\rho=b-a$:
\begin{align*}
	&\max_{x,y \in [a,b]} \abs{f(x)- f(y)} =\max_{\sigma, \tau \in [0,1]} \abs{f_{a,\rho}(\sigma)-f_{a,\rho}(\tau)} \le C \llfloor f_{a,\rho} \rrfloor_{s,p,[0,1]}\\
	&=C \rho^{s-\frac{1}{p}} \llfloor f \rrfloor_{s,p,[a,b]} = C (b-a)^{s-\frac{1}{p}} \llfloor f \rrfloor_{s,p,[a,b]}.
\end{align*}
Inserting this in the left hand side of \eqref{eq:A2.401} gives
\begin{align*}
		&\int_{0\le x\le y \le 1} \frac{\max_{\sigma, \tau \in [x,y]}\abs{f(\sigma)-f(\tau)}^p}{\abs{x-y}^{1+ps'}}\;dxdy \\
		&\le C \int_{0\le x \le y \le 1} \frac{(y-x)^{ps-1}}{(y-x)^{1+ps}} \int_{x \le \sigma \le \tau \le 1} \frac{\abs{f(\sigma)-f(\tau)}^p}{(\tau-\sigma)^{1+ps}} \; d\tau d\sigma\; dx dy\\
		&\le C \int_{0\le \sigma \le \tau \le 1} \left( \int_0^\sigma \int_{\tau}^1 (y-x)^{p(s-s')-2} dy dx\right) \frac{\abs{f(\sigma)-f(\tau)}^p}{(\tau-\sigma)^{1+ps}} \; d\tau d\sigma\\
		&\le C \int_{0\le \sigma \le \tau \le 1} \frac{\abs{f(\sigma)-f(\tau)}^p}{(\tau-\sigma)^{1+ps}} \; d\tau d\sigma.
\end{align*}
The constant $C$ is determined by
\[
	\int_0^\sigma \int_{\tau}^1 (y-x)^{\delta-2} dy dx\le \int_0^\sigma\int_{\sigma}^1 (y-x)^{\delta-2} dy dx \le \begin{cases}
		\frac{1-2^{1-\delta}}{\delta (\delta-1)}, &\text{ if } \delta=p(s-s')\neq 1 \\
		\ln(2), &\text{ if } \delta=p(s-s')=1
	\end{cases}.
\]
Making use of Almgren's bilipschtiz embedding $\boldsymbol{\xi}$ we deduce that \eqref{eq:A2.401} holds as well for multivalued functions i.e. for any $u \in W^{s,p}([0,1], \A_Q(\R^n))$
\begin{equation}\label{eq:A2.403}
		\int_{0\le x\le y \le 1} \frac{\max_{\sigma, \tau \in [x,y]}\G(u(\sigma),u(\tau))^p}{\abs{x-y}^{1+ps'}}\;dxdy \le C \int_{0\le \sigma \le \tau\le 1} \frac{\G(u(\sigma), u(\tau))^p}{\abs{\sigma-\tau}^{1+ps}} \;d\sigma d\tau.
\end{equation}

We observed $W^{s,p}([0,1], \A_Q(\R^n)) \subset C^{0,s-\frac{1}{p}}([0,1], \A_Q(\R^n))$, so that we may apply the theory of regular selections developed in \cite{Lellis select}. Especially we use the proof of \cite[theorem 1.1]{Lellis select}. For a given $u \in W^{s,p}([0,1],\A_Q(\R^n)$ we can find $v=(v_1, \dotsc, v_Q):[0,1] \to (\R^n)^Q$ continuous with the property that $[v(t)]=\sum_{i=1}^Q \llbracket v_i(t) \rrbracket = u(t)$ on $[0,1]$ and there is a constant $C_Q>0$ s.t. for any $0\le x \le y \le 1$
\[
	\abs{v(x)-v(y)} \le C_Q \max_{\sigma, \tau \in [x,y]} \G(u(\sigma),u(\tau)).
\]
Combining this with \eqref{eq:A2.403} gives the remaining part (ii) of the lemma.
\end{proof}



\section{Construction of bilipschitz maps between $B_{1+}$ and $\Omega_F\cap B_1$} 
\label{sec:construction_of_bilipschitz_maps_between_b_1_and_omega_fcap_b_1_}
Before showing the general situation, $\Omega_F\cap B_1$ with $\Omega_F=\{ (x',x_N)  \in \R^N \colon x_N > F(x') \}$, $F\in C^1(\R^{N-1})$, we consider the similar case of a bilipschitz map between $B_1$ and the upper half ball $B_{1+}=B_1 \cap \{x_N > 0 \} $ that preserves "radial" homogeneity.\\

It is of interest for us to preserve "radial" homogeneity in the context of constructing competitors. We want to make use of the interpolation lemma on annuli, lemma \ref{lem:A2.1}. We cannot use a generic bilipschitz map between $B_1$ and $B_{1+}$, because in general it is not true that if $G: U \to V$ is bilipschitz and $\psi_k: U \to U$ a sequence of diffeomorphisms that satisfy $\psi_k \to id$ then $G \circ \psi_k \circ G^{-1} \to 1$ with $Lip(G \circ \psi_k \circ G^{-1} ) \to 1$ as $k \to \infty$.

\begin{lemma}\label{lem:A4.1}
There is a bilipschitz map $G:\overline{B_1} \to \overline{B_{1+}}$ that preserves "radial" homogeneity in the sense that
\[
	G \circ \frac{1}{R} \circ G^{-1}(y)= \left(1- \frac{1}{R} \right) c + \frac{1}{R} y;
\]
where $c= \frac{e_N}{2}=\left(0, \dotsc, 0, \frac{1}{2}\right)$ and $0<R$. 
\end{lemma}
\begin{proof}
We make the ansatz $G(x)=c + s(\widehat{x}) x$ for a piecewise $C^1$ function $s: \Sp^{N-1} \to \partial B_{1+}$ with bounded derivative, where $\widehat{x}= \frac{x}{\abs{x}}$. The constrains $\abs{c+s(x)x}^2=1$ for $x\in \Sp^{N-1}\cap \{x_N \ge a\}$ and $\langle e_N, c+s(x)x \rangle =0$ for $x \in \Sp^{N-1}\cap \{x_N \le a\}$ for some $-1<a<0$ determine $s$ and $a$ uniquely to $a= -\frac{1}{\sqrt{5}}$ and
\[
	s(x)=s(x_N)=\begin{cases}
		\frac{1}{2}\left( -x_N + \sqrt{x_N^2 +3}\right), &\text{ if }x_N \ge -\frac{1}{\sqrt{5}}\\
		-\frac{1}{2x_N}, &\text{ if }x_N \le -\frac{1}{\sqrt{5}}.
	\end{cases}
\]
The derivative is
\[
	s'(x_N)=\begin{cases}
		-\frac{1}{2}\left( 1 -\frac{x_N}{\sqrt{x_N^2 +3}}\right), &\text{ if }x_N > -\frac{1}{\sqrt{5}}\\
		\frac{1}{2x_N^2}, &\text{ if }x_N < -\frac{1}{\sqrt{5}};
	\end{cases}
\]
So we may check the bounds $\abs{s'}<3$ and $\frac{1}{2} \le s(x_N) \le \frac{\sqrt{5}}{2}$. Furthermore we got $\grad s(x)=\grad_{\Sp^{N-1}} s(x)= s'(x_N) (\mathbf{1}- x\otimes x)e_N$.\\
The inverse is explicitly given by $G^{-1}(y)= \frac{1}{s(\widehat{y-c})} \,(y-c)$. We got that $G$ and $G^{-1}$ are almost everywhere $C^1$ with derivatives
\begin{align*}
	DG(x)&= s(\widehat{x}) \,\mathbf{1} + \widehat{x} \otimes \grad s(\widehat{x})\\
	DG^{-1}(y)&= \frac{1}{s(\widehat{y-c})}\,\mathbf{1} - \widehat{y-c} \otimes\frac{\grad s(\widehat{y-c})}{s^2(\widehat{y-c})}.
\end{align*}
The "radial" homogeneity follows i.e. $G\circ \frac{1}{R} \circ G^{-1}(y)= G(\frac{1}{s(\widehat{y-c})} \,\frac{y-c}{R})= \left(1- \frac{1}{R} \right) c + \frac{1}{R} y$. Therefore $DG\circ \frac{1}{R} \circ G^{-1} = \frac{1}{R} \, \mathbf{1}$ converging to $\mathbf{1}$ as $R \to 1$.
\end{proof}

\begin{lemma}\label{lem:A4.2}
For any $F\in C^1(\R^{N-1})$ that satisfies $F(0)=0, \grad F(0)=0$ and $\norm{\grad{F}}_\infty < \frac{1}{4}$ there exists a $C^1$-diffeomorphism 
\[
	G_F: \overline{B_{1+}} \to \overline{ \Omega_F \cap B_1} 
\]
with bounds $\norm{DG_F - \mathbf{1}}_\infty, \norm{DG^{-1}_F - \mathbf{1}}_\infty < 10 \norm{\grad{F}}_\infty$.\\
Furthermore if $F_k$ is a sequence of admissible maps with $F_k \to F$ in $C^1$ then $G_{F_k} \to G_F$in $C^1$.
\end{lemma}

\begin{proof}
Let $F$ be fixed, then $\psi: (x',x_N) \mapsto (x', x_N+ F(x'))$ is a $C^1$-diffeomorphism between $\R^N_+$ and $\Omega_F$. Its inverse is $\psi^{-1}(x',x_N)=(x', x_N-F(x'))$. We make again an ansatz for $G=G_F$. Set $G(x)=\psi(s(\widehat{x})\,x)$ where $s: \Sp^{N-1} \to \R_+$ satisfies $\psi(s(y)\,y) \in \Omega_F \cap \Sp^{N-1}$ for all $y \in \Sp^{N-1}_+$. The inverse for such a $G$ is $G^{-1}(x)= \frac{1}{s(\widehat{\psi^{-1}(x)})}\; \psi^{-1}(x)$.\\
As a consequence of the implicit function theorem applied to the level set at $1$ of the auxiliary function 
\[
	h(y,s)=\abs{\psi(s\,y)}^2,
\]
$s \in C^1(\Sp^{N-1}_+, \R_+)$ has the desired properties. Note that $s(e_N)=1$ because $h(e_N,1)=1$.\\

\emph{Existence:} to every $y \in \Sp^{N-1}_+$ there exists $s(y)\in \R_+$ s.t. $h(y,s(y))=1$ and $1- \norm{\grad{F}}_\infty \le \frac{1}{s} \le 1+ \norm{\grad{F}}_\infty$, because
\begin{align*}
	h(y,s)&= s^2 \,\abs{y+ \frac{ F(s\, y')}{s}\,e_N}^2\\
	&\le s^2 \left(1+\norm{\grad{F}}_\infty\right)^2 < 1 \text{ if } s < \frac{1}{1+ \norm{\grad{F}}_\infty}\\
	&\ge s^2 \left(1-\norm{\grad{F}}_\infty\right)^2 > 1 \text{ if } s > \frac{1}{1- \norm{\grad{F}}_\infty}.
\end{align*}\\

\emph{$C^1_{loc}$ homeomorphism:} every tuple $(y_0,s_0)$ with $h(y_0,s_0)=1$ has a neighbourhood $U\times I$ in $\Sp^{N-1}_+ \times \R_+$ and a $C^1$ map $s: U \to I$, $C^1$ with $h(y, s(y))=1$ on $U$. This follows from the implicit function theorem, because at $x_0=s_0 \, y_0$ 
\begin{align*}
	\frac{1}{2} s \frac{\partial h}{\partial s}&=1- \langle \psi(x_0), \psi(x_0) -  d\psi(x_0)x_0 \rangle\\
	 &= 1 - \psi_N(x_0)\left( F(x_0')- \langle \grad F(x'_0), x_0' \rangle\right) \ge 1- 2 \norm{\grad{F}}_\infty \ge \frac{1}{2}.
\end{align*}\\

\emph{Uniqueness/ well-definition: } this is a consequence of $\frac{\partial h}{\partial s} >0$ for each such tuple $(y_0,s_0)$, so there cannot be two $s_1<s_2$ with $h(y_0,s_1)=1=h(y_0,s_2)$.\\

\emph{Bounds on $\grad s=\grad_{\Sp^{N-1}}s$:} Fix any generic $\tau \in T_{y}\Sp^{N-1}$ and so $ 0 = \left(D_\tau h+ \frac{\partial h}{\partial s}\, D_\tau s\right)(y,s(y))$. Furthermore writing $x=s(y)y$ we have 
\[\frac{1}{2s} D_\tau h(y,s)=\frac{1}{s} \langle \psi(x), d\psi(x)s\tau \rangle = \tau_N F(x') + \psi_N(x') \langle \grad F(x'), \tau' \rangle,\]
that gives 
\[\abs{\frac{1}{2s} D_\tau h(y,s)} \le \sqrt{2} \norm{\grad{F}}_\infty.\]
We conclude
\[
	\abs{D_\tau s(y)} = s^2\frac{\abs{\frac{1}{2s} D_\tau h}}{\abs{\frac{1}{2}s \frac{\partial h}{\partial s}}} \le 3 s^2 \norm{\grad{F}}_\infty \le 16 \norm{\grad{F}}_\infty.
\]\\

\emph{Bounds on $DG,DG^{-1}$:} One calculates explicitly that
\begin{align*}
	DG(x)&=d\psi(s(\widehat{x})x)\left(s(\widehat{x}) \mathbf{1} + \widehat{x}\otimes \grad s(\widehat{x}) \right)\\
	&= s(\widehat{x})\mathbf{1} + \widehat{x}\otimes \grad s (\widehat{x})  + \left(e_N \otimes \grad F\right) \left(s(\widehat{x}) \mathbf{1} + \widehat{x}\otimes \grad s (\widehat{x}) \right).
\end{align*}
As we have seen $\abs{s(\widehat{x})-1} \le \frac{\norm{\grad{F}}_\infty}{1-\norm{\grad{F}}_\infty}$. Combining all obtained bounds one can conclude $\norm{DG(x)- \mathbf{1}}_\infty \le 10 \norm{\grad{F}}_\infty$. $DG^{-1}$ is given explicitly by
\begin{align*}
	DG^{-1}(x)&=\frac{1}{s(\widehat{\psi^{-1}(x)})} d\psi^{-1}(x) - \widehat{\psi^{-1}(x)} \otimes \frac{\grad s(\widehat{\psi^{-1}(x)})}{s^2(\widehat{\psi^{-1}(x)})}\\
	&= \frac{1}{s(\widehat{\psi^{-1}(x)})}\mathbf{1} -\frac{1}{s(\widehat{\psi^{-1}(x)})} e_N \otimes \grad{F} -  \widehat{\psi^{-1}(x)} \otimes \frac{\grad s(\widehat{\psi^{-1}(x)})}{s^2(\widehat{\psi^{-1}(x)})}.
\end{align*}
Combing as before all obtained bounds especially $\abs{\frac{1}{s(\widehat{\psi^{-1}(x)})}- 1} \le \norm{\grad F}_\infty$ one can get $\norm{DG^{-1}(x)- \mathbf{1}}_\infty \le 6 \norm{\grad{F}}_\infty$.\\

The convergence statement follows as a consequence of the implicit function theorem, because $F_k \to F$ in $C^1$ then implies $s_{F_k} \to s_{F_k}$ in $C^1$.
\end{proof}


\end{appendix}


\begin{thebibliography}{99}
\bibitem{Brothers}
\textit{Some open problems in geometric measure theory and its applications
 suggested by participants of the 1984 AMS summer institute.}, edited by J. E. Brothers, Proc. Sympos. Pure Math. \textbf{44}, Amer. Math. Soc., Providence, RI,  1986

\bibitem{Almgren}
 F.J.~Almgren, \textit{Almgren's big regularity paper. Q-valued functions minimizing Dirichlet's integral and the regularity of area-minimizing rectifiable currents up to codimension 2.}, World Scientific Monograph Series in Mathematics, 1. World Scientific Publishing Co., Inc., River Edge, NJ,  (2000). xvi+955 pp. ISBN: 981-02-4108-9
%
%

\bibitem{Axler}
S.~Axler, P.~Bourdon, W.~Ramey, \textit{ Harmonic function theory. Second edition.}, Graduate Texts in Mathematics, 137. Springer-Verlag, New York,  (2001). xii+259 pp. ISBN: 0-387-95218-7

\bibitem{dePauw}
P.~Bouafia, T.~De Pauw, J.~Goblet, \textit{Existence of p harmonic mutliple valued maps into a separable Hilbert space}, preprint




%
%
%
%
%
 
\bibitem{Giusti}
E.~Giusti, \textit{Direct methods in the calculus of variations }, World Scientific Publishing Co., Inc., River Edge, NJ,  (2003). viii+403 pp. ISBN: 981-238-043-4
 
%
 
\bibitem{Lellis select}
C.~De Lellis, C.R.~Grisanti, P.~Tilli, \textit{Regular selections for multiple-valued functions}, Ann. Mat. Pura Appl. (4)  \textbf{183}  (2004),  no. 1, 79--95.

\bibitem{Lellis09}
C.~De Lellis,M.~Focardi,E.~Spadaro, \textit{Lower semicontinuous functionals for Almgren's multiple valued functions}, Ann. Acad. Sci. Fenn. Math.  \textbf{36}  (2011),  no. 2, 393--410.

\bibitem{Lellis}
C.~De Lellis, E.N.~Spadaro, \textit{Q-valued functions revisited}, Memoirs of the AMS \textbf{211} (2011), no. 991


\bibitem{Lellis Lp}
 C.~De Lellis, E.~Spadaro, \textit{Regularity of area-minimizing currents I: Gradient $L^p$ estimates.}, to appear
 
\bibitem{Lellis note}
 C.~De Lellis, \textit{Almgren's Q-valued functions revisited}, Proceedings of the International Congress of Mathematicians. Volume III, (2010), 1910--1933, Hindustan Book Agency, New Delhi

\bibitem{Goblet06},
J.~Goblet,\textit{A selection theory for multiple-valued functions in the sense of Almgren}, Ann. Acad. Sci. Fenn. Math.  \textbf{31}  (2006),  no. 2, 297--314.

\bibitem{Goblet09}
J.~Goblet, \textit{Lipschitz extension of multiple Banach-valued functions in the sense of Almgren}, Houston J. Math. \textbf{35}  (2009),  no. 1, 223--231.

\bibitem{GobletZhu08}
J.~Goblet, W.~Zhu, \textit{Regularity of Dirichlet nearly minimizing multiple-valued functions}, J. Geom. Anal.  \textbf{18}  (2008),  no. 3, 765--794

%


\bibitem{Mattila83}
P.~Mattila, \textit{Lower semicontinuity, existence and regularity theorems for elliptic variational integrals of multiple valued functions}, Trans. Amer. Math. Soc.  \textbf{280}  (1983),  no. 2, 589--610.


%
%
%
%
%
\bibitem{Spadaro}
E.~Spadaro, \textit{Complex varieties and higher integrability of Dir-minimizing Q-valued functions}, Manuscripta Math. \textbf{132}  (2010),  no. 3-4, 415--429.
 
\bibitem{Tartar}
L.~Tartar, \textit{An introduction to Sobolev spaces and interpolation spaces.} Lecture Notes of the Unione Matematica Italiana, 3. Springer, Berlin; UMI, Bologna,  (2007). xxvi+218 pp. ISBN: 978-3-540-71482-8; 3-540-71482-0

 
\bibitem{Zhu}
W.~Zhu, \textit{Two-dimensional multiple-valued Dirichlet minimizing functions}, Comm. Partial Differential Equations \textbf{33} (2008), 1847 -1861

\bibitem{Zhu06An}
W.~Zhu, \textit{Analysis on Metric Space Q}, arXiv (2006)

\bibitem{Zhu06Fre}
W.~Zhu, \textit{A Theorem on Frequency Function for Multiple-Valued Dirichlet Minimizing Functions}, arXiv (2006)

\bibitem{Zhu06reg}
W.~Zhu, \textit{A regularity theory for multiple-valued Dirichlet minimizing maps}, arXiv (2006)

\bibitem{Zhu06flow}
W.~Zhu, \textit{An Energy Reducing Flow for Multiple-Valued Functions}, arXiv (2006)
\end{thebibliography}
\end{document}